\documentclass[12pt]{article}

\usepackage{amsmath}
\usepackage{amssymb}
\usepackage{graphicx}
\usepackage{color}

\DeclareMathOperator{\id}{id}

\newcommand{\beqn}{\begin{eqnarray*}}
\newcommand{\eeqn}{\end{eqnarray*}}

\newtheorem{theorem}{Theorem}[section]
\newtheorem{lemma}[theorem]{Lemma}
\newtheorem{definition}[theorem]{Definition}
\newtheorem{corollary}[theorem]{Corollary}
\newtheorem{proposition}[theorem]{Proposition}
\newtheorem{remark}[theorem]{Remark}
\newtheorem{remarks}[theorem]{Remarks}
\newtheorem{example}[theorem]{Example}

\numberwithin{equation}{section}

\newcommand{\qed}{\hfill\mbox{\raggedright $\Box$}\medskip}
 
\newcommand{\etal}{\emph{et al.}}

\newcommand{\ignore}[1]{}
\newcommand{\shf}{\mbox{\footnotesize $\frac{1}{2}$}}
\newcommand{\Z}{\mathbb{Z}}
\newcommand{\R}{\mathbb{R}}
\newcommand{\T}{\mathbb{T}}

\newcommand{\U}{\mathbb{U}}

\newcommand{\X}{\mathbb{X}}
\newcommand{\Y}{\mathbb{Y}}

\newcommand{\hd}{{\mathcal H}}
\newcommand{\tl}{{\mathcal T}}

\renewcommand{\AA}{{\mathcal A}} 
\newcommand{\BB}{{\mathcal B}}
\newcommand{\CC}{{\mathcal C}}

\newcommand{\EE}{{\mathcal E}}
\newcommand{\FF}{{\mathcal F}}
\newcommand{\GG}{{\mathcal G}}
\newcommand{\HH}{{\mathcal H}}
\newcommand{\OO}{{\mathcal O}} 
\newcommand{\PP}{{\mathcal P}}

\newcommand{\RR}{{\mathcal R}}
\newcommand{\TT}{{\mathcal T}}

\newcommand{\VV}{{\mathcal V}}

\newcommand{\eps}{\varepsilon}
\newcommand{\bz}{\bowtie^0}
\newcommand{\tbz}{\tilde\bowtie^0}

\newcommand{\simODE}{\stackrel{\mathrm{ODE}}{\sim}}

\newcommand{\sot}{\mbox{\footnotesize $\frac{1}{3}$}}

\newcommand{\Matrix}[1]{\ensuremath{\left[\begin{array}%
{cccccccccccccccccccccccc} #1 \end{array}\right]}}

\newcommand{\compose}{\raisebox{.15ex}{\mbox{{\scriptsize $\circ$}}}}
\newcommand{\Sone}{{\bf S}^1}
\newcommand{\supp}{\mathrm{supp}}

\title{Overdetermined ODEs and\\
Rigid Periodic States in Network Dynamics} 

\author{Ian Stewart \\ Mathematics Institute \\ University of Warwick 
	\\ Coventry CV4 7AL \\United Kingdom \\	\\ \texttt{\small i.n.stewart@warwick.ac.uk}
	\\ \texttt{\small ORCID: 0000-0002-3759-0142}}

\begin{document}

\maketitle

\begin{abstract}
We consider four long-standing Rigidity Conjectures about synchrony and phase patterns
for hyperbolic periodic orbits of admissible ODEs for networks.
Proofs of stronger local versions of these conjectures, published in 2010--12,
are now known to have a gap, but remain valid for a broad class of networks.
Using different methods we prove local versions of the conjectures
under a stronger condition, `strong hyperbolicity', which is
related to a network analogue of the Kupka-Smale Theorem.
Under this condition we also deduce global versions of the conjectures 
and an analogue of the $H/K$ Theorem in equivariant dynamics.
We prove the Rigidity Conjectures for
all 1- and 2-colourings and all 2- and 3-node networks by
proving that strong hyperbolicity is generic in these cases.
\end{abstract}




\section{Introduction}
\label{S:intro}

The phenomenon of synchrony in network dynamics
has been widely studied for decades; see for example
Boccaletti \etal~\cite{BPP01} and Wang~\cite{W02}. 
In the network context, sets of synchronous nodes are often called {\em clusters}. 
Another term is `partial synchrony': see 
Belykh \etal~\cite{BBH00, BBNH03}, Belykh and Hasler~\cite{BH11},
Pecora \etal~\cite{PSHMR13}, Pogromsky~\cite{P08}, Pogromsky \etal~\cite{PSN02}. 
In neurobiology, neurons are synchronised if they
`fire together', a relationship that is significant for neural processing
and the architecture of the brain (Kopell and LeMasson~\cite{KL94}, Singer~\cite{S99}, 
Uhlhaas \etal \cite{UPLMNNS09}).
Manrubia \etal~\cite{MMZ04} discuss synchronisation in neural
networks. Van Vreeswijk and Hansel~\cite{VH01} analyse several models, including
a coupled system of Hodgkin-Huxley neurons 
that can produce spikes and bursts.
Mosekilde \etal~\cite{MMP02} describe a model of synchronisation in
nephrons, structures in the kidneys that 
help regulate blood pressure.

A closely associated phenomenon is the occurrence of  {\em phase patterns}:
specific phase shifts (as fractions of the period) between nodes with otherwise identical time-periodic waveforms. 
Such phenomena occur in models of animal locomotion (Buono and Golubitsky~\cite{BG01},
Golubitsky \etal~\cite{GSBC98,GSCB99}), 
peristalsis (Chambers \etal~\cite{CTB13}, Gjorgjieva \etal~\cite{GBEE13}), 
respiration (Butera \etal~\cite{BRS99a, BRS99b}), 
and binocular rivalry and visual illusions (Curtu~\cite{C10}, Diekman \etal~\cite{DG14,DGMW12,DGW13,SG19}.
There are also applications in the physical sciences, for instance
to robotics (Campos \etal~\cite{CMS10}, Liu \etal~\cite{LCZ09})
and coupled lasers (Glova~\cite{G03}, Zhang \etal~\cite{ZPYLZX19}).

The common occurrence of such patterns suggests that a unified theory,
providing a conceptual framework applicable to arbitrary networks, 
could prove useful. One such framework is the
`coupled cell' formalism of~\cite{GS02,GS22,GST05,SGP03}, which
takes its inspiration from the topological approach to nonlinear dynamics
of Arnold~\cite{A89}, Smale~\cite{S67}, and many others, and the analogous
theory of equivariant dynamics and bifurcation of Golubitsky \etal~\cite{GS85,GSS88}.
In this formalism, a network determines a class of admissible ODEs,
and the primary aim is to relate the dynamics of such equations to the network
architecture. Synchrony and phase patterns are closely associated
with balanced colourings, which are combinatorial features of
the network, and the associated quotient networks, whose admissible ODEs
prescribe the dynamics of synchronous clusters. There are numerous
existence theorems for steady and periodic states with prescribed synchrony
and phase patterns, and some ideas extend to synchronised chaos.
We summarise this formalism in Section~\ref{S:DN}.

However, the theory of synchrony and phase patterns for periodic states
remains incomplete, because several key results are still
conjectural. The main ones are the Rigidity Conjectures, discussed in Section~\ref{S:RC}.
They have been proved for a broad class of networks, 
but it has recently been realised that 
the published proofs make a tacit assumption that fails for some networks:
see Section~\ref{S:PR}.
The aim of this paper is to prove these conjectures without that assumption,
but under an extra technical hypothesis: `strong hyperbolicity' of the periodic orbit.
In some cases it is possible to dispense with this condition; in particular
we prove the Rigidity Conjectures unconditionally for networks with up to 3 nodes and for
all 1- and 2-colour patterns on any network. The general case, without assuming strong hyperbolicity, remains open, but could be dealt with using similar methods if it is possible to prove
suitable network analogues of the Kupka-Smale Theorem, Section~\ref{S:KST}.
The method shows that if counterexamples to the Rigidity Conjectures exist, they
must have extremely degenerate systems of periodic orbits, Section~\ref{S:FinRem}.

For technical
reasons indicated in Section~\ref{S:AODE}, the analysis is carried out within a mild 
generalisation of the standard coupled cell formalism, 
developed in detail in~\cite{GS22}. We
focus on synchrony and phase relations between nodes (formerly called `cells'), defined as follows. 
Consider an admissible ODE for the network (one that respects the network structure):
\begin{equation}
\label{e:netwkODE}
\dot{x}=f(x)\qquad x_c \in P_c
\end{equation}
where $c$ runs through the set $\CC = \{1, \ldots, n\}$ of nodes,  
$x = (x_1, \ldots, x_n)$, $f = (f_1, \ldots, f_n)$, and the node spaces
$P_c$ are finite-dimensional real vector spaces.
An orbit $\X=\{x(t)\}$ of~\eqref{e:netwkODE}
is {\em synchronous} for nodes $c, d$ if the corresponding node states are equal for all time:
\begin{equation}
\label{E:synch_eq}
x_c(t) \equiv x_d(t)
\end{equation}
The orbit $\X$  is {\em phase-related} for nodes $c, d$ if the corresponding node states are 
equal for all times, up to a phase shift $\theta$:
\begin{equation}
\label{E:phase_eq}
x_c(t) \equiv x_d(t+\theta)
\end{equation}
The {\em synchrony pattern} for $x(t)$ is the equivalence relation determined by all
synchronous pairs. It can also be viewed as a partition of the nodes into synchronous
clusters, or as a colouring in which synchronous nodes are given the same colour.
We pass between these three interpretations without further comment.

For technical reasons, explained in Sections~\ref{S:CLS} and \ref{S:LRSGP}, we prefer to use local versions of these relations.
Let $t_0 \in \R$. Then $\X$ is is {\em locally synchronous} for nodes $c, d$ 
at $x(t_0)$ if there exists an open set $J \subseteq \R$ containing $t_0$
such that \eqref{E:synch_eq} holds for all $t \in J$,
and {\em locally phase-related} for nodes $c, d$ at $x(t_0)$ if \eqref{E:phase_eq}
holds for all $t \in J$.
The {\em local synchrony pattern} at $x(t_0)$ is the equivalence relation determined by all
synchronous pairs.

The  {\em local phase pattern} at a point $x(t_0) \in \X$ is slightly more complicated.
Phase relations between pairs of nodes need not be unique, because
the minimal period of $x_c(t)$ at a node $c$ may differ from the minimal 
period of the entire orbit $\X$.
For an oscillating node this can happen for a multirhythm
state~\cite{GS06}; in fact it is possible for all minimal periods at nodes to differ
from the overall minimal period. Moreover, 
 the state of a steady node is fixed by all phase shifts.
The {\em phase pattern} for $\X$ encodes phase relations between
pairs of nodes as sets $\Theta(c,d)$ containing all $\theta$ 
such that~\eqref{E:phase_eq} is valid. The 
set $\Theta = \{ \Theta(c,d)\}$ has a natural groupoid structure~\cite{SP07}.
This definition also has a local version where again we consider only $t \in J$.

\subsection{Rigidity}
\label{S:rigidity}

The notion of rigidity is central to the general theory of synchrony and phase patterns,
because it excludes artificial examples where these patterns arise for non-generic
reasons, such as couplings that are generically non-zero but vanish on the states concerned.
A prerequisite for defining rigidity is that the relevant state (which we
take to be either an equilibrium or a periodic orbit) should be hyperbolic.
For an equilibrium this means that the Jacobian has no zero or purely imaginary
eigenvalues, Guckenheimer and Holmes~\cite[Section 1.4]{GH83}. 
For periodic orbits, the Floquet multipliers should not lie on
the unit circle except for a simple eigenvalue 1 associated with the periodic orbit,
Hassard \etal~\cite[Section 1.4]{HKW81}.
Equivalently, the derivative of a Poincar\'e return map should have no eigenvalues
on the unit circle~\cite[Section 1.5]{GH83}.

In a general dynamical system, hyperbolicity of a given periodic orbit or
equilibrium is a generic property. That is, it is:

{\em Dense}: Every periodic orbit becomes hyperbolic after an arbitrarily small perturbation 
(if necessary).

{\em Open}: After any sufficently small perturbation, a hyperbolic periodic orbit remains hyperbolic.

The density property follows from the Kupka-Smale Theorem, a considerably stronger
statement; see Kupka~\cite{K63}, Smale~\cite{S63}, and Peixoto~\cite{P66}. Openness is obvious
because eigenvalues (of the Jacobian at an equilibrium 
or the derivative of a Poincar\'e map at a periodic orbit) perturb continuously,
Lancaster and Tismenetsky \cite{LT85}.

We describe the situation for periodic orbits; there are simpler analogous statements
for equilibria. Consider an ODE 
\begin{equation}
\label{E:ODE_fx}
\dot x = f(x) \qquad x \in \R^m
\end{equation} 
for a smooth (that is, $C^\infty$) vector field $f$ on a finite-dimensional
Euclidean space $\R^m$.
In the theory of general dynamical systems, any hyperbolic periodic 
state $\X$ of this ODE with period $T$
persists when $f$ is replaced by any sufficiently small 
perturbation $\tilde f$; see Hirsch \etal~\cite{HPS77}. 
That is, locally there is a unique perturbed periodic state
$\tilde x(t)$ near $x(t)$ with period $\tilde T$ near $T$.
Throughout this paper, `small' refers to the $C^1$ norm, Section~\ref{S:C1norm}.

\begin{definition}\em
\label{D:P_rigid}
Suppose that \eqref{E:ODE_fx} is an admissible ODE for a network.
A property $\PP$ of a hyperbolic periodic orbit $\X$ is {\em rigid} if, for any 
admissible perturbation $\tilde f = f+p$ of the vector field $f$, where $p$ is
sufficiently small, the perturbed periodic orbit $\tilde\X$
also has property $\PP$.
\qed\end{definition}

Hyperbolicity ensures that a locally unique perturbed
periodic orbit exists, so this definition makes sense.
We allow property $\PP$ to depend on the period $T$, which
is replaced by $\tilde T$ in the perturbed ODE. So, for example,
`nodes 1 and 2 are out of phase by half a period' might be a rigid property.
Rigidity is an `openness' condition: in a suitable topology, the set of $\X$ with
property $\PP$ is open.

\subsection{Motivation from Equivariant Dynamics}
In this paper we focus on rigid synchrony and phase patterns of
orbits of admissible ODEs for networks. The natural setting
for these patterns, and an important source of motivation for the network theory,
is equivariant dynamics~\cite{GS02,GSS88}. Here the map $f$ in~\eqref{E:ODE_fx} 
is equivariant for the action of a group $\Gamma$ on $\R^k$; that is,
\[
f(\gamma x) = \gamma f(x) \quad \forall \gamma \in \Gamma
\]
Associated with any periodic orbit $\X=\{x(t)\}$ are two subgroups of $\Gamma$:
\beqn
K &=& \{\gamma: \gamma x(t) = x(t)\ \forall t \in \R\} \\
H &=& \{\gamma: \gamma \X = \X\ \forall t \in \R\} 
\eeqn
We call $K$ the group of pointwise symmetries of $\X$, and $H$ the
the group of setwise symmetries. The possible pairs $(H,K)$ are classified
by the $H/K$ Theorem of Buono and Golubitsky~\cite{BG01}. For finite $\Gamma$
the main conditions are
that $K \lhd H$ and the quotient group  $H/K \cong \Z_r$ is cyclic.
There are other technical conditions if the state space has low dimension.
It can then be shown that $\X$ is a discrete rotating wave:
\[
\gamma x(t) = x(t+\theta) \qquad \theta = mT/r
\]
for all $t \in \R$ and for some integer $m$. 
The group $K$ plays the role of a (global) synchrony pattern ($\theta = 0$), while $H$
(or $H/K$) plays the role of a (global) phase pattern. 
By~\cite[Corollary 3.7]{GS02}, both $H$ and $K$ are rigid properties of
$\X$. (The term used there is `robust'.) So the synchrony and phase
patterns for $(H,K)$ are rigid.
It is easy to prove that if $\GG$ is a network with symmetry group $\Gamma$
then every admissible map is $\Gamma$-equivariant. 
Therefore any synchrony or
phase pattern arising from a pair $(H,K)$ is rigid.
However, equivariant maps need not be admissible~\cite[Section 3.1]{AS07}.
Examples of synchrony and phase patterns of these kinds can be found
in many papers, for instance~\cite{AS06,AS07,AS08},
Buono and Golubitsky~\cite{BG01}, Golubitsky \etal~\cite{GMS16,GNS04,GS06,LG06,SG19}. 

It is well known that rigid synchrony patterns in networks can arise for reasons more general
 than symmetry. In particular, any balanced colouring of the nodes determines a rigid
synchrony pattern~\cite{GST05,SGP03}. If the Rigid Synchrony Conjecture holds for the network,
the converse is true. Moreover, if the Rigid Phase Conjecture also holds for the network, its rigid phase patterns arise from cyclic symmetries of 
quotient networks by balanced colourings~\cite{SP08}.

Another source of motivation for the Rigidity Conjectures is their analogues
for equilibria. It is proved in~\cite{GST05} that if a synchrony pattern of a steady state
is rigid, then the corresponding colouring is balanced. 
That is, synchronous nodes 
have synchronous inputs, up to input isomorphism. Aldis~\cite[Chapter 7]{A10} gives another
proof using transversality.
A third proof, using methods along the lines of this paper, is in~\cite{S20}.

\subsection{Rigidity Conjectures}
\label{S:RC}
In this paper, an orbit $\X=\{x(t)\}$ is defined to be {\em periodic} if there exists $T > 0$
such that $x(t+T) \equiv x(t)$, and $x(t)$ is not constant (which would be a steady state).
 A periodic orbit of a network may be steady at some nodes; 
 that is, some components $x_c(t)$ can be constant as $t$ varies.
However, $x_d(t)$ must {\em oscillate} (not be constant) for some node $d$.
This can happen, for example, in feedforward networks, where a steady node is an input to
oscillating nodes; see for example Golubitsky \etal~\cite{GNS04}.

The Rigidity Conjectures, stated about 15 years ago, comprise:
\begin{itemize}
\item[\rm (a)]
{\em Rigid Input Conjecture}:
For any rigid synchrony or phase pattern, synchronous or phase-related
nodes are input isomorphic. That is, they have the same number of input 
arrows for any given arrow-type.
\item[\rm (b)]
{\em Rigid Synchrony Conjecture}: 
For any rigid synchrony pattern, corresponding input nodes inherit 
the same synchrony pattern, if suitably identified.
\item[\rm (c)]
{\em Rigid Phase Conjecture}:
For any rigid phase pattern, corresponding input nodes inherit the same 
phase pattern, if suitably identified. 
\item[\rm (d)]
{\em Full Oscillation Conjecture}: If a transitive network has a hyperbolic
periodic state, there exist arbitrarily small admissible perturbations for which
every node oscillates. 
\end{itemize}
The identifications in (b) and (c) must be made using input isomorphisms,
which preserve the node dynamics and the numbers and types of couplings.
Although rigidity is not mentioned specifically in (d),
it is implicit: a node that is steady after any small perturbation is
rigidly steady. Moreover, (d) is a simple consequence of (c).

All four conjectures can also be stated as local versions, in which the relevant hypotheses are
assumed to hold only for some non-empty open interval of time. We append the word
`Local' to distinguish these. The local versions are not just generalisations:
they have technical advantages and are essential to this paper. As it happens,
the local versions imply the global ones, but this is not immediate
from the definitions. The Rigidity Conjectures are all motivated by the same intuition: if two nodes
have the same dynamics, except perhaps for a phase shift, then the same should be
true of the nodes that input to them, up to some bijection and for the same phase shift. In (a) 
the common feature is having the same number of inputs of each arrow type;
in (b) it is synchrony, in (c) it is a phase relation, and in (d) it is the node not oscillating. 

The condition of rigidity is required for  (a), (b), and (c), because without it, 
these conjectures are false \cite[Section 7]{SP07}. All known
counterexamples are `non-generic', having very special features
that can be destroyed by small admissible perturbations of the ODE.
It therefore makes sense to impose suitable genericity conditions.
The natural choice is rigidity: the property concerned persists under small 
 admissible perturbations. As mentioned in Section~\ref{S:rigidity}, the periodic orbit 
must be hyperbolic to ensure that a locally unique perturbed periodic orbit exists. 
See Section~\ref{S:HPO} for details.

For (d) a network is transitive (strongly connected, path-connected) if,
for any two nodes $c,d$, there is a directed path from $c$ to $d$.
Statement (d) can be false if the network is not transitive; for instance,
in feedforward networks nodes upstream from an oscillating node can be rigidly steady. 

\subsection{Previous Results}
\label{S:PR}

We summarise the current state of play for the Rigidity Conjectures.

Conjectures (a,b,c) were stated in 2006 in Golubitsky and Stewart~\cite[Section 10]{GS06},
with a claim that (a) had been proved using strongly admissible coordinate changes.
At that time, all four conjectures had been `folklore' for some years. Fully oscillatory
states are mentioned by Josi\v{c} and T\"or\"ok~\cite{JT06}
as hypotheses of existence theorems for periodic states of 
symmetric networks. Conjecture
(d) is stated explicitly in~\cite{SP08}.
A proof of conjecture (a) is presented in~\cite{SP07}. Proofs
of conjectures (a,b,d) are presented in Golubitsky \etal~\cite{GRW10} for stronger local versions.
Similar methods were applied in Golubitsky \etal~\cite{GRW12} to prove (c), again in
a local version. Since the local versions have weaker hypotheses,
they are stronger than the global ones. Curiously, the local versions are
technically more tractable, Sections~\ref{S:LRS} and \ref{S:CLS}.
Joly~\cite{J12} proves the Full Oscillation Conjecture, 
but only for fully inhomogeneous networks.
This proof uses transversality methods.

Readers familiar with the literature may wonder why we
refer to the above statements as conjectures. The reason is
that it has recently been noticed that
there is a gap in the proofs in~\cite{GRW10,GRW12,SP07}, 
which is related to the coordinate changes employed:
see~\cite[Appendix]{S20}. Specifically, it is assumed that certain coordinate changes are
`strongly admissible', when sometimes they are not. This gap can be repaired by requiring
the network to be {\em semihomogeneous}: such that
input equivalence is the same as state (previously cell) equivalence.
This class includes all homogeneous networks (all nodes have
the same number of input arrows of each arrow type)
and all fully inhomogeneous networks (all nodes and all arrows have different
types and no multiple arrows or self-loops occur), see
Golubitsky \etal~\cite{GGPSW19,GS17}. However, not all networks are 
semihomogeneous; a simple example is discussed in Section~\ref{S:3NE}.

As already remarked, 
in this paper we prove all four conjectures, for any finite network, but only by making
an extra assumption on the periodic orbit concerned, which we call
`strong hyperbolicity'. A weaker property, `stable isolation',
suffices, but existing proofs of this property for specific networks
establish strong hyperbolicity in any case. Both properties are
closely related to the Kupka-Smale Theorem, which (among other things)
asserts that hyperbolicity of all equilibria and periodic orbits is
generic in a general dynamical system: see Section~\ref{S:KST}. 
A Kupka-Smale analogue for admissible ODEs would imply strong hyperbolicity:
see Section~\ref{S:KSN}.

The analysis shows that
any hypothetical counterexample
to any of the Rigidity Conjectures must have at least 4 nodes and
involve a synchrony pattern with at least 3 colours. Moreover,
the periodic orbit concerned must have properties that are
highly non-generic in a general dynamical system, and seem
unlikely even when network constraints are imposed. Thus the result proved
here, although requiring stronger hypotheses, represent a significant 
strengthening of the evidence in support of the Rigidity Conjectures.

\subsection{Kupka-Smale Theorem}
\label{S:KST}

The Kupka-Smale Theorem plays a central role in this paper. It
asserts a form of genericity (open and dense)
for discrete dynamical systems (diffeomorphisms) or continuous ones (flows).
We paraphrase the result:

\begin{theorem}[\bf Kupka--Smale Theorem]
\label{T:KupkaSmale}
Let $M$ be a $C^\infty$ manifold and let
$\VV$ be the space of all $C^r$ vector fields on $M$ with the $C^r$ topology, $r \geq 1$.
For a general dynamical system on $M$, the following three
properties are generic in $\VV$; that is open and dense --- indeed, residual:
\begin{itemize}
\item[\rm (a)]
Every equilibrium point is hyperbolic.
\item[\rm (b)]
Every periodic orbit is hyperbolic. 
\item[\rm (c)]
The stable and unstable manifolds of all equilibria and periodic orbits
intersect transversely.
\end{itemize}
\end{theorem}
\begin{proof}
For precise statements and definitions 
see Kupka~\cite{K63}, Smale~\cite{S63}, and the simplified
proof in Peixoto~\cite[Section 2]{P66}.
\qed\end{proof}

There is a corresponding theorem for discrete dynamical systems (diffeomorphisms).
In this paper we require both the continuous and the discrete versions, but
the more delicate property (c) is not needed.

The original theorem of Kupka and Smale was restricted
 to compact manifolds, where the $C^r$ topology is the usual metric
 topology defined by bounding norms of derivatives of order up to $r$.
This restriction was removed by Peixoto, assuming the Whitney $C^r$ topology.
In this paper state spaces are $\R^k$, hence noncompact, but
 we can effectively reduce to the case of a compact manifold by using
 bump functions to make the admissible vector field vanish outside
 a compact set whose interior contains the periodic orbit under consideration.
 Therefore we can work with the metric $C^r$ topology. (The space
 of $C^\infty$ maps is not closed in this topology.) 
 We make these statements precise in Section~\ref{S:PRH}.

\subsection{Is Hyperbolicity Generic for Networks?}

For a general dynamical system (equivalent to an all-to-all coupled
fully inhomogeneous network) the Kupka-Smale Theorem proves
that hyperbolicity of all periodic orbits is a generic property.
Field~\cite{F80} proves an equivariant analogue,
which applies to some symmetric networks. Not all, because in this case admissible maps
are always equivariant, but equivariant maps for networks need not be admissible~\cite{AS06}.

Hyperbolicity and strong hyperbolicity are probably generic properties
for all networks, except for one class where this statement is known to 
be false by Jos\'ic and T\"or\"ok~\cite[Remark 1]{JT06}: see Section~\ref{S:FH}.
Indeed, the Kupka-Smale theorem fails for such networks.
We suspect that this class of networks contains all exceptions,
and that it can be avoided
in the proof of the Rigidity Conjectures below, but currently we are unable to prove
either of these statements.

\subsection{Implications Between the Conjectures}
\label{S:IBC}

In principle, any of the Rigidity Conjectures might be valid for some
networks, or some periodic orbits, but false for others.
It is therefore convenient to phrase the results of the conjectures as positive
properties of the network and the periodic orbit:

\begin{definition}\em
\label{D:properties}

Let $\GG$ be a network and let $\X$ be a hyperbolic periodic orbit of an admissible ODE.
Then the pair $(\GG,\X)$ has the following properties
if the stated conditions hold:
\begin{itemize}
\item[\rm (a)]
{\em Rigid Input Property (RIP)}:
For any rigid synchrony or phase pattern of $\X$, synchronous or phase-related
nodes are input equivalent. 
\item[\rm (b)]
{\em Rigid Synchrony Property (RSP)}: 
For any rigid synchrony pattern of $\X$, corresponding input nodes inherit 
the same synchrony pattern, if suitably identified.
\item[\rm (c)]
{\em Rigid Phase Property (RPP)}:
For any rigid pattern of phase relations of $\X$, corresponding input nodes inherit the same 
phase pattern, if suitably identified. 
\item[\rm (d)]
{\em Full Oscillation Property (FOP)}: If $\GG$ is transitive and $\X$ is a hyperbolic
periodic state, there exist arbitrarily small perturbations for which
every node oscillates. 
\end{itemize}
If any of the above statements holds for all hyperbolic $\X$, we say that $\GG$
has the property concerned.
\qed\end{definition}
Each conjecture states that all networks have the corresponding property.
There are local versions, where the stated conditions hold on a non-empty
open interval of time, which we call the Local Rigid Synchrony Property (LRSP), and so on.

The properties are closely related. Known implications among 
the corresponding properties, for any specific $\GG$, are:
\[
\mathrm{FOP} \Longleftarrow \mathrm{RPP} \Longleftrightarrow \mathrm{RSP}\implies \mathrm{RIP}
\]
Three implications are trivial:
$\mathrm{RPP} \implies \mathrm{RSP}$, 
$\mathrm{RSP} \implies \mathrm{RIP}$, and
$\mathrm{RPP} \implies \mathrm{RIP}$.
It is also clear that $\mathrm{RPP} \implies \mathrm{FOP}$,
because a node variable $x_c$ is in equilibrium if and only if
$x_c(t) \equiv x_c(t+\theta)$ for all $\theta \in \Sone$ (see
Section~\ref{S:FOC} for a full discussion and a more general version of the result).

The most surprising implication is that $\mathrm{RSP} \implies \mathrm{RPP}$.
This follows using the `doubling' trick of Golubitsky \etal~\cite{GRW12},
which converts a phase relation on $\GG$ into a synchrony relation 
on two disjoint copies $2\GG$, for a special
periodic orbit on a 2-torus foliated by periodic orbits. Some care is needed to show that 
$2\GG$ has suitable versions of the required properties, see Section~\ref{S:RPC}.

\section{Summary of Paper}

In~\cite{GRW10,GRW12} the conjectures are proved in the order
\[
\mbox{Rigid Input}\ \implies\ \mbox{Full Oscillation}\
	 \implies\ \mbox{Rigid Synchrony}\ \implies\ \mbox{Rigid Phase}
\]
The viewpoint there is local: local versions are more tractable, and
lead to stronger results while avoiding technical obstacles.
Here we also consider local versions, for similar reasons, but we
employ a different strategy:
\beqn
&& \mbox{Rigid Synchrony}\ \implies\ \mbox{Rigid Phase}\ \implies\ \mbox{Full Oscillation} \\
&& \qquad \Downarrow \\
&& \mbox{Rigid Input}
\eeqn
The key result for the method employed in this paper is therefore the local form of
(b), the Local Rigid Synchrony Property.

Consider an admissible ODE $\dot x = f(x)$, where $x = (x_1, \ldots, x_n)$,
and write it in components as $\dot x_c = f_c(x)$, where $1 \leq c \leq n$.
Assume, for a contradiction, that this ODE has a hyperbolic periodic orbit $\X=\{x(t)\}$ with
a rigid synchrony pattern $\bowtie$. Any such orbit satisfies the condition
\[
c \bowtie d \implies x_c(t) = x_d(t)\quad \forall t \in \R
\]
which implies that $x_c$ satisfies two equations:
\[
\dot x_c = f_c(x) \qquad
\dot x_c = f_d(x)
\]
If the colouring is balanced, these equations are identical, but if it
is unbalanced, they are formally inconsistent --- they involve different functions
or the same function evaluated at different points. This does not of itself create a contradiction,
because $f_c$ and $f_d$ might agree on $\X$ while
being distinct elsewhere. 
It seems highly implausible that such
a situation can persist under all small admissible perturbations, but proving
that is another matter. In the unbalanced case, the result of this substitution is 
an `overdetermined ODE' (OODE), with more equations than unknowns,
and the method of proof that we employ is to construct
perturbations that exploit the formal inconsistency of the OODE to
derive a contradiction.

Our method requires an additional assumption on the orbit $\X$:  
`strong hyperbolicity' (Section~\ref{S:SH}), or the weaker property of
`stable isolation' (Section~\ref{S:SI}). We can establish this property
rigorously for some networks and colourings.
Heuristically, this extra property is plausible, and hyperbolic periodic orbits that lack it
`ought to be' of infinite codimension, hence highly non-generic. However,
we are unable to prove this in full generality with current methods. For further discussion
see Sections~\ref{S:SH} and \ref{S:SI}.

As motivation, and to provide a simple example of the
proof technique, Section~\ref{S:3NE} considers the special case of
a 3-node directed ring with two arrow types. In this network, admissible diagonal maps
are not strongly admissible, so the results of Golubitsky \etal~\cite{GRW10,GRW12} do not
apply. Using the relevant OODEs, we prove the Rigid Synchrony Conjecture for this network.
This example emphasises the central role of the Kupka-Smale Theorem.

Section~\ref{S:DN} recalls
the relevant features of the basic formalism of coupled cell networks (henceforth just `networks')
from~\cite{GST05,SGP03}, including properties of the quotient network
by a balanced colouring. An important point is a simple characterisation of admissible
maps in Proposition~\ref{P:char_admiss}, originally proved
in~\cite[Proposition 4.6]{SGP03}, which reduces admissibility to invariance under
the appropriate vertex group.

Section~\ref{S:QQ} sets up a generalisation of the usual quotient network
construction for a balanced colouring $\bowtie$, by throwing away the balance condition.
The resulting `quasi-quotient' depends on a choice of representatives
$\RR$ for the colouring, and fails to have most of the useful properties of the quotient.
However, it retains two key properties in relation to states with synchrony
pattern $\bowtie$: such states induce solutions for the induced ODE on the quasi-quotient,
and solutions of the induced equation that also satisfy a set of constraint equations
lift to solutions of the original ODE with synchrony pattern $\bowtie$.

Section~\ref{S:PRH} reviews standard results concerning hyperbolicity
and genericity in general dynamical systems, provides rigorous definitions of the $C^1$ norm
and rigidity, relates these concepts to network dynamics, and introduces
two properties that are central to the methods of this paper: stable isolation and
strong hyperbolicity. It ends with a discussion of the Kupka-Smale Theorem and
network analogues.

Section~\ref{S:LR} begins the general programme to prove the
Local Rigid Synchrony Conjecture for any network under the assumption of strong hyperbolicity. Here we define a local
version of rigid synchrony, needed to set up the proof. We discuss
lower semicontinuity of colourings, the existence of generic points for
local rigid synchrony, and technical obstacles that arise if instead
we try to work with global rigid synchrony.

Section~\ref{S:CAP} describes a method for constructing admissible
perturbations with small support and small $C^1$ norm. The construction
uses bump functions and a symmetrisation technique.

The Local Rigid Synchrony Conjecture, which lies at the heart of this paper,
is proved in Section~\ref{S:LRSC} under the hypothesis of strong hyperbolicity.  
The main obstacle to finding such a proof has always been to gain enough control
over how the perturbed periodic orbit $\tilde\X$ moves 
when the vector field $f$ is perturbed. Previous
attacks on the conjecture employ various strategies to do this, such as
constructing flows geometrically~\cite{SP07} or using 
perturbations related to adjacency matrices and delicate estimates 
for integrals along the periodic orbit~\cite{GRW10}.
Instead, we analyse the quasi-quotient $\GG^\RR$ for a set of representatives $\RR$
of the synchrony colouring and the structure of the
resulting OODE. This leads to an induced ODE for $\GG^\RR$ on the
synchrony space $P^\RR$, together with
constraint equations. We retain control of $\tilde\X$ by constructing perturbations that
leave the corresponding orbit in $P^\RR$ unchanged, assuming local rigidity.
Indeed the constraint equations created by local rigidity imply that
$\tilde\X = \X$ for sufficiently small perturbations of that kind, 
so the periodic orbit in $P$ is also unchanged.
The constraint equations then lead to a contradiction
when the synchrony colouring is unbalanced.
The Rigid Input Property follows immediately because balanced
colourings refine input equivalence.

Section~\ref{S:GRS} parlays this local theorem into a proof of the standard
global version of the Rigid Synchrony Property, using the lattice of
colourings \cite{S07}.

The Local Rigid Phase Property is deduced in Section~\ref{S:RPC}
using the same trick that underlies the proof in~\cite{GRW12}, 
namely, form two disjoint copies of the network and
convert a phase relation into a synchrony relation. The 
Local Rigid Synchrony Property does not apply directly, because doubling
the network destroys hyperbolicity and strong hyperbolicity;
instead of being isolated, periodic orbits
defined by phase-shifted pairs foliate a 2-torus. However, rigidity of the phase
shift gives a {\em canonical} choice of perturbed periodic
orbit among those on the torus. The proof of the Local Rigid Synchrony Conjecture
generalises directly to this situation, and the conclusion that any 
rigid synchrony relation is balanced implies the Local Rigid Phase Conjecture.

We deduce a local version of the Full Oscillation Conjecture in Section~\ref{S:FOC}. 
The main observation is that
if a periodic state of the network is steady at some node, for some time interval $J$,
then on $J$ that node
is phase-related to itself by all possible phase shifts. 

Section~\ref{S:CAGHKT} provides a brief discussion of implications 
for analogues of the $H/K$ Theorem in equivariant dynamics, 
whose proofs, given in~\cite{SP08,GRW12} are contingent upon the four 
Rigidity Properties. These results are striking consequences of those properties
because they classify all possible locally or globally rigid synchrony and phase 
patterns. In particular they show that rigid phase patterns arise from cyclic group
symmetries of a quotient network (or its completion if the quotient is
not transitive).

In special cases, strong hyperbolicity of the periodic orbit
can be proved to be generic using current techniques. In Section~\ref{S:12C} we carry out this programme for
1- and 2-colourings, and deduce the Rigidity Properties for all networks with
at most 3 nodes. Here a key role is played by the notion of ODE-equivalence:
distinct networks with the same space of admissible maps. We classify
2-node networks up to ODE-equivalence.

Finally, we remark on the strong link that emerges from the methods of this paper,
relating the four Rigidity Properties to results of Kupka-Smale type for networks,
and point out that if a counterexample to any of the Rigidity Conjectures exists,
the periodic orbit $\X$ must satisfy conditions that would be extremely degenerate in a
general dynamical system.

\section{3-Node Example}
\label{S:3NE}

We set the scene by analysing a simple example, used later to motivate the
proof of the Rigid Synchrony Conjecture for strongly hyperbolic periodic orbits. 
(For this network, strong hyperbolicity can be proved.)
To focus on the main idea we postpone
formal definitions of the terminology to Section~\ref{S:DN}. 
The discussion should be clear without these.
The main feature of the general case that does not arise for 
this example is invariance under
vertex groups. This step requires a straightforward symmetrisation, Section~\ref{S:symmetrisation}.

Figure~\ref{F:3node_ring} shows a network $\GG$ with three nodes,
$\CC = \{1,2,3\}$. There are two arrow types, solid and dashed.
The shading on the nodes indicates `input type', formalised below.
Each of nodes 1 and 2 receives its input from a single solid arrow, so they have the same
input type. Node 3 receives its input from a single dashed arrow, so it has a
different input type.

A network is {\em semihomogeneous} if state equivalent nodes are input equivalent.
This class includes all homogeneous networks and all fully inhomogeneous ones, together with many others. The proofs in \cite{GRW10,GRW12} are valid for
all semihomogeneous networks.
However, there exist networks that are not semihomogeneous,
the simplest example being $\GG$. Indeed,  
all nodes of $\GG$ are cell (or state) equivalent, but there are two distinct
input equivalence classes, $\{1,2\}$ and $\{3\}$.
Therefore the results of~\cite{GRW10, GRW12} do not apply to $\GG$. 
However, by constructing suitable admissible perturbations, we
prove that in fact it does have all four Rigidity Properties. 
The central role of the Kupka-Smale Theorem arises naturally from the method.

\label{ex:3ring}
\begin{figure}[!htb]
\centerline{%
\includegraphics[width=.18\textwidth]{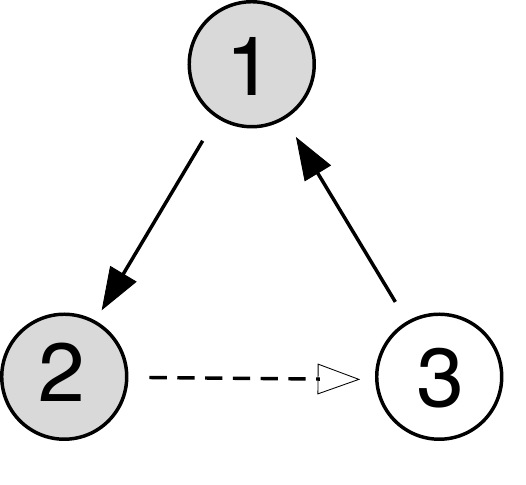}
}
\caption{\label{F:3node_ring} A $3$-node directed ring $\GG$, with two arrow types,
that is not semihomogeneous.}
\end{figure}

Admissible ODEs for $\GG$ (defined in Section~\ref{S:AODE}) have the form
\begin{equation}
\label{e:3nodeODE}
\begin{array}{rcl}
\dot{x}_1 &=& f(x_1,x_3) \\
\dot{x}_2 &=& f(x_2,x_1) \\
\dot{x}_3 &=& g(x_3,x_2)
\end{array}
\end{equation}
Here the variables $x_c$ lie in {\em node spaces} $P_c$, which we
take to be real vector spaces $\R^{k_c}$.
In order for \eqref{e:3nodeODE} to respect the network architecture,
the domains and ranges of these functions must be:
\beqn
f &:& P_1 \times P_3 \to P_1 \\
f &:& P_2 \times P_1 \to P_2 \\
g&:& P_3 \times P_2 \to P_3
\eeqn
The function $f$ occurs twice, so the domains and ranges in the two cases
must coincide. That is, $P_1 = P_2$ and $P_3 = P_1$. In other words,
the network structure and the definition of admissible ODEs requires
$P_1 = P_2 = P_3$. This is an example of `state equivalence', which
replaces the usual `compatibility conditions' on head and tail nodes of arrows
in~\cite{GST05, SGP03}.
The reasons for this change are discussed in Section~\ref{S:AODE},
and in greater detail in~\cite{GS22}.

\subsection{Strong Admissibility}

We find the strongly admissible maps for this network and verify
their composition properties directly. 
Some notation defined in Section~\ref{S:DN} is convenient. To avoid complications
concerning state equivalence, we use the previous notation $\sim_C$
for cell equivalence and $\sim_I$ for input equivalence.
here. (In this case, cell equivalence with the previous compatibility
conditions turns out to be the same as state equivalence.)

Strongly admissible maps are defined in~\cite{GST05} as `diagonal'  maps 
\begin{equation}
\label{E:diag_phi}
g(x) = (g_1(x_1), \ldots, g_n(x_n))
\end{equation}
such that $g_i = g_j$ whenever $i \sim_C j$. It is proved there that
if $f$ is admissible and $g$ is strongly admissible, then both
$f \compose\, g$ and $g\, \compose f$ are admissible.

In~\cite{GRW10, GRW12, SP07} it was tacitly assumed that the same composition
properties hold if $g_i = g_j$ whenever $i \sim_I j$; that is, if $g$ is
diagonal and admissible. We show that this statement is false for $\GG$,
implying that new methods are required to 
prove any of the Rigidity Conjectures for this network. 

Since the two solid arrows have the same type, the previous compatibility
condition requires all three nodes to have the same
cell type. The network $\GG$ has
two different input types $\{1,2\}$ and $\{3\}$, so $\sim_C$ is different from $\sim_I$.
Admissible diagonal maps~\eqref{E:diag_phi}
have $g_1 = g_2 \neq g_3$. Strongly admissible
maps, as defined in~\cite[Definition 7.2]{GST05}, have $g_1 = g_2 = g_3$.

First, we show that the only diagonal maps $g$ that compose
on the right with admissible maps $f$ to give admissible maps 
are maps $g$ where $g_1 = g_2 = g_3$. 
That is, $g_i=g_j$ whenever $i \sim_C j$, in accordance with~\cite[Lemma 7.3]{GST05}. 
We have
\[
f(x) = \Matrix{f_1(x_1,x_3) \\ f_1(x_2,x_1) \\ f_2(x_3,x_2)} \qquad
	f \compose g(x) = \Matrix{f_1(g_1(x_1),g_3(x_3)) \\ f_1(g_2(x_2),g_1(x_1)) \\ f_2(g_3(x_3),g_2(x_2))}
\]
If $f \compose g$ is admissible for all $f$, then the first two components yield
\[
f_1(g_2(x_2),g_1(x_1)) = f_1(g_1(x_2),g_3(x_1))
\] 
Take $f_1(u,v) = u$ to give
$g_2(x_2) = g_1(x_2)$, and then $f_1(u,v) = v$ to give $g_1(x_1) = g_3(x_1)$.
Therefore $g_1=g_2=g_3$.

Conversely, any map of this form composes on the right to give an admissible map.

In contrast, we now show that the maps $g$ that compose
on the left with admissible maps $f$ to give admissible maps
are maps $g$ where $g_1 = g_2$. That is, $g_i=g_j$ whenever $i \sim_I j$.
These are precisely the admissible diagonal maps.
Now
\[
	g \compose f(x) = \Matrix{g_1(f(x_1,x_3)) \\ g_2(f(x_2,x_1)) \\ g_3(f(x_3,x_2))}
\]
If $g \compose f$ is admissible for all $f$, then the first two components yield
\[
g_2(f(x_1,x_3)) = g_1(f(x_1,x_3))
\] 
Take $f(u,v) = u$, obtaining $g_2(u) = g_1(u)$. Therefore $g_1=g_2$. 

Conversely, any map of this form composes on the left to give an admissible map.

\subsection{Construction of Suitable Perturbations}
\label{S:CSP}

Since the network $\GG$ of Figure~\ref{F:3node_ring} is not semihomogeneous, 
the results of~\cite{GRW10, GRW12, SP08} do not apply.
Nevertheless, we now prove by a different method
that $\GG$ has the Local Rigid Synchrony Property. As mentioned in
Section~\ref{S:IBC} and proved in Sections~\ref{S:RPC} and \ref{S:FOC},
it therefore has the other three Local Rigidity Properties as well.
The method used for this example motivates the subsequent approach
to synchrony patterns on arbitrary networks. For this
network we obtain a complete proof, because we can apply
the standard Kupka-Smale Theorem and the equivariant version
of Field~\cite{F80} for the symmetry group $\Z_2$. In the general case
some network version, not yet proved, is required: this is why we impose
strong hyperbolicity in the bulk of this paper.

We discuss the first case in detail, to establish the logic, and
provide less detail for the other cases.

Admissible ODEs take the form $\dot{x} = f(x)$ with certain conditions on the components $f_c$.
We consider an arbitrary $1$-parameter family of
perturbations $\dot{x} = f(x)+\eps p(x)$, where $p$ is also admissible.
Explicitly, admissibility requires:
\begin{equation}
\begin{array}{rcl}
\label{E:3nodeODEpert}
\dot{x}_1 &=& f(x_1, x_3) + \eps p(x_1, x_3)\\
\dot{x}_2 &=& f(x_2, x_1)  + \eps p(x_2, x_1) \\
\dot{x}_3 &=& g(x_3, x_2)  + \eps q(x_3, x_2)
\end{array}
\end{equation}
where $p,q$ are arbitrary smooth functions because vertex symmetries are
trivial. We can choose $p,q$ to be bounded
using bump functions, so the perturbation is $C^1$-small when $\eps \ll 1$. See
Section~\ref{S:BF}.

The only balanced colouring is the trivial one with all nodes of different colours.
We show that for every other colouring $\bowtie$, rigid synchrony leads to a contradiction. 
This is obtained by applying the Kupka-Smale Theorem
(or Field's equivariant version) for certain $1$- and $2$-node networks, when $p$ is constructed
to have certain properties that depend on the colouring $\bowtie$. Throughout we assume
only that the synchrony pattern $\bowtie$ is valid for $t$ in some non-empty open interval $J$,
and choose a point $t_0 \in J$.

\vspace{.1in}
\noindent
Case (A): $\bowtie = \{\{1,2,3\}\}$.

(Here and elsewhere we write $\bowtie$ as a partition of $\CC$.)

For given $\eps$, any (periodic) orbit $\X^\eps = (x_1^\eps,x_2^\eps,x_3^\eps)$ 
with this synchrony pattern has the fully synchronous form
\[
(u^\eps(t),u^\eps(t),u^\eps(t))
\]
We assume that $\X^0$ is hyperbolic and $0 < \eps \ll 1$.
Taking a suitable Poincar\'e section $\Sigma$ at $x(t_0)$ and setting initial conditions
by requiring $x^\eps(t_0) \in \Sigma$, we can assume that $x^\eps(t)$ varies
continuously (indeed, by the Implicit Function Theorem applied to a
first-return map, smoothly) with $\eps$ and $t$.

Substituting $x_1^\eps=x_2^\eps=x_3^\eps = u^\eps$ in~\eqref{E:3nodeODEpert}, 
this state must satisfy the conditions
\begin{equation}
\label{E:3node_eps}
\begin{array}{rcl}
\dot u^\eps &=& f(u^\eps,u^\eps) + \eps p(u^\eps,u^\eps)\\
\dot u^\eps &=& f(u^\eps,u^\eps) + \eps p(u^\eps,u^\eps)\\
\dot u^\eps &=& g(u^\eps,u^\eps) + \eps q(u^\eps,u^\eps)
\end{array}
\end{equation}
The first component determines $u^\eps$ uniquely, for given initial conditions.
The second is the same as the first. The third is different, and potentially
contradictory; we use it to derive a contradiction.

The projection $\U^0=\{(u^0(t)\}$ of $\X^0$ into $P_1=P_2=P_3$ is a periodic orbit of
 the `induced ODE'
\begin{equation}
\label{E:induced_3node_1}
\dot{y} = f(y,y)
\end{equation}
We pre-prepare $f$ so that $\U^0$ is hyperbolic on $P_1$. This follows from the
Kupka-Smale Theorem, since \eqref{E:induced_3node_1} is 
an ODE on $P_1$ and any perturbation $p(y)$ can be expressed in the form $g(y,y,y)$.
By rigidity, the local synchrony pattern $\{\{1,2,3\}\}$ applies to this perturbed ODE
provided we make the perturbation small enough. The open interval $J$ may have to
be replaced by a smaller open interval $J'$ where $t_0 \in J' \subseteq J$.

Having pre-prepared $f$ and $\X^0$ to make $\U^0$ hyperbolic on $P_1$,
we can realise the contradiction as follows.
To simplify notation, write
\[
u^* = u^0(t_0)
\]
Define $p(x) \equiv 0$ for all $x$, and define $q$ so that $q(u^*,u^*) \neq 0$.
This is possible because $p,q$ are arbitrary independent smooth maps.
The first equation now becomes 
\[
\dot u^\eps = f(u^\eps,u^\eps)
\]
which is the same as the unperturbed equation. Therefore, near $(u^*,u^*)$,
the periodic orbits $u^\eps$
and $u^0$ satisfy the same ODE, and $u^\eps \to u^0$
as $\eps \to 0$.

Since $\U^0$ is hyperbolic on $P_1$, there is a locally unique periodic orbit
near $u^0(t)$. But $u^\eps \to u^0$ as $\eps \to 0$. Therefore, for small $\eps>0$,
we have $u^\eps(t) \equiv u^0(t)$ 
for all $t$ near $t_0$ (which implies equality for all $t$ by uniqueness of solutions to ODEs).
Therefore $\dot u^\eps(t) \equiv \dot u^0(t)$. Set $t = t_0$ to obtain
\[
g(u^*,u^*) + \eps q(u^*,u^*) = g(u^*,u^*) + 0. q(u^*,u^*) = g(u^*,u^*)
\]
This implies that $\eps q(u^*,u^*)=0$ for some $\eps > 0$. 
However, we chose $q$ so that this is false. This contradiction implies
that $\X$ cannot have local rigid synchrony pattern $\{1,2,3\}$.

\vspace{.1in}
\noindent
Case (B): $\bowtie = \{\{1,2\},\{3\}\}$.

By Case (A) we may assume that $\bowtie$ is the finest colouring
such that $\X$ has local synchrony pattern $\bowtie$ at $t_0$.
We follow similar reasoning, and omit routine details. 

For given $\eps$, any (periodic) orbit $\X^\eps = (x_1^\eps,x_2^\eps,x_3^\eps)$ 
with this local synchrony pattern has the form
\[
(u^\eps(t),u^\eps(t),v^\eps(t))
\]
Substitute $x_1^\eps=x_2^\eps= u^\eps$ and $x_3^\eps = v^\eps$
 in~\eqref{E:3nodeODEpert} to obtain
\begin{eqnarray}
\label{E:(12)(3)a}
\dot u^\eps &=& f(u^\eps,v^\eps) + \eps p(u^\eps,v^\eps)\\
\label{E:(12)(3)b}
\dot u^\eps &=& f(u^\eps,u^\eps) + \eps p(u^\eps,u^\eps)\\
\label{E:(12)(3)c}
\dot v^\eps &=& g(v^\eps,u^\eps) + \eps q(v^\eps,u^\eps)
\end{eqnarray}
Components \eqref{E:(12)(3)a} and \eqref{E:(12)(3)c} determine $u^\eps$ and $v^\eps$
uniquely, for given initial conditions.
Equation \eqref{E:(12)(3)b} is formally different from \eqref{E:(12)(3)a}, 
and potentially contradictory.

The perturbation terms in \eqref{E:(12)(3)a} and \eqref{E:(12)(3)c}
have the form $(p(u^\eps,v^\eps),q(v^\eps,u^\eps))$, which is
a general vector field on $P_1 \times P_3$ with variables $u^\eps,v^\eps$.
We can therefore apply the Kupka-Smale Theorem (for a general dynamical system)
to pre-prepare $f, \X$ so that the projected orbit
$\U^0 = \{(u^0(t),v^0(t))\}$ is hyperbolic on $P_1\times P_3$.
We retain the same notation.

Choose a time $t_1 \in J$ so that if $u^* = u^0(t_1)$ and $v^* = v^0(t_1)$
then $u^* \neq v^*$. 
If this is not possible then we are in Case (A), already dealt with; this is also
contrary to $\bowtie$ being the finest local synchrony pattern.

Define $q \equiv 0$, and define $p$ so that $p(u, v) \equiv 0$
in a neighbourhood of $(u^*, v^*)$,
but $p(u^*,u^*) \neq 0$. This is possible since $u^* \neq v^*$, so
$(u^*,u^*) \neq (u^*,v^*)$. Indeed, we can use a bump function to
make $p$ vanish outside
a small neighbourhood of $(u^*,u^*)$, but be nonzero near $(u^*,u^*)$.

When $(u,v)$ is near $(u^*,v^*)$, equations \eqref{E:(12)(3)a} and \eqref{E:(12)(3)c} reduce to
\begin{eqnarray}
\label{E:eq1_eps}
\dot u^\eps &=& f(u^\eps,v^\eps) \\
\label{E:eq2_eps}
\dot v^\eps &=& g(v^\eps,u^\eps)
\end{eqnarray}
which is the same ODE as the unperturbed equation, but with variables 
$u^\eps, v^\eps$ in place of $u^0, v^0$.

As before, the pre-preparation guarantees
local uniqueness of perturbed periodic orbits on $P_1 \times P_3$, 
so this implies that $p(u^*, u^*) = 0$, a contradiction.

\vspace{.1in}
\noindent
Case (C): $\bowtie = \{\{1,3\}, \{2\}\}$

The argument has a similar structure. Synchronous orbits have the form
\[
(u^\eps(t),v^\eps(t),u^\eps(t))
\]
Substitute $x_1^\eps = x_3^\eps= u^\eps$ and $x_2^\eps = v^\eps$
 in~\eqref{E:3nodeODEpert} to obtain
\begin{eqnarray}
\label{E:(13)(2)a}
\dot u^\eps &=& f(u^\eps,u^\eps) + \eps p(u^\eps,u^\eps)\\
\label{E:(13)(2)b}
\dot v^\eps &=& f(v^\eps,u^\eps) + \eps q(v^\eps,u^\eps)\\
\label{E:(13)(2)c}
\dot u^\eps &=& g(u^\eps,v^\eps) + \eps p(u^\eps,v^\eps)
\end{eqnarray}
Components \eqref{E:(13)(2)b} and \eqref{E:(13)(2)c} determine $u^\eps$ and $v^\eps$
uniquely, for given initial conditions.
Equation~\eqref{E:(13)(2)a} is formally different from \eqref{E:(13)(2)a}.
When $\eps = 0$ equations \eqref{E:(13)(2)b} and \eqref{E:(13)(2)c}
are a general ODE on $P_1 \times P_2$, and $(p,q)$ is an arbitrary
vector field on $P_1 \times P_2$. We can use the Kupka-Smale Theorem to
pre-prepare $f,g$ so that the periodic orbit $\U^0 = \{u^0(t),v^0(t)\}$ 
is hyperbolic on $P_1 \times P_2$.

Choose the perturbation so that $q \equiv 0$, $p(v,u) \equiv 0$ near $(v^*,u^*)$,
but $p(u^*,u^*) \neq 0$. When $(u,v)$ is near $(u^*,v^*)$, the orbit $(u^\eps,v^\eps)$ satisfies
the same ODE as $(u^0,v^0)$, and hyperbolicity implies local uniqueness,
so $(u^\eps(t),v^\eps(t)) = (u^0(t),v^0(t))$ near $t_1$. As before, this
implies that $p(u^*,u^*) = 0$, a contradiction.

\vspace{.1in}
\noindent
Case (D): $\bowtie = \{\{2,3\},\{1\}\}$

Again the argument has a similar structure. Synchronous orbits have the form
\[
(u^\eps(t),v^\eps(t),v^\eps(t))
\]
Substitute $x_1^\eps = u^\eps$ and $x_2^\eps =  x_3^\eps= v^\eps$
 in~\eqref{E:3nodeODEpert} to obtain
\begin{eqnarray}
\label{E:(23)(1)a}
\dot u^\eps &=& f(u^\eps,v^\eps) + \eps p(u^\eps,v^\eps)\\
\label{E:(23)(1)b}
\dot v^\eps &=& f(v^\eps,u^\eps) + \eps q(v^\eps,u^\eps)\\
\label{E:(23)(1)c}
\dot v^\eps &=& g(v^\eps,v^\eps) + \eps p(v^\eps,v^\eps)
\end{eqnarray}
Components \eqref{E:(23)(1)a} and \eqref{E:(23)(1)b} determine $u^\eps$ and $v^\eps$
uniquely, for given initial conditions.
Equation~\eqref{E:(23)(1)c} is formally different from \eqref{E:(23)(1)b}.

When $\eps = 0$ equations \eqref{E:(23)(1)a} and \eqref{E:(23)(1)b}
are a general $\Z_2$-equivariant ODE on $P_1 \times P_2$,
where $\Z_2$ swaps $u^\eps$ and $v^\eps$,
and $(p,q)$ is an arbitrary $\Z_2$-equivariant
vector field on $P_1 \times P_2$. We can use Field's equivariant
Kupka-Smale Theorem to
pre-prepare $f,g$ so that the periodic orbit $\U^0 = \{u^0(t),v^0(t)\}$ 
is hyperbolic on $P_1 \times P_2$.

Choose the perturbation so that $q \equiv 0$, $p(u,v) \equiv 0$ near $(u^*,v^*)$,
but $p(v^*,v^*) \neq 0$. When $(u,v)$ is near $(u^*,v^*)$, the orbit $(u^\eps,v^\eps)$ satisfies
the same ODE as $(u^0,v^0)$, and hyperbolicity implies local uniqueness,
so $(u^\eps(t),v^\eps(t)) = (u^0(t),v^0(t))$ near $t_1$. As before, this
implies that $p(v^*,v^*) = 0$, a contradiction.

We conclude that the only local
rigid synchrony pattern is trivial, verifying the Local Rigid Synchrony Property for $\GG$.
The other three Local Rigidity Properties follow, as outlined in
Section~\ref{S:IBC} and discussed in detail in Sections~\ref{S:RPC} and \ref{S:FOC}.

\begin{remark}\em
Since the Rigidity Conjectures were first stated it has been clear that
the main obstacle to proving them is to retain enough control over
the behaviour of the perturbed periodic orbit. In~\cite{GRW10,GRW12}
this is achieved by delicate estimates. The method employed above
controls the perturbed periodic orbit by {\em not changing it}. Obviously a zero
perturbation has this property, but the admissible perturbation that we construct
changes the constraint equations. This construction leads to a contradiction
when the local synchrony
colouring is rigid but not balanced.
 
This example suggests that a similar type of perturbation of the induced
ODE on the synchrony space might be used
for an arbitrary network, and that the main obstacle is to prove
a suitable version of the Kupka-Smale Theorem, so that (assuming rigidity) the
perturbed periodic orbit is the same as the unperturbed one, but
the constraints of synchrony lead to a contradiction. In the rest of the paper
we show that this approach succeeds, modulo a version of Kupka-Smale
for networks. 
\qed\end{remark}

\section{Formal Definition of a Network}
\label{S:DN}

We now proceed to the general case.
First, we briefly recall some basic concepts of the `coupled cell' network formalism
introduced in~\cite{SGP03}
and generalised in~\cite{GST05}, and state some standard notations, definitions, and results.
For further details, see~\cite{GRW10,GS06,GS22,GST05, S20}.
We introduce a further slight generalisation, which resolves
the dual role of `cell equivalence' in the previous formalism. All of
the standard theory extends to this more general setting, which applies to
a wider range of ODEs with network structure. Full details are 
presented in~\cite{GS22}; everything in this paper is valid in this more general setting.

We begin with the formal setting for networks:
\begin{definition}\em
\label{d:networkdef}
A {\em network} $\GG = (\CC,  \sim_C, \AA,  \sim_A, \HH, \TT)$ consists of:

(a) A finite set of {\em nodes} $\CC$ and a {\em node-type} assigned to each node.  Write
\[
c \sim_C d
\]
if $c, d\in\CC$ have the same node-type.

(b) A finite set of {\em arrows} $\AA$ and an {\em arrow-type}
assigned to each arrow.  Write
\[
a \sim_A b
\]
if $a, b \in \AA$ have the same arrow-type.  (The previous notation 
uses $\EE$ for $\AA$ and $\sim_E$ for $\sim_A$.)

The node type can be viewed as a distinguished `internal' arrow-type.

(c) Each $a \in \AA$ has a {\em head node} $\hd(a)$ and
 a {\em tail node}
$\tl(a)$ in $\CC$. When viewing a node $c\in\CC$ as an internal arrow, 
we define $\hd(c)=c=\tl(c)$.
\qed\end{definition}

\begin{remark}\em
\label{R:compatibility}
Readers familiar with the literature will observe that we have
omitted from this definition the standard `compatibility condition' that arrow-equivalent
arrows have node-equivalent heads and node-equivalent tails. 
This condition combines two roles for cell-equivalence that are
better kept distinct, namely equality of state spaces (which we call $\sim_S$ below)
and equality of the distinguished `internal arrows' on nodes 
(where we retain the notation $\sim_C$). In its place,
we impose a natural condition on the state spaces (or phase spaces) assigned to nodes,
see Definition~\ref{D:admiss}(a).
\qed\end{remark}

\subsection{Input Sets and Tuples}

\begin{definition}\em  \label{D:input}
Let $c, d$ be nodes in $\CC$. 

(a)	The {\em input set} of $c$ is the set
$I(c)$ of all arrows $a\in\AA$ such that $\hd(a) =c$.

(b) An {\em input isomorphism} is an arrow-type preserving bijection
$\beta:I(c) \to I(d)$.
That is, $c\sim_C d$ and $a \sim_A \beta(a)$ for all $a \in I(c)$.  

(c) Two nodes $c$ and $d$ are {\em input isomorphic} 
or {\em input equivalent} if there exists an input isomorphism from $I(c)$ to $I(d)$. 
In this case we write
\[
c\sim_I d
\]
\qed\end{definition}

The set of input isomorphisms from $c$ to $d$ is denoted by $B(c,d)$. 
The disjoint union of these sets 
\begin{equation} \label{e:groupoidn}
\BB = \dot\bigcup_{c,d\in\CC} B(c,d)
\end{equation}
is a groupoid, see Brown~\cite{B87}, Higgins~\cite{H71}, and~\cite{GS06,GST05,SGP03}.

\subsection{Redundancy}
\label{S:redundancy}

The definitions of node- and arrow-types, as stated, allow nodes or arrows
to be assigned the same type even when they are not related by an input isomorphism ---
that is, they are in different groupoid orbits.
This {\em redundancy} is often convenient, especially when drawing
network diagrams. It does not affect the class of admissible maps, which
depends only on the input isomorphisms,
but it can cause problems in some constructions
and introduces an ambiguity into the definition of the adjacency matrix
for a given arrow type. Redundancy can be avoided by
requiring the types to be the same if and only if the nodes or arrows
are related by an input isomorphism. The resulting network is
said to be {\em irredundant}, and we assume this throughout.

\subsection{Admissible Maps and ODEs}
\label{S:AODE}

We now define admissible maps and ODEs, and state an equivalent property
that is central to this paper.
 
Assign to each node $c \in \CC$ a {\em node space} $P_c$. This
is usually taken to be a real vector space $\R^{k_c}$, and we make this assumption throughout
the paper. The overall state space of the {\em network system} (or {\em coupled cell system})
of ODEs is
\[
P = P_1 \times \cdots \times P_n
\]
In node coordinates, a map $f: P \to P$ has components $f_c$ for $c \in \CC$ such that
\[   
f_c: P \rightarrow P_c
\]

\begin{remark}\em
More generally,  node spaces can be 
$C^\infty$ manifolds, Field~\cite{F04}. In {\em phase oscillator} models
all node spaces are the circle, so $P_c = \Sone$.
The methods employed in this paper probably generalise to manifolds. However,
Golubitsky \etal~\cite{GMS16} show that the topology
of node spaces can change the list of possible phase patterns in the $H/K$ Theorem, 
so it should not be assumed that all of the results proved here automatically remain
valid when node spaces are manifolds, or that they are independent of their topology.
\qed\end{remark}

In Example~\ref{ex:3ring} we noted that in order for admissible ODEs for make sense,
certain equalities are forced on node state spaces. These equalities
arise whenever nodes $c \neq d$ are input isomorphic. For any input arrow
$e \in I(c)$, and to $e' = \beta(e) \in I(d)$ where $\beta \in B(c,d)$, we require
\begin{equation}
\label{E:state_equiv}
P_{\hd(e)} = P_{\hd(e')} \qquad P_{\tl(e)} = P_{\tl(e')}
\end{equation}
The first equation reduces to $P_c = P_d$, so input isomorphic nodes
must have the same state space. However, the second equation can impose
further equalities. We say that $i,j$ are {\em state-equivalent}, written
$i\sim_S j$, if the above equations, taken over all $c,d$, require $P_i = P_j$.
This is the transitive closure of the relation on $\CC$ defined by~\eqref{E:state_equiv}.
It resolves an ambiguity in the usual concept of cell equivalence by
distinguishing between having the same node space, and having the same node dynamic.
It also extends the possible types of network without changing any of the basic 
theorems or proofs \cite{GS22}.

For any tuple of nodes $\mathbf{c} = (c_1, \ldots, c_m)$ we write
\beqn
P_{\mathbf{c}} &=& P_{c_1} \times \cdots \times P_{c_m} \\
x_{\mathbf{c}} &=& (x_{c_1}, \ldots, x_{c_m})
\eeqn

The input set of node $c$ defines an {\em input tuple} of nodes
$I(c) =(c,\tl(i_1),\cdots,\tl(i_\nu))$ 
where the $i_j$ are the arrows satisfying $\HH(i_j)=c$. 
For brevity we follow~\cite{SP07, SP08} and define the $\nu$-tuple of tail nodes 
$x_{T(c)}$ and the space $P_{T(c)}$ by:
\[
x_{T(c)} = x_{\tl(I(c))} = (x_{\TT(i_1)}, \ldots, x_{\TT(i_\nu)} )\in P_{\TT(I(c))} = P_{T(c)}
\]

\begin{definition}\em
\label{D:admiss}
Let $\GG$ be a network. 
A map $f: P \rightarrow P$ is $\GG$-{\em admissible} if:

(a)	 {\em Node Compatibility:} The node state spaces satisfy $P_c=P_d$ whenever $c\sim_S d$.

(b) {\em Domain Condition}: For  every node $c$, there exists a function 
$\hat{f}_c: P_c \times P_{T(c)} \rightarrow P_c$ such that
\[
f_c(x) \equiv \hat{f}_c(x_c, x_{T(c)})
\]
In particular, the domain of $\hat{f}_c$ (which, in effect, is the relevant 
domain of $f_c$) is $P_c \times P_{T(c)}$.

(c) {\em Pullback Condition}:  If nodes $c, d$ are input equivalent, then
for every $\beta \in B(c,d)$: 
\begin{equation}
\label{e:pullbackeq}
\hat{f}_d(x_d, x_{T(d)}) \equiv \hat{f}_c(x_d, \beta^\ast x_{T(d)}) 
\end{equation}
where the {\em pullback map} is defined by:
\begin{equation}  \label{e:pullback_long}
\beta^\ast x_{T(d)} = (x_{\tl(\beta(i_1))}, \ldots, x_{\tl(\beta(i_\nu)}) \in P_{T(c)}
\end{equation}
\qed\end{definition}

In particular, we can apply~\eqref{e:pullbackeq} when $c=d$. This shows that
\[
\hat{f}_c(x_c, x_{T(c)})\ \mbox{is}\ B(c,c)\mbox{-invariant}
\]
where the {\em vertex group} $B(c,c)$ acts trivially on the first coordinate
$x_c$ and permutes the coordinates of $ x_{T(c)}$ according to the
pullback maps~\eqref{e:pullback_long}. 
That is, the action of $\beta$ is:
\begin{equation}
\label{E:beta_action}
(x_c, x_{T(c)}) \mapsto (x_c, \beta^\ast x_{T(c)})
\end{equation}
Triviality of this action on
the first coordinate, and the distinguished nature of that coordinate, are crucial to this paper.

\begin{remarks}\em
(a) The group $B(c,c)$ is finite and is a direct product of symmetric groups,
one for each arrow-type.

(b) From now on it is convenient to omit the hat on $\hat{f}_c$ and consider
$f_c$ as a map $f_c : P_c \times P_{T(c)} \to P_c$.
\qed\end{remarks}

\subsection{Alternative Characterisation of Admissibility}

The definition of pullback maps provides a `coordinate-free'
definition of admissible ODEs.
We now deduce a standard characterisation of admissible maps, based on a specific
choice of coordinates in the domains of component maps $f_c$,
which is more convenient for the purposes of this paper.

Choose an ordering on arrow types, so that arrows of a given type
occur is a block; then order arrows arbitrarily within each block.
Call this a {\em standard ordering} of arrows. It is easy to prove that
the groupoid $\BB$ is generated by all vertex symmetry groups
$B(c,c)$ together with a single {\em transitional map} $\beta_{cd}:I(c) \to I(d)$
for each $c \neq d$ with $c \sim_I d$. Moreover, if input variables for input equivalent nodes
are listed in standard order, the natural transitional map is the identity.
This is why the usual way to represent symmetries of components using an overline on the
relevant input variables is possible~\cite{GST05,SGP03}. 
The overlines correspond to the blocks
of arrows with a given arrow type, and substitution of corresponding
variables gives the identity transition map.

The group $B(c,c)$ acts on the input set
set $I(c)$ by permuting arrows and preserving arrow-type, so it preserves
blocks of arrows in standard order. We can now characterise admissible maps 
in terms of $B(c,c)$-invariance, avoiding explicit reference to pullback maps: 

\begin{proposition}
\label{P:char_admiss}
A map $f:P \to P$ is admissible if and only if, in standard order:

{\em (a)} $f_c$ is invariant under $B(c,c)$ for each $c$ in a set of
representatives of the input equivalence classes.

{\em (b)} $c \sim_I d \implies f_c = f_d$.
\end{proposition}
\begin{proof}
This follows from~\cite[Lemma 4.5 and Proposition 4.6]{SGP03}, 
with the extra observation that 
when the inputs are in standard order the $\beta_{cd}$
can be taken to be the identity.
\qed\end{proof}

Proposition~\ref{P:char_admiss} implies that
admissible maps can be constructed as follows.
Choose a set of representatives $\mathcal{S}$ for input equivalence. For each
$s \in \mathcal{S}$ let $p_s$ be any smooth $B(s,s)$-invariant map
$p_s : P_s \times P_{T(s)} \to P_s$.
The maps $p_s$ can be chosen independently for each $s$.
In standard order, for all $c \in \CC$, define
\[
p_c = p_s\ \mbox{where}\ s \sim_I c\ \mbox{and}\ s \in \mathcal{S}
\]
The resulting map $p$ is admissible because
\[
p_d(x_d, x_{T(d)}) = p_c(x_d, \beta^\ast x_{T(d)}) = p_c(x_d, x_{T(d)})
\]
when $\beta^\ast = \id$.

\subsection{Balanced Colourings}

A {\em colouring} of a network is a partition of the nodes into
disjoint subsets, the {\em parts}: 
\[
\CC = \CC_1 \ \dot{\cup}\ \CC_2\ \dot{\cup}\ \cdots\ \dot{\cup}\ \CC_m
\]
The {\em colour} $[c]$ of node $c$ is the unique $i$ such that $c \in \CC_i$.
A colouring can also be viewed as an equivalence relation `in same part'
or `same colour'. We pass without comment between these three interpretations,
but mainly refer to colourings. We use the same symbol $\bowtie$ for all three,
and often specify it as a partition.

Associated with any colouring $\bowtie$ is the {\em polydiagonal} (or {\em synchrony space})
\[
\Delta_{\bowtie} = \{x \in P:  x_c = x_d  \iff c \bowtie d \}
\]
The name indicates that this notion is a generalisation of the
usual diagonal subspace $\{(y,y,y, \ldots, y)\}$.
Another common term is {\em synchrony space}.
 
\begin{definition}\em
\label{D:balance}
A colouring is {\em balanced} if whenever $c$ and $d$ have the same colour,
there is a colour-preserving input isomorphism $\beta:I(c) \to I(d)$.
That is, $\tl(e)$ and $\beta(\tl(e))$ have the same colour for all arrows $e \in I(d)$.
Symbolically,
\[
c \bowtie d \implies \tl(e) \bowtie \beta(\tl(e))\quad \forall e \in I(c)
\]
\qed\end{definition}
This concept is central to network dynamics because $\Delta_{\bowtie}$
is {\em flow-invariant}, that is, invariant under {\em any} admissible map, 
if and only if $\bowtie$ is balanced: see~\cite[Theorem 4.3]{GST05} or 
\cite[Theorem 6.5]{SGP03}.
The space $\Delta_{\bowtie}$ is defined even when $\bowtie$ is unbalanced,
but is no longer flow-invariant.

Associated with any balanced colouring of $\GG$ is a {\em quotient network} 
$\GG/\!\!\bowtie$ whose
admissible maps are precisely the restrictions to $\Delta_{\bowtie}$ of the
admissible maps of the original network when $\Delta_{\bowtie}$ is canonically 
identified with $\prod_{s \in \mathcal{S}} P_s$ for a set of
representatives $\mathcal{S}$ of $\bowtie$ \cite[Section 5]{GST05}.
The validity of this theorem requires the multiarrow formalism introduced in that paper;
the differences that occur in the single-arrow formalism are described in~\cite{DS04}.

\subsection{Synchrony and Phase Relations: Sufficient Conditions}
\label{S:SPRSC}

The calculations that motivate the Rigid Synchrony and Rigid Phase Conjectures
combine the pullback condition~\eqref{e:pullbackeq} for admissibility with
equations~\eqref{E:synch_eq} and~\eqref{E:phase_eq}, as follows. From~\eqref{E:synch_eq} 
we obtain $\dot{x}_c(t) \equiv \dot{x}_d(t)$, so
\[
{f}_c(x_c, x_{T(c)}) \equiv {f}_d(x_d, x_{T(d)}) \equiv {f}_c(x_d, \beta^* x_{T(d)})
\]
Therefore a sufficient condition for synchrony of nodes $c,d$ is 
\begin{equation}
\label{E:RS_suff_cond}
x_{T(c)}(t) \equiv \beta^* x_{T(d)}(t) \quad \forall t \in \R\quad \mbox{whenever}\ x_c(t) \equiv x_d(t)
\end{equation}
The Rigid Synchrony Conjecture states that with the additional hypothesis of
rigidity, condition~\eqref{E:RS_suff_cond} is also necessary. This condition
is equivalent to the relation of synchrony being balanced.
Similar reasoning for~\eqref{E:phase_eq} leads to the sufficient condition
\begin{equation}
\label{E:RP_suff_cond}
x_{T(c)}(t) \equiv \beta^* x_{T(d)}(t+\theta) \quad \forall t \in \R\quad \mbox{whenever}\ x_c(t) \equiv x_d(t+\theta)
\end{equation}
The Rigid Phase Conjecture states that  condition~\eqref{E:RP_suff_cond}
is also necessary for~\eqref{E:phase_eq} to hold, with the additional hypothesis of
rigidity.

\subsection{$C^1$ Norm}
\label{S:C1norm}

We end this section by clarifying the sense in which a perturbation 
is to be considered `small', a technical point that we have hitherto slid over.
In order for hyperbolicity to imply the existence of a locally
unique perturbed periodic orbit, we use the $C^1$ topology.
It is also convenient to define this in a way that is tailored
to the network setting, with distinguished node spaces, as follows.

Choose a fixed state space $P = P_1 \times \cdots \times P_n$ where
$P_c = \R^{k_c}$ for finite $k_c$ and $1 \leq c \leq n$. Let
$C^1(P,P)$ be the Banach space of admissible $C^1$-bounded $C^1$ maps
$f : P \to P$ with the $C^1$ norm
\begin{equation}
\label{E:C1norm}
\| f\|_1 = \sup_{x \in P} (\|f(x)\|, \| \mathrm{D} f(x)\|)
\end{equation}
where $\mathrm{D}f$ is the derivative. 
In the context of this paper it is
convenient to define the norm on state space $P$ by
\begin{equation}
\label{E:C1normP}
\| (x_c)_{c \in \CC} \| = \max \{\| x_c \|_E : c \in \CC\}
\end{equation}
where $\|v\|_E$ is the Euclidean norm. 

By~Abraham \etal~\cite[Proposition 2.1.10 (ii)]{AMR83}, all norms on a
finite-dimensional real vector space are equivalent, so this
definition is equivalent to the usual $C^1$ norm.

\section{Quasi-Quotients}
\label{S:QQ}

Each induced ODE obtained in Section~\ref{S:3NE} can be
characterised as an admissible ODE for a smaller network
whose nodes correspond to the colours. If a colouring $\bowtie$ is balanced,
the smaller network is the usual quotient network. If $\bowtie$ is not balanced,
we can still construct a smaller network as a `quasi-quotient' $\GG^\RR$ for any set
of representatives $\RR$. Uniqueness now fails: different choices
of $\RR$ can give different quasi-quotients. Dynamics with synchrony
pattern $\bowtie$ projects to give dynamics on $\GG^\RR$,
but the converse fails because the discarded `constraint equations' 
need not be satisfied. For these reasons, quasi-quotients seem
not to have been considered previously.
However, they arise naturally from the methods of this paper,
and they have one very useful
property: all admissible maps for $\GG^\RR$ lift to
(that is, are induced from) admissible maps for $\GG$. 
This property allows us to construct $\GG^\RR$-admissible
perturbations using only the topology of $\GG^\RR$, and then
lifting them to admissible perturbations on $\GG$. 
We therefore develop the basic properties of quasi-quotients
required in later proofs.

\subsection{Definition of Quasi-Quotient}

Let $\GG$ be a network with nodes $\CC$, let $\bowtie$ be
a colouring of $\GG$ (which need not be balanced),
and choose a set $\RR$ of representatives for $\bowtie$. 

\begin{definition}\em               
\label{D:bracket}
If $c \in \CC$, write $[c]$ for the unique element of $\RR$ such that
$r \bowtie c$. 

In particular, $[r]= r$ if and only if $r \in \RR$.
\qed\end{definition}

\begin{definition}\em
\label{D:QQ}
The {\em quasi-quotient} network $\GG^\RR$ has nodes $r \in \RR$,
whose node type is the same as the node-type of $r$ in $\GG$.

The arrows $e$ for $\AA^\RR$ of $\GG^\RR$ are identified (via a bijection
$e \mapsto e'$) with the arrows 
\[
e' \in \bigcup_{r \in \RR} I(r)
\]
Under this identification, head and tail nodes in $\GG^\RR$ are defined by
\[
\hd^\RR(e) = \hd(e') \qquad \tl^\RR(e) = [\tl(e')]
\]
Arrows $e_1, e_2$ in $\AA^\RR$ have the same arrow type if and only if
$e_1', e_2'$ have the same arrow type in $\GG$.
\qed\end{definition}

Informally, we construct $\GG^\RR$ by taking all nodes in $\RR$,
together with their input arrows. Then the tail node of each input arrow
is found by replacing its tail node in $\GG$ by the unique node in $\RR$
that has the same colour.

For example, let $\GG$ be the 3-node network of Figure~\ref{F:3node_ring}.
The corresponding quasi-quotients for all nontrivial $(\bowtie,\RR)$
are shown in Figure~\ref{F:3node_ring_QQ}.

\begin{figure}[htb]
\centerline{
\includegraphics[width=.6\textwidth]{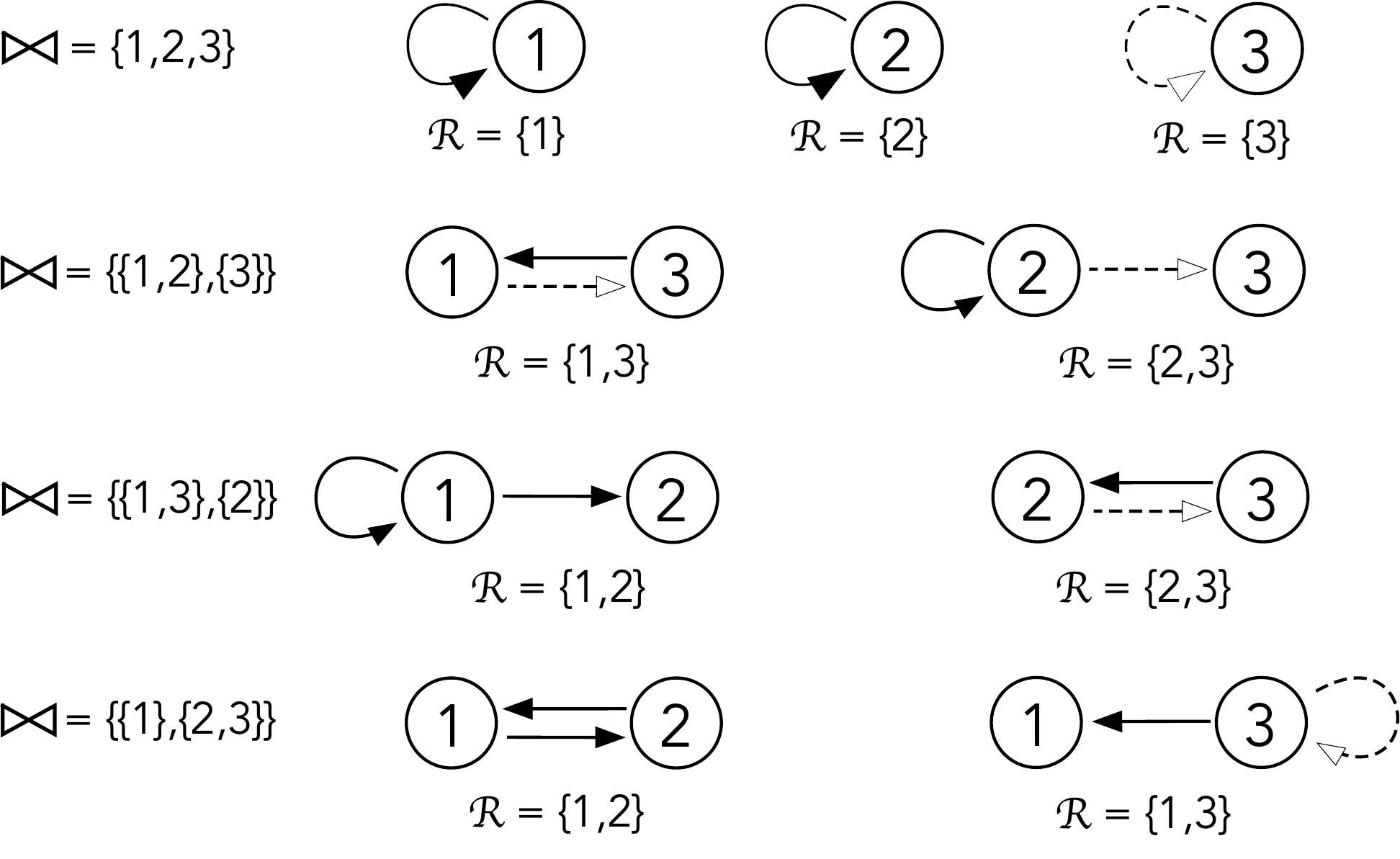}}
\caption{Quasi-quotients of the 3-node ring.}
\label{F:3node_ring_QQ}
\end{figure}

\subsection{Admissible Maps for Quasi-Quotients}

Given node spaces $P_c$ for $c \in\CC$, define the node space for
$r \in \RR$ to be $P_r$. The state space of $\GG^\RR$ is then
\[
P^\RR = \prod_{r\in\RR} P_r
\]
For any tuple $(c_1, \ldots, c_k)$ of nodes $c_i \in \CC$, define
the corresponding tuple of nodes of $\RR$ by:
\begin{equation}
\label{E:[c1...]}
[(c_1, \ldots, c_k)] = ([c_1], \ldots, [c_k])
\end{equation}

We now show that any $\GG$-admissible map on $P$ defines
a unique $\GG^\RR$-admissible map on $P^\RR$. Conversely,
every $\GG^\RR$-admissible map on $P^\RR$ lifts to
a $\GG$-admissible map on $P$, but this need not be unique.
The key observation is:

\begin{proposition}
\label{P:RRinput-type}

{\rm (a)} Any input isomorphism from $I(r)$ to $I(s)$ in $\GG^\RR$
identifies naturally with an input isomorphism from $I(r)$ to $I(s)$ in $\GG$.

{\rm (b)} Nodes $r,s \in \RR$ have the same input type in $\GG^\RR$
if and only if they have the same input type in $\GG$.
\end{proposition}
\begin{proof}
(a)
This is immediate from the definition of arrows and arrow-types in
Definition~\ref{D:QQ}. In detail: we have identified arrows in $\AA^\RR$
with arrows in $\AA$ via the bijection $e \mapsto e'$. The definition of head nodes
of these arrows implies that the input arrows of $r$ in $\GG^\RR$
correspond bijectively to the input arrows of $r$ in $\GG$, preserving arrow types.
The same goes for input arrows of $s$. Therefore any input isomorphism
$I(r) \to I(s)$ in $\GG^\RR$ corresponds to an input isomorphism
$I(r) \to I(s)$ in $\GG$, and conversely.

(b) This is now immediate.
\qed\end{proof}

\subsection{Properties of Quasi-Quotients}

Most features of the usual quotient network construction do not carry over
to quasi-quotients, but a few useful ones do. We exploit these features in the proof of
the main theorem, Theorem~\ref{T:localRSC}.

\begin{theorem}
\label{T:QQadmiss}

{\rm (a) Restriction:} Every $\GG$-admissible map $f:P \to P$ determines 
a $\GG^\RR$-admissible map $f^\RR:P^\RR \to P^\RR$ defined by
\begin{equation}
\label{E:RRrestriction}
f^\RR_r(x_r, x_{T(r)}) = f_r(x_r, x_{[T(r)]})
\end{equation}

{\rm (b) Lifting:} For every $\GG^\RR$-admissible map $g:P^\RR \to P^\RR$,
there exists a $\GG$-admissible map $f:P \to P$ such that $f^\RR = g$.

{\rm (c) Smallness of Lift:} If $\| g\|_1 < \eps$, we can define $f$ so 
that $\|f\|_1 < \eps$.
\end{theorem}

\begin{proof}

(a) 
By Definition~\ref{D:QQ}, the domain condition for $f^\RR$ is 
that there exists $\hat f^\RR: P_r \times P_{[T(r)]}$ such that
$f^\RR(x) = \hat f^\RR(x_r, x_{[T(r)]})$. This is consistent with~\eqref{E:RRrestriction}.

The identifications
in Definition~\ref{D:QQ} imply that if $r,s \in \RR$ and $\beta \in B(r,s)$ (for $\GG$)
then $\beta$ identifies with an input isomorphism in $\GG^\RR$,
which we also denote by $\beta$. The pullback condition for $f$ is
\[
f_d(x_d, x_{T(d)}) = f_c(x_d, \beta^\ast x_{T(d)})
\]
for all input isomorphisms $\beta:I(c) \to I(d)$. Therefore
\[
f^\RR_s(x_s, x_{[T(s]}) = f_s(x_s, x_{T(s)}) = f_r(x_r, \beta^*x_{T(r)})
= f^\RR_r(x_s, \beta^*x_{[T(s]})
\]
which is the pullback condition for $f^\RR$.

(b) Here it is convenient to use the alternative characterisation
of admissible maps in Proposition~\ref{P:char_admiss}.
Let $g:P^\RR \to P^\RR$ be a $\GG^\RR$-admissible map. We must
construct a $\GG$-admissible map $f:P \to P$ such that $f^\RR = g$.
The nodes $\CC$ split into two disjoint subsets $\CC_1, \CC_2$
defined by
\[
\CC_1 = \{c \in \CC: \exists\, r \in \RR, c \sim_I r\}
\qquad \CC_2 = \CC\setminus\CC_1
\]
With arrows in standard order, define $f$ by:

\begin{equation}
\label{E:liftQQ}
f_c(x_c,x_{T(c)}) = \left\{\begin{array}{lcl}
f_{[c]}(x_{[c]},x_{[T(c)]})& \mbox{if} & c \in \CC_1 \\
0 & \mbox{if} & c \in \CC_2
\end{array}\right.
\end{equation}
Clearly $f$ satisfies the required domain conditions to be $\GG$-admissible.
To verify the pullback conditions, we must show that
$f_c$ is $B(c,c)$-invariant on $P_{T(c)}$. The rest follows because
transition maps are now the identity.
Since $c \in \CC_1$, invariance under $B(c,c)$ follows from Proposition~\ref{P:RRinput-type}
and $B([c],[c])$-invariance of $f_{[c]}$ on $P_{[T(c)]}$.

(c) In~\eqref{E:liftQQ}, we have  $f_{[c]}(x_{[c]},x_{[T(c)]}) = 0$ when $c \in \CC_2$.
With the $C^1$ norm defined as in Definition~\ref{E:C1normP},
 $\| g\|_1 < \eps$ implies that $\|f\|_1 < \eps$.
\qed\end{proof}

\begin{remark}\em
For (b), the choice of $f_c$ for $c \in \CC_2$ can be replaced by the corresponding
components $g_c$ of any $\GG$-admissible map $g$, by Proposition~\ref{P:char_admiss}. 
The choice on $\CC_1$ is unique.

In the proof of  the key Lemma~\ref{L:p_exists} below, we 
make the lift have small compact support, which implies that it is $C^1$-bounded,
but we do not want it to vanish identically for $c \in \CC_2$. This can be done
by making $g$ have small compact support but not requiring $g \equiv 0$.
\qed\end{remark}

\subsection{Induced ODE}
\label{S:IODE}

Associated with any quasi-quotient is a version of the 
usual restricted ODE for a quotient network. Because $\Delta_{\bowtie}$
need not be flow-invariant, the domain of the ODE is restricted to
$P^\RR$ and its codomain is projected onto $P^\RR$:

\begin{definition}\em
\label{D:ind+con}
Consider an admissible ODE~\eqref{E:admiss_ODE} and write it in components as
\begin{equation}
\label{E:admiss_ODE_comp}
\dot{x}_c = f_c(x_c, x_{T(c)}) \quad x \in P,\ 1 \leq c \leq n
\end{equation}
The {\em induced ODE} for the pair $(\bowtie, \RR)$ is the ODE
\begin{equation}
\label{E:admiss_ODE_induced}
\dot{x}_r = f_r(x_r,x_{[T(r)]}) \quad x \in P,\ r \in \RR
\end{equation}
where $[T(r)]$ is defined by~\eqref{E:[c1...]}.
The {\em constraint equations} are the corresponding equations on the other components,
implied by the synchrony relations $x_c(t) \equiv x_{[c]}(t)$:
\begin{equation}
\label{E:Delta_constraints}
\dot{x}_{[c]} = f_c(x_{[c]},x_{[T(c)]}) \quad x \in P, c \in \CC\setminus \RR
\end{equation}
\qed\end{definition}

\begin{example}\em
\label{ex:3nodeQQ_induced}

Again consider the 3-node network of Figure~\ref{F:3node_ring},
with admissible ODEs \eqref{e:3nodeODE}. 
Consider Case (B) of Section~\ref{S:CSP} with colouring
$\bowtie = \{\{1,2\},\{3\}\}$. There are two choices of $\RR$.

If $\RR = \{1,3\}$, the induced ODE is 
\beqn
\dot{x}_1 &=& f(x_1, x_3) \\
\dot{x}_3 &=& g(x_3, x_1)  
\eeqn
with constraint
\[
\dot{x}_1 = f(x_1, x_1)  
\]

If $\RR = \{2,3\}$, the induced ODE is 
\beqn
\label{E:3nodeODEind_1}
\dot{x}_2 &=& f(x_2, x_2)  \\
\dot{x}_3 &=& g(x_3, x_2)  
\eeqn
with constraint
\[
\dot{x}_2 = f(x_2, x_3) 
\]

The induced ODEs are the admissible ODEs for the
corresponding quasi-quotients in Figure~\ref{F:3node_ring_QQ},
in accordance with Theorem~\ref{T:QQadmiss}.
The same remark holds for all other choices of $(\bowtie,\RR)$
in the figure.
\qed\end{example}

The induced ODE depends on the choice of $\bowtie$ and $\RR$,
and in general solutions need not lift back to~\eqref{E:admiss_ODE_comp}.
More precisely:

\begin{theorem}
\label{T:induced_ODE}
\begin{itemize}
\item[\rm (a)]
Any orbit $\X = \{x(t)\}$ of \eqref{E:admiss_ODE_comp} defines
an orbit $\X^\RR = \{x^\RR(t)\}$ of \eqref{E:admiss_ODE_induced}.
\item[\rm (b)]
The orbit $x(t) \in \Delta_{\bowtie}$ if and only if it satisfies the
additional constraints \eqref{E:Delta_constraints}.
\item[\rm (c)]
Solutions of \eqref{E:admiss_ODE_induced} 
lift uniquely to a solution of \eqref{E:admiss_ODE_comp},
provided that such a solution satisfies the constraints.\qed
\end{itemize}
\end{theorem}
The proof is obvious.
We emphasise that (c) requires the constraint equations to
be satisified as well as the induced ODE.
This can, for example, be implied by rigidity. 
Our aim in this paper is to obtain a contradiction
to this property in suitable circumstances.

When $\bowtie$ is not balanced, then for any choice of $\RR$ the
constraints include at least one component that differs formally
from the corresponding component of the induced equation.
We exploit this formal difference to obtain a contradiction to rigidity.

When $\bowtie$ is balanced, the constraints just repeat the corresponding
components of the induced ODE , and this is the same as the usual restricted ODE.
In this case, no contradiction occurs.

\subsection{Perturbations}

Suppose that an ODE $\dot y = g(y)$
has a non-hyperbolic periodic orbit $\Y$. If $g$ is perturbed to 
a nearby map $\tilde g$, there may be no periodic orbits near $\Y$,
or more than one. Thus we cannot talk of `the' perturbed
periodic orbit $\tilde \Y$.

If $f$ is $\GG$-admissible, with a hyperbolic periodic orbit $\X$,
and we consider the induced ODE $\dot y = g(y)$ for $\GG^\RR$,
the same remark applies to $\Y = \X^\RR$, because $\Y$
need not be hyperbolic. However,

\begin{lemma}
\label{L:canonical_ppo}
If $\X$ is hyperbolic with rigid synchrony pattern
$\bowtie$, and $\RR$ is a set of representatives, then for any
small $\GG^\RR$-admissible perturbation $\tilde f^\RR$ = of
$f^\RR$ the induced orbit $\X^\RR$ has a
uniquely defined canonical perturbed periodic orbit $\tilde \X^\RR$.  
\end{lemma}
\begin{proof}
Let $p^\RR = \tilde f^\RR-f^\RR$ where $\|p^\RR\|_1$ is small. By (c) we can lift
$p^\RR$ to a $\GG$-admissible map $p$ whose norm is equally small.
If $\tilde f = f + p$ then $\tilde f^\RR= f^\RR + p^\RR = \tilde f^\RR$.
Let $\tilde \X$ be the unique perturbed periodic orbit near $\X$
for $\GG$. By rigidity, $\tilde\X^\RR$ is a periodic orbit of $f^\RR$,
and is near $\X^\RR$. This procedure defines $\tilde\X^\RR$
uniquely.
\qed\end{proof}

\begin{corollary}
\label{C:properties}
If all periodic orbits of $\tilde f^\RR$ near $\X^\RR$ are hyperbolic,
then $\tilde \X^\RR$ is hyperbolic.
\qed\end{corollary}

\section{Properties Related to Hyperbolicity}
\label{S:PRH}

In this section we recall background results and concepts that are needed
for the proofs of the main theorems, and provide rigorous definitions
for concepts that until now have been treated informally for illustrative purposes.

\subsection{$C^1$-Bounded Maps}
\label{S:C1BM}

In order for locally unique perturbed periodic orbits to exist,
we must work with perturbations $p$ for which the $C^1$-norm is bounded:
\[
\| p\|_1 = \sup_{x\in P}\max(\|p(x)\|, \|\mathrm{D} p(x)\|) < \infty
\]
where $\mathrm{D}$ is the derivative. These maps form a Banach space.

This condition can always be arranged using a bump function, 
Abraham \etal~\cite[Lemma 4.2.13]{AMR83}, to modify any admissible map $f$ 
so that it vanishes outside some large compact set $K$ that contains
$\X$, while leaving $\X$ unchanged and $f$ unchanged  
in a neighbourhood of $\X$. However, we require
$C^1$-boundedness only for perturbations $p$ of $f$, not for $f$ itself,
and $p$ will always be defined in a manner that ensures it is bounded,
so this modification of $f$ is not required in this paper.

\subsection{Hyperbolic Periodic Orbits}
\label{S:HPO}

Hyperbolicity (for equilibria, periodic orbits, or more generally for invariant submanifolds)
is defined in many sources, for example
Abraham and Mardsen~\cite{AM78}, Arrowsmith and Place~\cite{AP90}, 
Hirsch and Smale~\cite{HS74}, Guckenheimer and Holmes~\cite{GH83},
and Katok and Hasselblatt~\cite{KH95}. An equilibrium $x^0$ of~\eqref{e:netwkODE}
is hyperbolic if and only if the derivative (Jacobian) $\mathrm{D}f|_{x^0}$
has no eigenvalues on the imaginary axis (zero included). A periodic orbit
is hyperbolic if its linearised Poincar\'e return map, for some (hence any) Poincar\'e section,
has no eigenvalues on the unit circle. That is, exactly one Floquet multiplier (equal to $1$)
lies on the unit circle; equivalently, exactly one Floquet exponent (equal to $0$)
lies on the imaginary axis~\cite{HKW81}.

The following result is standard, and can be proved by applying the 
Implicit Function Theorem to a Poincar\'e map. A more general proof for invariant
submanifolds can be found in Hirsch \etal~\cite[Theorem 4.1(f)]{HPS77}.
For the purposes of this paper it is convenient to
state it for 1-parameter families of perturbations.

\begin{lemma}
\label{L:hyperbolic}
Let $\X = \{x(t)\}$ be a hyperbolic periodic orbit of a smooth ODE 
$\dot{x} = f(x)$ on $\R^n$. Let $f+\eps p$ be any 1-parameter family of perturbations,
with $\|p\|_1$ bounded.
Then for $\eps \ll 1$ there exists, near $x(t)$, a locally unique
periodic orbit $\X^\eps = \{x^\eps(t)\}$ of the perturbed ODE $\dot{x} = f(x)+\eps p(x)$.
\qed \end{lemma}
The perturbed periodic orbit is locally unique in the sense
that, for any 1-parameter family of sufficiently small perturbations,
there is a locally unique path of periodic orbits that includes the unperturbed one. 
Here and elsewhere, $\eps \ll 1$ means $\eps < \eps^*$ for some $\eps^* > 0$ with specified properties.

\subsection{Open Properties}

Again, 
choose a fixed state space $P = P_1 \times \cdots \times P_n$, and let
$C^1(P,P)$ be the Banach space of admissible $C^1$-bounded $C^1$ maps
$f : P \to P$ with the $C^1$ norm~\eqref{E:C1norm}.
Let $C^\infty(P,P)$ be the space of admissible $C^\infty$ maps
$f : P \to P$, which is not a Banach space. We use $\|\ \|_1$ to put a
topology on the space
\[
\FF(P) = C^1(P,P) \cap C^\infty(P,P)
\]
This topology is applied only to `small perturbations' $p$ of $C^\infty$ maps $f$,
because smooth maps need not be $C^1$-bounded. 
We use the same notation when all maps are required to be $\GG$-admissible for
a network $\GG$, indicating this condition by context.

If we focus only on
a suitable compact subset $K$ of state space, we can replace any admissible map
$f$ by a bounded one that agrees with $f$ on $K$, see
see Section~\ref{S:C1BM}. However, the space of admissible $C^1$-bounded
$C^\infty$ maps $f : P \to P$ is still not a Banach space. Nevertheless,
we can state:

\begin{definition}\em
\label{D:open}
A property $\OO$ of maps $f \in C^\infty(P,P)$ is {\em open} 
if whenever $f$ has property $\OO$,
there exists $\eps > 0$ such that, for all $q \in \FF(P)$ with $\|q\|_1 < \eps$,
the map $f+q$ has property $\OO$. Equivalently, the set of all $f$ with property
$\OO$ is open in the $C^1$ norm, and is preserved by all
$C^1$-small perturbations of $f$.

We use the same terminology for network dynamics,
requiring the maps involved to be admissible.
\qed\end{definition}

In the sequel we use a $1$-parameter family of perturbations 
$q = \eps p$ for a fixed $p \in \FF(P)$, that is, we consider
 the family $f+\eps p$ for $\eps \ll 1$.
 We use the weaker condition
that $\OO$ holds for all $p$ and all $\eps \ll 1$. That is, we do not
require the upper bound on $\eps$ to be uniform in $p$.

\subsection{Rigidity}

We extend Definition~\ref{D:open} to properties of a hyperbolic periodic orbit 
of an admissible ODE, replacing `open' by `rigid' to preserve traditional terminology. 
Restating~\eqref{e:netwkODE} for convenience, let the admissible ODE be
\begin{equation}
\label{E:admiss_ODE}
\dot{x} = f(x) \quad x \in P
\end{equation}
and let $\X = \{x(t) : t \in \R\}$ be a hyperbolic periodic orbit
with period $T$. By Lemma~\ref{L:hyperbolic}, 
if $p \in \FF(P)$ is any admissible map and
$\eps$ is sufficiently small then the perturbed ODE
\begin{equation}
\label{E:admiss_ODE_pert}
\dot{x^\eps} = f(x^\eps) + \eps p(x^\eps) \quad x^\eps\in P
\end{equation}
has a {\em unique} perturbed periodic orbit $\X^\eps = \{x^\eps(t) : t \in \R\}$
that is near $\X$ in the Hausdorff metric for the $C^1$ topology,
with period $T^\eps$ near $T$. This equation becomes~\eqref{E:admiss_ODE}
when $\eps = 0$. In particular, $\X^0 = \X$ and $T^0 = T$.

As is customary, we use the same symbol $x^\eps$ to
denote an arbitrary variable in $P$ and a specific solution (orbit, trajectory) of the ODE.
The alternative is to introduce cumbersome notation to distinguish the two meanings.

Using a fixed Poincar\'e section $\Sigma$ to $\X$ to define the initial condition
by $x^\eps(t_0) \in \Sigma$, and considering the Poincar\'e map and hyperbolicity,
we can assume that for $\eps \geq 0$ and any $t\in\R$ the point $x^\eps(t)$ varies smoothly with
$\eps$, and so does $T^\eps$.

\begin{definition}\em
\label{D:rigid_property}
A property $\OO$ of $\X$ relative to $T$ is {\em rigid} if $\X^\eps$
has property $\OO$ relative to $T^\eps$ for $\eps \ll 1$.
\qed\end{definition}

Hyperbolicity of a given $\X$ for $f$ is an open property of $f$,
and also a rigid property of $\X$. Rigid synchrony and phase patterns
of $\X$ are obviously rigid properties of $\X$. So are local
rigid synchrony and phase patterns, defined in Section~\ref{S:intro}.

If $\X$ has a balanced synchrony pattern (or local synchrony pattern)
$\bowtie$, the corresponding periodic orbit on the quotient $\GG/\!\!\bowtie$
is also hyperbolic, because $\Delta_{\bowtie}$ is flow-invariant.
However, this result no longer applies if $\bowtie$ is not balanced.
This is why we require $\X$ to be stably isolated or strongly hyperbolic,
Sections~\ref{S:SI} and \ref{S:SH}.

\subsection{Failure of Hyperbolicity}
\label{S:FH}

There is one class of networks for which we do not expect an analogue of the Kupka-Smale
Theorem to hold, for feedforward reasons. To discuss it, we use
standard ideas from Floquet theory~\cite{HKW81}.

Every network decomposes into 
{\em transitive components}~\cite{GGPSW19, wiki_strongly}, and the set $\TT$ of transitive components has a
natural partial ordering induced by directed paths. Dynamically, this ordering
gives admissible ODEs a feedforward structure. If $\TT$ has more than one maximal
element, and the periodic state $\X$ oscillates on at least two maximal components,
this state is not hyperbolic. This follows because the Floquet operator has block-triangular
form induced by the partial ordering. 

\begin{figure}[htb]
\centerline{
\includegraphics[width=0.15\textwidth]{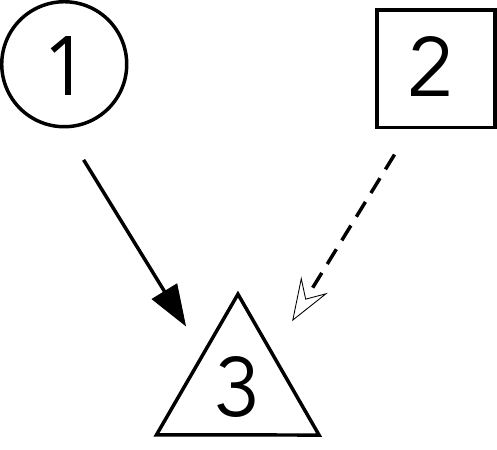}}
\caption{Connected network with two distinct maximal transitive components.}
\label{F:2max_tc}
\end{figure}

In more detail, Jos\'ic and T\"or\"ok~\cite[Remark 1]{JT06} observe that 
for the network of Figure~\ref{F:2max_tc},  periodic orbits cannot be hyperbolic 
unless node 1 or node 2 is steady. 
To prove this, observe that  admissible ODEs
for this network have the form
\begin{equation}
\label{E:no_hyp}
\begin{array}{rcl}
\dot x_1 &=& f(x_1) \\
\dot x_2 &=& g(x_2) \\
\dot x_3 &=& h(x_1,x_2,x_3)
\end{array}
\end{equation}
We must show that if such an ODE has a hyperbolic periodic orbit 
\[
\X = \{(x_1(t),x_2(t),x_3(t)): t \in \R\}
\]
then either $x_1(t)$ or $x_2(t)$ is steady. This follows because,
setting $F = (f,g,h)$, 
the Floquet multipliers of are the eigenvalues of
$\mathrm{D}_{x(t)}F$, which is lower triangular. Therefore the evolution operator 
is also lower triangular, and has two eigenvalues equal to 1. 
These eigenvalues correspond to the two diagonal blocks describing 
the evolution in the spaces of the variables $x_1$ and $x_2$.
(A periodic orbit always has a Floquet multiplier 1 for 
an eigenvector along the orbit: see~\cite[Chapter 1 Note 5]{HKW81}.)

As Jos\'ic and T\"or\"ok observe, 
similar remarks apply if we replace nodes 1 and 2 by two disjoint 
transitive components that are maximal in the partial ordering:
that is, they force (feed forward into) 
the rest of the network, which replaces node $3$. In~\eqref{E:no_hyp} let $x_1, x_2, x_3$ be
coordinates on, respectively, the two maximal components and 
the rest of the network. The same argument then applies.

\subsection{Implications for Quasi-Quotients}
\label{S:IQQ}

In the present context, 
such networks cannot occur as the overall network $\GG$, since we
assume $\X$ hyperbolic. However, we also require hyperbolicity
(or a similar property) for certain quasi-quotients $\GG^\RR$.
A `bad' choice of representatives
can create more than one maximal transitive component in $\GG^\RR$.
Figure~\ref{F:2col_reg} (left) shows the simplest (connected) example of this kind.
The `good' choice $\RR = \{1,3\}$ yields an induced network
with feedforward structure and a single maximal component,
Figure~\ref{F:2col_reg} (middle).
In contrast, the `bad' choice $\RR = \{1,4\}$ yields an induced network
with two disconnected maximal components, Figure~\ref{F:2col_reg} (right).
If nodes 1 and 3 force the same component of some larger network,
the corresponding quasi-quotient has no hyperbolic periodic orbits.

\begin{figure}[htb]
\centerline{
\includegraphics[width=0.75\textwidth]{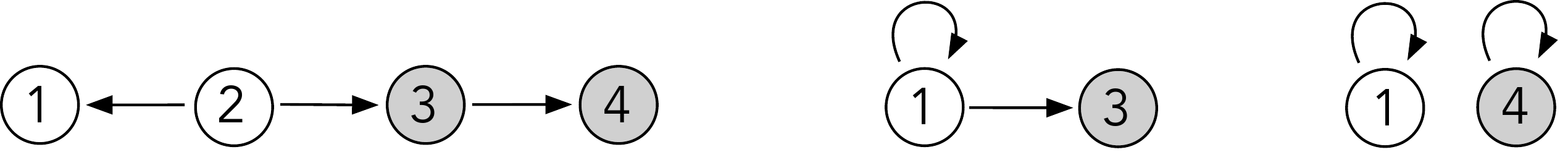}}
\caption{{\em Left}: $2$-coloured regular network. {\em Right}: Two distinct induced networks,
one `good', one `bad'.}
\label{F:2col_reg}
\end{figure}

We do not know whether there always exists such an $\RR$
when $\GG$ has only one maximal component.

\subsection{Strong Hyperbolicity}
\label{S:SH}

This paper relies on the following concept:

\begin{definition}\em
\label{D:strongly_hyperbolic}
A  periodic orbit $\X$ is {\em strongly hyperbolic}
for a colouring $\bowtie$ if $\X$ is hyperbolic and
there exists a set of representatives $\RR$ such that, if necessary after an
arbitrarily small perturbation, the
induced orbit $\X^\RR$ is a hyperbolic periodic orbit
of the induced ODE.

The orbit $\X$ is {\em strongly hyperbolic}
if it is strongly hyperbolic for every colouring $\bowtie$.
\qed\end{definition}

Suppose that $\X$ is hyperbolic, and perturb $f$ to a nearby admissible map $\tilde f$.
There is a locally unique perturbed periodic orbit $\tilde \X$ near $\X$.
Lemma~\ref{L:canonical_ppo} implies that there is a unique 
canonical choice for `the' perturbed induced periodic orbit $\tilde \X^\RR$,
which is near $\X^\RR$.
Thus when $\X$ has a rigid synchrony pattern
we can keep track of each $\X^\RR$ in a meaningful manner
when $f$ is perturbed, even when $\X^\RR$ is not known to be hyperbolic
on $P^\RR$.

The condition of strong hyperbolicity lets us `pre-perturb' the admissible map
$f$ to $\tilde f$, thereby ensuring that without loss of generality 
the uniquely defined periodic orbit $\X^\RR$ is hyperbolic on $P^\RR$.
This step is crucial for the proof of the Local Rigid Synchrony Property.

\subsection{Local Kupka-Smale Theorem}

In Section~\ref{S:RC2C} we infer strong hyperbolicity from
a local version of the Kupka-Smale Theorem:

\begin{lemma}
\label{L:LKS}
Let $f$ be any smooth vector field on $R^k$, with a periodic orbit $\X$ of period $T$.
Then 

{\rm (a)} For all sufficiently small $\delta_1 > 0$ there exists $\delta_1$ with $0 < \delta_2 < \delta_1$,
such that for all $t$ with $0 \leq t \leq 2T$ we have
\[
\psi^t(N_{\delta_2}(\X)) \subseteq N_{\delta_1}(\X)
\]
where $\psi^t$ is the flow of $f$ and $N_\delta$ is the tubular
neighbourhood of $\X$ of radius $\delta$.

{\rm (b)} There exists a smooth perturbation $\tilde f$ of $f$ and
$\delta_3, \delta_4 > 0$ with $\delta_3 \leq \delta_1$ and
$\delta_4 \leq \delta_2$
such that
\[
\tilde\psi^t(N_{\delta_4}(\X)) \subseteq N_{\delta_3}(\X)
\]
where $\tilde\psi^t$ is the flow of $f$,
and every periodic orbit of $\tilde f$ inside $N_{\delta_4}(\X)$ is hyperbolic.
\end{lemma}
\begin{proof}
This is a restatement of Lemma 3 of Peixoto~\cite{P66}. The bound $2T$ on $t$ is introduced
because the perturbed period $\tilde T$ may be (slightly) larger than $T$.
\qed\end{proof}

\begin{definition}\em
\label{D:GLKS}
Let $\GG$ be a network and let $\X$ be a hyperbolic periodic orbit of an admissible
vector field $f$.
If for every colouring $\bowtie$ there exists a set of representatives
$\RR$ such that statement (b) holds for all $\X^\RR$, we say that
$\GG$ is {\em locally Kupka-Smale} for $\X$.
\qed\end{definition}

\begin{lemma}
\label{LindiucedLKS}
If $\GG$ is locally Kupka-Smale for $\X$, then for any $\bowtie$ there
is a choice of $\RR$ such that the 
canonical induced periodic orbit $\X^\RR$ can be made
hyperbolic by an arbitrarily small admissible perturbation.
\end{lemma}
\begin{proof}
The perturbed induced periodic orbit $\tilde\X^\RR$ lies inside $N_{\delta_4}(\X)$.

\qed\end{proof}

\subsection{Stable Isolation}
\label{S:SI}

A weaker rigid property is:

\begin{definition}\em
\label{D:stable_isolation}
A hyperbolic periodic orbit $\X$ is {\em stably isolated}
for $\bowtie$ if there exists $\RR$ such that, if necessary after an
arbitrarily small perturbation:

\noindent
(a) The induced orbit $\X^\RR$ is an isolated periodic orbit
of the induced ODE.

\noindent
(b) This property is preserved by any further sufficiently small perturbation.

The orbit $\X$ is {\em stably isolated}
if it is stably isolated for every colouring $\bowtie$.
\qed\end{definition}

\begin{remarks}\em
(a) Condition (b) ensures that `having a stably isolated periodic orbit' 
is an open property of $f$.
Here `stable' refers to structural stability, not stability to 
perturbations of initial conditions such as asymptotic or linear stability.

(b) If $\bowtie$ is balanced, any hyperbolic $\X$ is stably isolated,
because $\Delta_{\bowtie}$ is flow-invariant,
so $\X^\RR$ is hyperbolic.
\qed\end{remarks}

\begin{proposition}
\label{P:SH->SI}
If $\X$ is strongly hyperbolic for $\bowtie$ then it is stably isolated for $\bowtie$.

If $\X$ is strongly hyperbolic then it is stably isolated.
\end{proposition}
\begin{proof}
Hyperbolic periodic orbits are isolated by the uniqueness assertion in Lemma~\ref{L:hyperbolic}.
\qed\end{proof}

\subsection{Kupka-Smale Networks}
\label{S:KSN}

There is a connection between strong hyperbolicity and the Kupka-Smale Theorem,
requiring only analogues of (a) and (b) in Theorem~\ref{T:KupkaSmale}. This 
motivates:

\begin{definition}\em
\label{D:NKS}
A network $\GG$ is a {\em Kupka-Smale} network if properties (a) and (b)
in Theorem~\ref{T:KupkaSmale} hold for a generic set of admissible vector fields.
\qed\end{definition}

\begin{theorem}
\label{T:SH_implications}
Let $\X$ be a hyperbolic periodic orbit and let $\bowtie$
be a colouring.

{\rm (a)} If $\GG^\RR$ is a Kupka-Smale network for some set of representatives $\RR$
for $\bowtie$ then
$\X$ is stably isolated for $\bowtie$.

{\rm (b)} If $\GG^\RR$ is locally Kupka-Smale near $\X^\RR$ for some $\RR$, then
$\X$ is stably isolated for $\bowtie$.

{\rm (c)} If $\X$ is stably isolated for $\bowtie$ then it is strongly hyperbolic
for $\bowtie$.

{\rm (d)} If $\X$ is stably isolated for $\bowtie$ then $\X$ is not a
limit of a continuum of periodic orbits, whose members do not equal $\X$ on
some open neighbourhood.

{\rm (e)} The implications {\rm (a)--(d)} hold when `for $\bowtie$' is deleted from all of them.
\end{theorem}

\begin{proof}
All implications are trivial.
\qed\end{proof}

In the remainder of the paper we prove that when the local synchrony
pattern $\bowtie$ of $\X$ is rigid,
the property stated in (d)
implies that $\bowtie$ is balanced. That is, $\GG$ has the Rigid Synchrony Property for $\X$.
The implications (a)--(d) provide alternative strategies for proving the 
Local Rigid Synchrony Property in specific cases.

\section{Local Rigidity}
\label{S:LR}

In this section we compare and contrast local and global versions of rigidity.

\subsection{Synchrony Patterns on Subsets}
\label{S:SPS}

Let $\X = \{x(t): t \in \R \}$ as usual. Let $J \subseteq \R$ be any subset.
Define the {\em local synchrony pattern of $\X$ on $J$} to be
the colouring $\bowtie_J$ defined by
\begin{equation}
\label{E:bowtieJ}
c \bowtie_J d \iff x_c(t) \equiv x_d(t) \ \forall t \in J
\end{equation}
If $J = \{x\}$ is a singleton, we write $\bowtie_x$ instead of $\bowtie_{\{x\}}$.

In this notation the {\em global synchrony pattern} of $\X$ is the colouring $\bowtie_\R$.
However, it is just as valid, and more intuitive to denote it by $\bowtie_\X$, which
we do from now on. The corresponding polydiagonal is accordingly denoted $\Delta_\X$.
This pattern is {\em globally rigid} if, for small enough perturbations,
the perturbed periodic orbit has the same global synchrony pattern.
The usual form of the Rigid Synchrony Conjecture (see~\cite[Section 10]{GS06}) 
states that any rigid global synchrony pattern $\bowtie_\X$ is balanced.

It is easy to establish an alternative characterisation of $\bowtie_\X$ in terms
of polydiagonals:
\begin{proposition}
\label{P:global_synch}
The global synchrony pattern of $\X$ is the unique colouring $\bowtie$ such that 

\noindent
{\em (a)} $\X \subseteq \Delta_{\bowtie}$

\noindent
{\em (b)} $\X$ is not contained in any polydiagonal strictly smaller than $\Delta_{\bowtie}$.
\qed\end{proposition}
Uniqueness follows from the lattice structure.

\subsection{Locally Rigid Synchrony}
\label{S:LRS}
For technical reasons, discussed in Section~\ref{S:CLS}, it is better to work with
a local version of rigidity. (The same point was made in~\cite{GRW10,GRW12}.)
We modify Definition~\ref{D:rigid_property}:

\begin{definition}\em
\label{D:_local_rigid_property}
A property $\PP$ of a hyperbolic periodic orbit
$\X$ relative to its period $T$ is {\em locally rigid near} $t_0 \in \R$ if 
\begin{itemize}
\item[\rm (a)]
There exists an open interval $J \subseteq \R$ with $t_0 \in J$ such that
$x^0(t)$ has property $\PP$ for all $t \in J$.
\item[\rm (b)]
There exists an open interval $J_1 \subseteq \R$ with $t_0 \in J_1$ such that
if $\eps$ is small enough, $x^\eps(t)$ has property $\PP$ (relative to $T^\eps)$
for all $t \in J_1$.\qed
\end{itemize}
\end{definition}

\begin{remark}\em
This $1$-parameter version is sufficient for the purposes of this paper. We do
not need the bound on $\eps$ to be uniform in $p$.
\qed\end{remark}

\subsection{The Lattice of Colourings}
\label{S:lattice}

In order to compare $x_c(t)$ with $x_d(t)$, they must both belong to the same space,
so we require $P_c = P_d$. Any colouring defined by synchrony
properties satisfies this condition, so from now on we do not refer to it explicitly.
Formally, for any colouring $\bowtie$, the condition $c \bowtie d$
requires $c$ and $d$ to be state equivalent, denoted by $c \sim_S d$.

Recall that every colouring $\bowtie$ defines a polydiagonal
\[
\Delta_{\bowtie} = \{x \in P: c  \bowtie d \implies x_c = x_d \}
\]
and it also defines a partition whose parts correspond to the colours.
There is a natural partial order $\preceq$ on colourings:
\begin{definition}\em
\label{D:finer}
A colouring $\bowtie_1$ is {\em finer than} a colouring $\bowtie_2$, written $\bowtie_1 \preceq\, \bowtie_2$, if
\[
c \bowtie_1 d \implies c \bowtie_2 d
\] 
Contrary to normal English, this includes the possibility that the colourings
are the same, up to a permutation of the colours.
We also say that $\bowtie_2$ is {\em coarser than} $\bowtie_1$. To remove
the possibility of equality we use the terms {\em strictly finer} and {\em strictly
coarser}.
\qed\end{definition}

The following proposition is obvious:
\begin{proposition}
\label{P:finer}
The following properties are equivalent:

\noindent {\em (a)} The colouring $\bowtie_1$ is finer than $\bowtie_2$.

\noindent {\em (b)} Every part of the partition defined by 
$\bowtie_1$ is contained in some part of the partition defined by $\bowtie_2$.

\noindent {\em (c)} $\Delta_{\bowtie_1} \supseteq \Delta_{\bowtie_2}$.
\qed\end{proposition}

Section 4 of~\cite{S07} proves, in slightly different terminology,
that with this partial ordering the set of all
colourings is a lattice in the sense of partially ordered sets, Davey and Priestley~\cite{DP90}.
Lemma 4.3 of that paper proves that this
lattice is dual to the lattice of polydiagonals under inclusion, property (c) of Proposition~\ref{P:finer}.
The balanced colourings form a sublattice, as do the balanced polydiagonals.
Since the number of polydiagonals is finite, these are finite lattices.

\subsection{Semicontinuity of Colourings}
The next proposition, which is well known and easy to prove,
states that sufficiently small changes to
$x \in P$ can make the colouring finer, but not strictly coarser.
(Recall that in Definition~\ref{D:finer} the terms `finer' and `coarser' permit equality.)
\begin{proposition}
\label{P:semicon}
If $x \in P$, with synchrony pattern $\bowtie_x$, and $y \in P$ with $\|y-x\|$ sufficiently small, then
$\bowtie_x$ is finer than $\bowtie_y$.
\end{proposition}

\begin{proof}
Suppose that $x_c \neq x_d$, and let $\|x_c-x_d\| = \delta > 0$. Suppose that 
$\|x-y\| \leq \delta/3$. Then
\beqn
\delta &=& \|x_c-x_d\| \leq \|x_c-y_c\|+\|y_c-y_d\| +\|y_d-x_d\| \\
	& \leq & \delta/3 + \|y_c-y_d\| + \delta/3
\eeqn
so $\|y_c-y_d\| \geq \delta/3 > 0$.
\qed\end{proof}

Thus one implication of local rigidity is that the synchrony pattern at $x(t_0) \in \X$ 
does not become coarser at sufficiently close points $\tilde{x}(t) \in \tilde\X$.

Intuitively, Proposition~\ref{P:semicon} states that small perturbations cannot create new equalities of coordinates,
but they might break up existing equalities. Equivalently, 
small perturbations cannot make the synchrony colouring coarser, 
but they might make it finer. Another equivalent condition is that
small perturbations cannot make the polydiagonal
$\Delta_{\bowtie_x}$ smaller, but they might make it larger.
Technically, the number of colours is {\em lower semicontinuous} with respect to
small perturbations, It\^o~\cite{I93}.

\subsection{Generic Points}

Given $\X$, the synchrony pattern $\bowtie_{x(t_0)}$ at a point
$x(t_0) \in \X$ is defined by
\[
c \bowtie_{x(t_0)} d \iff x_c(t_0) = x_d(t_0)
\]
In general this pattern may vary with $t_0$. By Proposition~\ref{P:semicon},
synchrony patterns are lower semicontinuous.
Intuitively, under sufficiently small perturbations colour
clusters can break up, but not merge.

In Section~\ref{S:SPS}
we defined the local synchrony pattern on an interval $J$.
Depending on $J$, this pattern might change as $t$ runs over $J$.
Semicontinuity implies that, by shrinking $J$ if necessary, we can assume that
\[
\bowtie_{x(t_0)} =\, \bowtie_{x(t_1)} \quad \forall t_0,t_1 \in J
\]
We assume this from now on.

When this condition holds, any point in $x(J)= \{x(t):t\in J\}$ is a {\em generic point}.
That is, it has the same synchrony pattern as the local pattern on $J$.
This result is a trivial consequence of Definition~\ref{D:LRSP}, but we state
it explicitly because it is central to the proof of the Rigid Synchrony 
Conjecture. It is valid for a local synchrony pattern, but need not be for 
the global synchrony pattern. Although the proof is trivial,
this observation is crucial to the construction of suitable perturbations,
and is why we consider locally rigid synchrony instead of globally rigid synchrony. 

\begin{proposition}
\label{P:separation}
Let $\bowtie_{x(t_0)}$ be the local rigid synchrony pattern for $\X$ at $x(t_0)$.
If $\RR$ is a set of representatives for $\bowtie_{x(t_0)}$ then
\[
i \neq j \in \RR \implies x_i(t_0) \neq x_j(t_0)
\]
\end{proposition}
\begin{proof}
If $x_i(t_0) = x_j(t_0)$ then $ i \bowtie_{x(t_0)} j$, so $i = j$ since $\RR$
is a set of representatives.
\qed\end{proof}

Local synchrony is technically more tractable than global synchrony 
for several reasons. One is that generic points need not exist for
global synchrony patterns: see Section~\ref{S:CLS}.
(Existence of a generic point is claimed in \cite{SP07}, but the proof there
is fallacious.) Another is that in the global case,
the perturbation technique requires admissible perturbations to vanish near
the whole of $\X$, but not at certain other points.
In certain circumstances this can conflict with admissibility.
Local perturbations --- those with small support --- avoid this problem.

Ironically, once we have used local rigid synchrony to prove
that (subject to strong hyperbolicity) a network has the 
Local Rigid Synchrony Property, we can deduce
the Global Rigid Synchrony Property and show that a rigid global
synchrony pattern is the same as any local one. But working from the beginning
with rigid global synchrony runs into the technical difficulties
mentioned above. See Section~\ref{S:GRS}.

\subsection{Changes in Local Synchrony}
\label{S:CLS}

In general dynamical system, different points on a periodic orbit can 
have different synchrony patterns. 
Moreover, a general dynamical system can be viewed
as a fully inhomogeneous network ODE with all-to-all coupling.
Therefore we must deal with the
possibility of a periodic orbit akin to Figure~\ref{F:Delta_changes}.

\begin{figure}[h!]
\centerline{
\includegraphics[width=.55\textwidth]{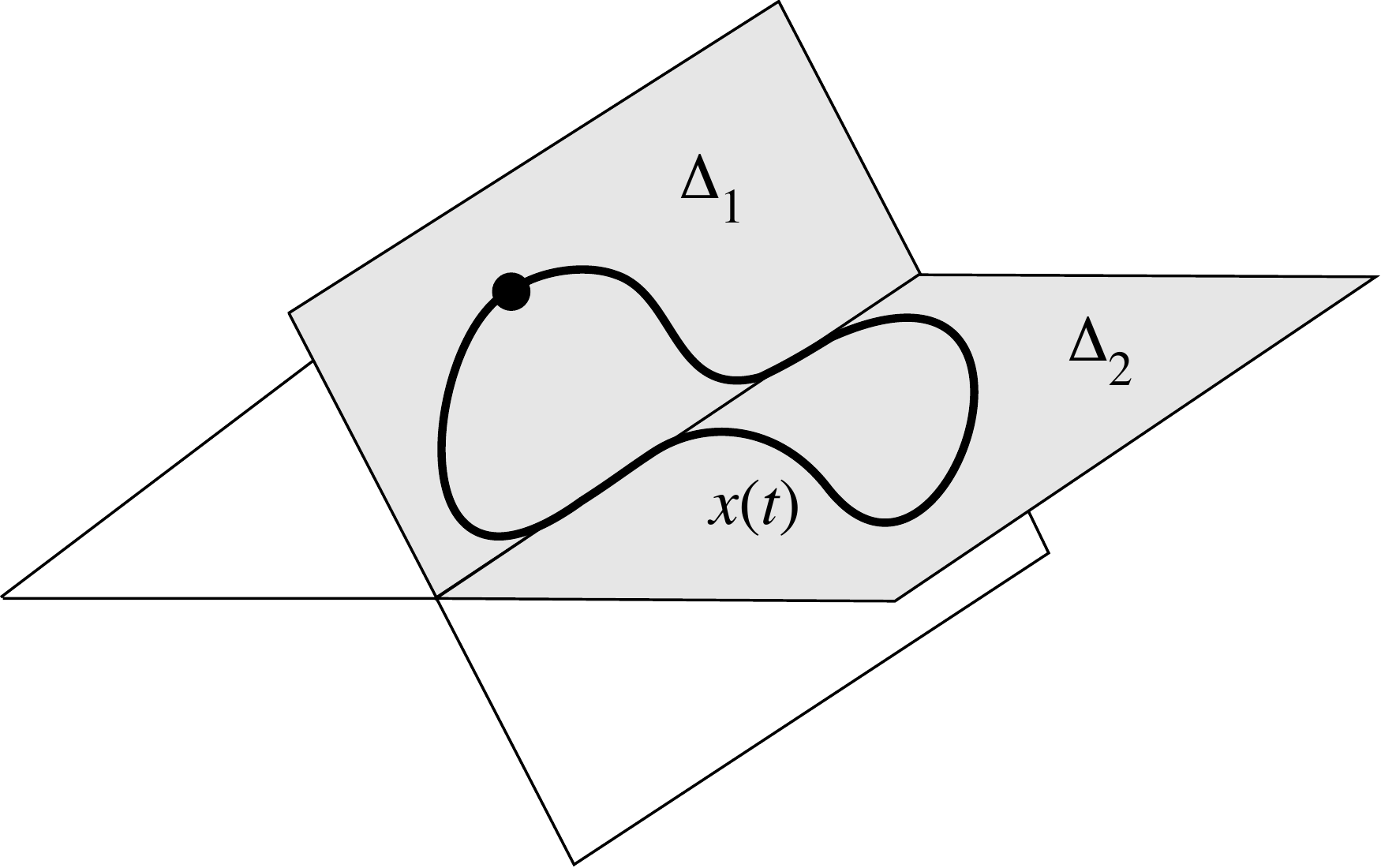}}
\caption{A periodic orbit with more than one local synchrony pattern,
which potentially might be rigid.}
\label{F:Delta_changes}
\end{figure}

Here $\Delta_1, \Delta_2$ are distinct polydiagonals.
The periodic orbit $\X =\{x(t)\}$ lies in $\Delta_1$ for some intervals of time $t$,
but in $\Delta_2$ for another interval. The transitions occur smoothly via intervals of time 
in which $x(t) \in \Delta_1 \cap \Delta_2$. Considerably more complicated changes in
the local synchrony pattern are possible; for example it might change infinitely many times.
If this or anything similar occurs, the  the polydiagonal $\Delta_\X$
for the global synchrony pattern $\bowtie_\X$ of $\X$ 
is different from both $\Delta_1$ and $\Delta_2$. Indeed,
\[
\Delta_1 \subsetneq \Delta_\X \qquad \Delta_2 \subsetneq \Delta_\X
\]
and no point or small interval $J$ on $\X$ has synchrony pattern $\bowtie_J\, =\, \bowtie_\X$,
the global synchrony pattern of $\X$.
We cannot infer local rigidity directly from global rigidity, because
perturbed orbits need not remain locally on either $\Delta_1$ or 
$\Delta_2$. All we know is that they stay inside the strictly larger polydiagonal
$\Delta_\X$. 

\subsection{Local Rigid Synchrony and Generic Points}
\label{S:LRSGP}

If $\X$ is as in Figure~\ref{F:Delta_changes}, there is no generic point for
the global synchrony pattern, because no point on the orbit 
lies in $\Delta_\X \setminus (\Delta_1 \cup \Delta_2)$.
However, there exist generic points for the local synchrony patterns.

Such behaviour is easily perturbed away in a general dynamical
system, but it it not clear whether this can always be done
for network dynamics. 
We get round this issue by considering a local version
of rigid synchrony. This property is local in time, but still requires a global condition:
existence of a hyperbolic periodic orbit. A local version of rigidity was introduced
by Golubitsky \etal~\cite[Section 2]{GRW10} for similar reasons;
see in particular their Definition 2.9(c).

The natural definition of the global pattern of rigid synchrony $\bowtie_\X$ on $\X$ is
to take $J = \R$ in~\eqref{E:bowtieJ}; which leads to
\[
c \bowtie_\X d \iff \tilde{x}_c(t) \equiv \tilde{x}_d(t) \quad \forall t \in \R
\]
where $\tilde{x}(t)$ is any perturbed periodic orbit for a sufficiently small
admissible perturbation. This colouring is the
same as $\approx^{\mathrm{rig}}$ in~\cite{SP07}.
 However, because of the possible non-existence
of generic points, we work locally:

\begin{definition}\em
\label{D:LRSP}
A colouring $\bowtie$ is a {\em locally rigid synchrony pattern} 
or {\em rigid local synchrony pattern} for $\X$
at $x(t_0)$ if and only if there exists $\eps > 0$ and $t^* > 0$ such that:

\noindent {\rm (a)}
$\X$ remains hyperbolic after a perturbation $p$ with $\|p\|_1 < \eps$.

\noindent {\rm (b)} The condition
\[
\tilde{x}^p_c(t) = \tilde{x}^p_d(t) \iff c \bowtie d
\]
is valid for all perturbations $p$ with $\|p\|_1 < \eps$ and all $t$ such that
$|t-t_0| < t^\ast$.
\qed\end{definition}

That is, the synchrony pattern for $\tilde{x}^p(t)$ is the same for all
small enough $p$ and for all $t$ close enough to $t_0$. 
We define this to be {\em the} local synchrony pattern of $\X$
near $t_0$, and denote it by
\[
\bowtie_{x(t_0)}
\]

\section{Construction of Admissible Perturbations}
\label{S:CAP}

We now set up a general method for defining admissible perturbations
with arbitrarily small support and arbitrarily small $C^1$ norm. We
describe the method for a general network $\GG$, but apply
it below to quasi-quotients $\GG^\RR$.

\subsection{Symmetrisation}
\label{S:symmetrisation}

We will use Proposition~\ref{P:char_admiss} to construct admissible pertubations,
supported on (usually small) sets, by defining suitable bump functions 
 and then symmetrising, as explained below
in Definition~\ref{D:symmetrisation}. Let
\beqn
g_c: P_c \times P_{T(c)} &\to& P_c \\
(x_c, x_{T(c)}) &\mapsto & g_c(x_c, x_{T(c)})
\eeqn
be a smooth admissible map. By~\eqref{E:beta_action} 
the vertex group $\Gamma = B(c,c)$ acts on such maps by
\[
\gamma g_c(x_c, x_{T(c)}) = g_c(x_c, \gamma^* x_{T(c)})
\]
where $\gamma^*$ is the pullback of $\gamma$. 

\begin{definition}\em
\label{D:symmetrisation}
The {\em symmetrisation} of $g_c$ is $g^\Gamma_c$, where
\[
g_c^\Gamma(x_c, x_{T(c)}) = \frac{1}{|\Gamma|}\sum_{\gamma \in \Gamma} (x_c, \gamma^*x_{T(c)})
\]
\qed\end{definition}
Symmetrisation on each input class, combined with standard-order identifications,
yields the groupoid symmetrisation construction of~\cite{SP07}
up to constant factors.
The normalisation factor $1/|\Gamma|$ ensures that
\begin{equation}
\label{E:symm_norm}
\| g_c^\Gamma \|_1 \leq \|g_c\|_1
\end{equation}
bearing in mind that $\|g_c^\gamma\|_1 = \|g_c\|_1$ for all $c \in \CC$
because $\Gamma$ acts by permuting coordinates. 
Clearly $g_c^\Gamma$ is $\Gamma$-invariant, that is, $B(c,c)$-invariant.
Moreover,
\begin{equation}
\label{E:supp_sym}
\supp(g_c^\Gamma) = \bigcup_{\gamma \in \Gamma} \gamma\; \supp(g_c)
\end{equation}
where $\supp$ is the support:
\[
\supp (h) = \mathrm{cl}\{x : h(x) \neq 0 \}
\]

In this paper we apply symmetrisation to components of
admissible maps for quasi-quotients $\GG^\RR$. 
The following simple lemma lets us perform symmetrisations on
$P^\RR$ and lift them to symmetrisations on $P$:

\begin{lemma}
\label{L:betastar}
With the above notation,
\[
[\beta^* x_{T(c)}] = \beta^* x_{[T(c)]} 
\]
\end{lemma}
\begin{proof}
Permuting entries, and replacing $c$ by $[c]$, commute. In more detail, let
$T(c) = (i_1, \ldots, i_k)$, so that $[T(c)] = ([i_1], \ldots, [i_k])$. Then
\beqn
[\beta^* x_{T(c)}] &=& (x_{[\beta(i_1)]}, \ldots,  x_{[\beta(i_k)]}) \\
	&=& (x_{\beta([i_1]}, \ldots, x_{\beta([i_k])}) \\
	&=& \beta^*(x_{[i_1]}, \ldots, x_{[i_k]}) \\
	&=& \beta^* x_{[T(c)]}
\eeqn
\qed\end{proof}

\subsection{Bump Functions}
\label{S:BF}

We use bump functions
to construct admissible $C^\infty$ maps $p:P \to P$ with a specified
compact support, taking a specific locally constant value $\mathbf{w}$ 
on a neighbourhood of a specified point $z \in P$. We denote the open ball
in the Euclidean norm of radius $r$ centre $x$  by $B_r(x)$.

\begin{proposition}
\label{P:network_bump}
Let $z \in P_c \times P_{T(c)}$ and suppose that $\delta > 0$. Let $\mathbf{w} \in P_c$.
Then there exists a $C^\infty$ map $\Psi_c:P_c \times P_{T(c)} \to P_c$ with compact support,
hence $C^1$-bounded, such that 
\beqn
\Psi_c(y) &=& \mathbf{w}\ \ \ \mbox{\rm if}\ \|y-z\| \leq \delta\\
\Psi_c(y) &=& 0\ \ \ \mbox{\rm if}\ \|y-z\| \geq 2 \delta 
\eeqn
\end{proposition}

\begin{proof}
Let $x \in P_c \times P_{T(c)}$ and let
$U = B_r(x), V = B_s(x)$ where $r < s$. By Lemma 4.2.13 of Abraham \etal~\cite{AMR83}
there is a $C^\infty$
function $\phi: \R^k \to \R$ such that $\phi(x) = 1$ when $x \in U$ and
$\phi(x) = 0$ when $x \not \in V$.
In particular, there exists a $C^\infty$ bump function $\psi:\R \to \R$ such that
\[
\psi(x) = \left\{ \begin{array}{rcl} 1 & \mbox{if} & x \in (-1,1) \\
	0 & \mbox{if} & x \not\in (-2,2) \end{array}\right. 
\]
This has compact support.
For any $\delta > 0$ let $\phi^\delta:\R^+ \to [0, \delta]$ be defined by
\begin{equation}
\label{E:phi_delta}
\phi^\delta(x) = \psi(x/\delta)
\end{equation}
Then if $\|\cdot\|_E$ is the Euclidean norm, the function
\begin{equation}
\label{E:Phi_bump}
\Phi(y) = \phi^\delta(\|y-z\|_E^2) \mathbf{w}
\end{equation}
is smooth, supported on the compact set $\overline{B_{2\delta}(z)}$, 
and satisfies $\Phi(y) = \mathbf{w}$ if $\|y-z\|\leq \delta$.
\qed\end{proof}

\subsection{Symmetrised Bump Functions}

We construct admissible perturbations using symmetrised bump functions.
In a general setting, let $\Gamma$ be a finite group acting linearly and orthogonally
on a real vector space $V$.
Denote the $\Gamma$-orbit of $A \subseteq V$ by 
\[
\OO(A) = \{ \gamma.a : \gamma \in \Gamma, a \in A\}
\]
For finite subsets $Y,Z \subseteq V$ write the (Hausdorff) distance as
\[
d(A,B) = \min\{\|y-z\| : y \in Y, z \in Z \}
\]
We state the following lemma for a general finite group action. It will be
applied when the group is a vertex group.
\begin{lemma}
\label{L:bump_sym}
Let $V,W$ be real vector spaces, and let a finite group $\Gamma$ act on $V$.
Let $A,B \subseteq V$ be finite sets with disjoint group orbits:
\[
\OO(A) \cap \OO(B) = \emptyset
\]
Let $0 \neq \mathbf{w} \in W$. Then there exists $\delta > 0$ 
and a $\Gamma$-invariant map $h:V \to W$ with compact support 
such that
\beqn
h(y) &=& 0\ \mbox{if}\ y \in \OO(A) \\
h(y) &=& \mathbf{w} \ \mbox{if}\ d(y,\OO(B)) \leq \delta \\
h(y) &=& 0 \ \mbox{if}\ d(y,\OO(B)) \geq 2\delta 
\eeqn
\end{lemma}
\begin{proof}
Since the group orbits $\OO(A)$ and $\OO(B)$ are disjoint,
 $d(\OO(A),\OO(B)) >0$. Let
\[
\delta < \shf d(\OO(A),\OO(B))
\]
so that $\delta$-neighbourhoods of $\OO(A)$ and $\OO(B)$ are disjoint.
By Proposition \ref{P:network_bump} there is a bump function $\Phi:V \to W$ such that
\beqn
\Phi(a) &=& 0\ \mbox{if}\ a \in A \\
\Phi(y) &=& \mathbf{w} \ \mbox{if}\ d(y, \OO(B)) \leq \delta \\
\Phi(y) &=& 0 \ \mbox{if}\ d(y,\OO(B)) > 2 \delta 
\eeqn
Symmetrise, to obtain $h = \Phi^\Gamma$. Then $h$ has compact
support by definition, and the stated
properties follow easily from Section~\ref{S:symmetrisation}.
\qed\end{proof}

\section{Proof of Local Rigid Synchrony Property\\
for Strongly Hyperbolic Periodic Orbits}
\label{S:LRSC}

We now adapt the method of Section~\ref{S:3NE} to prove the 
Local Rigid Synchrony Property for any network, under the extra
condition that the periodic orbit is strongly hyperbolic (indeed, `stably isolated' suffices).

If necessary, an arbitrarily small initial perturbation ensures that if $\RR$ is a
set of representatives for local synchrony colouring
then the induced periodic orbit $\X^\RR$ is hyperbolic in $P^\RR$,
hence is an isolated periodic orbit in $P^\RR$. The interval $J$ and the
size $\eps$ of a sufficiently small perturbation may have to be made smaller.
Without loss of generality, we assume such a `pre-perturbation' has been made, and
retain the same notation for $f, \X$, and so on. 

Recall that we consider a network $\GG$ with nodes $\CC = \{1,2, \ldots, n\}$,
with state spaces $P_c$ for nodes $c$ be $\R^{{k_c}}$, so that the total state space is
$P = P_1 \times \cdots \times P_n$.
Consider a $\GG$-admissible ODE \eqref{E:admiss_ODE}, 
written in coordinates as
\begin{equation}
\label{E:admissODE}
\dot{x}_c = f_c(x_c, x_{T(c)}) \quad c \in \CC
\end{equation}
Let \eqref{E:admiss_ODE_pert} be a family
of small admissible perturbations. Let $\X$ be a strongly hyperbolic periodic
orbit of \eqref{E:admiss_ODE}, with perturbed periodic orbit $\X^\eps$
for \eqref{E:admiss_ODE_pert}. Assume that at $t_0 \in \R$ the local
synchrony pattern $\bowtie_{x(t_0)}$ is rigid. 

We use Proposition~\ref{P:char_admiss} to define admissible perturbations $p$.
This implies that we need define $p_c$ only for one node $c$ from each 
input equivalence class, provided we ensure that $p_c$ is $B(c,c)$-invariant.

\begin{remark}\em
\label{r:identifications}
When $c \sim_I d$, Proposition~\ref{P:char_admiss} lets us identify $P_c$ with $P_d$
and $P_{T(c)}$ with $P_{T(d)}$. It also lets us identify
the actions of $B(c,c)$ and $B(d,d)$,
so $B(c,c)$ acts on $P_{T(d)}$ in the same way as it acts on $P_{T(c)}$.

We make these identifications without further comment from now on.
\qed\end{remark}

\subsection{Statement of Main Theorem}
\label{S:SP}

As in~\cite[Sections 2 and 6]{GRW10}, the issues raised in Section~\ref{S:CLS} 
make it necessary to prove a stronger local version of the Rigid Synchrony Property,
namely:

\begin{theorem}
\label{T:localRSC}
If a strongly hyperbolic periodic orbit of a $\GG$-admissible ODE has a
locally rigid synchrony pattern on a non-empty open interval of time, then that pattern is balanced.
\end{theorem}

The proof occupies the remainder of this section.

For any subset $J \subseteq \R$ write $x(J) = \{x(t) : t \in J \}$.
In the sequel $J$ is either a non-empty open interval or a point.
Local rigidity of $\X$ implies that there exists
a non-empty open interval $J$ such that the synchrony pattern $\bowtie_{x(t)}$
of any point $x(t) \in x(J)$ is the same for all $t \in J$, and is rigid. 
By Definition~\ref{D:LRSP},
this pattern includes {\em all} synchrony relations that hold on $x(J)$.
By local rigidity,
it is also the same colouring as that of $x^\eps(t)$ for $\eps \ll1, t \in J$.
(Here and elsewhere we may have to make $J$ or $\eps$ smaller for such
statements to be valid: we retain the same notation.)

Choose $t_0 \in J$, giving a point $x(t_0) \in \X$.  As in Section~\ref{S:CSP}, we
choose a Poincar\'e section $\Sigma$ transverse to $\X$ at $x(t_0)$.
Any perturbed periodic orbit meets $\Sigma$ transversely in a unique point, provided the
perturbation is sufficiently $C^1$-small. We use this intersection point to fix
initial conditions on perturbed periodic orbits $\tilde \X$ 
by requiring $\tilde x(t_0) \in \Sigma$. As just remarked,
$\bowtie_{x(t_0)} =\, \bowtie_J$,
and local rigidity similarly implies that $\bowtie_{\tilde x(t_0)} =\, \bowtie_{x(t_0)}$
for $\tilde x(t_0) \in \tilde\X$.

For a contradiction, assume that
\begin{equation}
\label{E:not_bal}
\bowtie_{x(t_0)} \mbox{is not balanced}
\end{equation}
The rest of this section proves that this cannot happen.

To simplify notation we write
\[
\bowtie_{x(t_0)} =\, \bz
\]
from now on.

\subsection{Induced OODE}

The first step is to replace the unperturbed ODE \eqref{E:admissODE}
by the corresponding induced overdetermined ODE (or induced OODE) 
for the synchrony pattern with colouring $\bz$.
(We get an OODE because $\bz$ is not balanced.) 
Choose some set of representatives $\RR$ for the colours, and renumber nodes
so that these are nodes $1,2, \ldots, m$. Denote the representative
of node $c$ in $\RR$ by $[c]$ as in Definition~\ref{D:bracket}. Then
$[c] = c$ when $c \in \RR$, that is, for $1 \leq c \leq m$.

The set $\RR$ defines coordinates $(u_1, \ldots, u_m)$ on the polydiagonal
$\Delta_{\bz} \subseteq P$ by setting 
\[
x_c = u_{[c]}
\]
For example, if $|\CC|=5$ and the synchrony pattern is the partition $\{\{1,2,4\}, \{3,5\}\}$,
then $(u_1,u_2)$ is identified with $(u_1,u_1,u_2,u_1,u_2) \in P$.
Now $u_r$ acts as the node coordinate for node $r$ in $\GG^\RR$
when $1 \leq r \leq m$, and canonically identifies $P^\RR$ with $\Delta_{\bz}$.

By Proposition~\ref{P:separation}, local rigidity of $\bz$ implies:
\begin{equation}
\label{E:points_distinct}
\mbox{the points}\ u^0_c(t_0)\ \mbox{are distinct for}\  1 \leq c \leq m
\end{equation}

Write input variables in standard order so that 
transitional pullback maps between distinct input equivalent
nodes can be taken to be the identity. By Proposition~\ref{P:char_admiss},
the only further requirement on the $c$-component for admissibility is
$B(c,c)$-invariance for each representative $c$ of the input equivalence classes. 
Moreover, we can identify
$B(d,d)$ with $B(c,c)$ whenever $c\sim_I d$, by Remark~\ref{r:identifications}.

We have identified $\Delta_{\bz} \subseteq P$ with $P^\RR$,
with coordinates $u_1 \ldots u_m$.
We now convert~\eqref{E:admissODE} 
into an OODE for by substituting
$u_{[c]}$ for $x_c$, for $1 \leq c \leq n$, as in Definition~\ref{D:ind+con}:
\begin{equation}
\label{E:R_OODE}
\dot{u}_{[c]} = f_c(u_{[c])}, u_{[T(c)]}) \quad 1 \leq c \leq n
\end{equation}
Recall that $[T(c)]$ is the tuple of colours of tail nodes of input arrows $I(c)$. That is, 
\[
[(i_1, \ldots, i_r)] = ([i_1], \ldots, [i_r])
\]
 for the input tuple $(i_1, \ldots, i_r)$. (The subscript on $f_c$
 remains $c$, not $[c]$, since we are not changing the map, just substituting
 different variables into it.)
 A solution $\X$ of~\eqref{E:admissODE} 
satisfies the OODE~\eqref{E:R_OODE} if and only if
 $\X$ has synchrony pattern $\bz$. 
 Thus $\X$ determines a canonical
 periodic orbit $\X^\RR$ for the $\GG^\RR$-admissible induced ODE.
 
 For all $u \in \X^\RR$ and $c \in \{1, \ldots, m\}$ we have
 \[
 u_c = u_{[c]} \qquad
 u_{[T(c)]}(t) \in P_{[T(c)]}
 \]

Say that the colouring is {\em unbalanced at nodes $c,d$} if distinct nodes
$c, d$ have the same colour, but their input sets are not
colour-isomorphic. This is equivalent to:
\begin{eqnarray}
\label{e:input1} && \mbox{either}\ c\not\sim_I d,\ \mbox{or}\\
\label{e:input2} && c \sim_I d\ \mbox{and}\
x_{[T(c)]}\ \mbox{and}\ x_{[T(d)]}\ \mbox{lie in distinct orbits of}\ B(c,c).
\end{eqnarray}
Here we again use Remark~\ref{r:identifications}, so that
$B(c,c)=B(d,d)$ acts on {\em both} input tuples.

By~\eqref{E:not_bal}, $\bz$ is not balanced at some $c \neq d$.
Renumber the nodes so that $c=1, d=m+1$. Then:
 $\bz$ is unbalanced at nodes $1, m+1$, and
node 1 has the same colour as node $m+1$.

Split~\eqref{E:R_OODE} into the {\em induced OODE} for $\RR$, which is

\begin{equation}
\label{e:induced}
\dot{u}_c = f_c(u_c, u_{[T(c)]})\quad 1 \leq c \leq m
\end{equation}
together with the {\em conflicting equation}
\begin{equation}
\label{e:conflict}
\dot u_{m+1} = \dot{u}_1 = f_{m+1}(u_1, u_{[T(m+1)]}) 
\end{equation}
(bearing in mind that $[m+1]=1$), with possible further constraint equations 
\begin{equation}
\label{e:constraint}
\dot{u}_{[c]} = f_c(u_{[c]}, u_{[T(c)]})\quad m+2 \leq c \leq n
\end{equation}
Because $\bz$ is unbalanced at nodes $1,m+1$, the two equations
\beqn
\dot{u}_1 &=& f_1(u_1, u_{[T(1)]}) \\
\dot{u}_1 &=& f_{m+1}(u_1, u_{[T(m+1)]}) 
\eeqn
conflict, in the sense of~\eqref{e:input1} or~\eqref{e:input2}: 
either nodes 1 and $m+1$ are not input equivalent, so that $f_{m+1} \neq f_1$,
or they are input equivalent, so that $f_{m+1} = f_1$, but $u_{[T(m+1)]}$
and $u_{[T(1)]}$ do not lie in the same $B(1,1)$-orbit, using Lemma~\ref{L:betastar}
to define the group action. (The same result holds if we replace $x$ by $u$ throughout.)
We will use this
conflict to deduce a contradiction. The remaining constraints~\eqref{e:constraint}
play no further role in the proof and are ignored. 
The various $f_c$ may or may not be equal, and we distinguish these cases in the proof.
Also, some input equivalence classes may not be represented
among nodes $\{1, \ldots, m\}$.

\subsection{Perturbations}

It is convenient to write 
\[
f^\RR = (f_1, \ldots, f_m) 
\]
for the map that appears in the induced ODE, to distinguish it from $f$.
Now $f:P \to P$ and $f^\RR:P^\RR \to P^\RR$. We use similar notation
to distinguish a perturbation $p$ on $P$ from $p^\RR$ on $P^\RR$.

Theorem~\ref{T:QQadmiss} implies that we can construct $\GG$-admissible
perturbations of $f$ on $P$ by first constructing $\GG^\RR$-admissible
perturbations $p^\RR$ of $f^\RR$ on $P^\RR$ and then lifting to $P$. Moreover,
lifting can be performed in a manner that does not increase the $C^1$ norm of the perturbation.
We therefore work initially with $\GG^\RR$-admissible maps on $P^\RR$,
using the previously defined coordinates $u_1 \ldots, u_m$ on $P^\RR$.


Write the OODE for the perturbed map $f+\eps p$ as
\begin{equation}
\label{e:UOODE_induced}
\dot{u}^\eps_c = f_1(u^\eps_c, u^\eps_{[T(c)]}) 
	+\eps p_c(u^\eps_c, u^\eps_{[T(c)]}) \quad 1 \leq c \leq m 
\end{equation}
with conflicting equation
\begin{equation}
\label{e:UOODE_conflict}
\dot{u}^\eps_{[m+1]} = \dot{u}^\eps_1 = f_{m+1}(u^\eps_1, u^\eps_{[T(m+1)]}) 
	+\eps p_{m+1}(u^\eps_1, u^\eps_{[T(m+1)]})
\end{equation}
and possible further constraints
\begin{equation}
\label{e:UOODE_constraint}
\dot{u}^\eps_c = f_1(u^\eps_c, u^\eps_{[T(c)]}) 
	+\eps p_c(u^\eps_c, u^\eps_{[T(c)]}) \quad m+2 \leq c \leq n 
\end{equation}

Assume, for a contradiction, that
the OODE defined by \eqref{e:UOODE_induced}, \eqref{e:UOODE_conflict},
and \eqref{e:UOODE_constraint} has a solution for any admissible
$C^1$-bounded $p$. So although \eqref{e:conflict} conflicts 
with the first component of \eqref{e:induced}, they must agree on $\X$. 
Rigidity implies that the same statement holds for $\tilde\X$.
Agreement of this kind
 is possible for a single admissible map $f$, but it is highly non-generic. We show 
that in the strongly hyperbolic case, it cannot remain valid for all small perturbations.

To avoid confusion, but at the expense of slightly more complicated notation,
we write points on components of the induced unperturbed periodic orbit as 
$u_c^0(t)$ instead of $u_c(t)$. In particular, $u_c^0(t_0)$ is the $c$-component of the
unique point (locally) at which
the unperturbed periodic orbit meets the Poincar\'e section $\Sigma$.
This notation is chosen to be consistent with the notation $u_c^\eps(t)$ for the
corresponding perturbed periodic orbit for a perturbation $\eps p$.

\subsection{Proof Strategy}

The Local Rigid Synchrony Theorem~\ref{T:localRSC} is an immediate
consequence of the following two results,
Lemma~\ref{L:p=>unbal} and Lemma~\ref{L:p_exists}.
The proof of Lemma~\ref{L:p_exists} is deferred to Section~\ref{S:coP}.

\begin{lemma}
\label{L:p=>unbal}
Suppose there exists an admissible map $p$ for $\GG$
such that for $1 \leq c \leq m$,
\begin{equation}
\label{E:allvanish}
p_c(x_c^\eps(t), x_{[T(c)]}^\eps(t)) = 0 \quad 1 \leq c \leq m,\ \eps \ll 1, |t-t_0| \ll 1
\end{equation}
and
\begin{equation}
\label{E:conflict}
p_{m+1}(x^0_{m+1}(t_0),x^0_{[T(m+1)]}(t_0)) \neq 0
\end{equation}
Then $\bz$ is not locally rigid.
\end{lemma}

\begin{proof}
The periodic orbit $\X$ induces a canonically defined periodic orbit $\X^\RR$
for \eqref{e:UOODE_induced} on $P^\RR$.
By~\eqref{E:allvanish}, for all $\eps \ll 1$ and $t$ near $t_0$, 
the induced equations for the perturbed $u^\eps_c(t)$ are the same
as those for the unperturbed coordinate $u^0_c(t)$. Now
$\X^\eps \to \X$ as $\eps \to 0$. By strong hyperbolicity, if $\eps \ll 1$ 
 then $\X^\eps = \X$.
Since initial conditions are determined by the Poincar\'e section $\Sigma$,
we have $x^\eps(t) \equiv x(t)$ for $t \in J$.  Thus
\[
u^\eps_c(t) \equiv u^0_c(t) \quad \forall t \in J
\] 
(By uniqueness of solutions of ODEs, this identity is valid for all $t\in \R$.)
The conflicting equation is
\[
\dot{u}^0_1 = f_{m+1}(u^0_1, u^0_{[T(m+1)]}) +\eps p_{m+1}(u^0_1, u^0_{[T(m+1)]})
	\quad \forall \eps \ll 1, t \in J.
\]
Since all terms except $\eps$ are independent of $\eps$,
\[
p_{m+1}(u^0_1(t), u^0_{[T(m+1)]}(t)) = 0 \quad \forall \eps \ll 1, t \in J
\]
This is a contradiction since $p_{m+1}$ is 
nonzero at $(u^0_1(t_0), u^0_{[T(m+1)]}(t_0))$, by~\eqref{E:conflict}.
\qed\end{proof}

Thus the proof of 
the Local Rigid Synchrony Property, Theorem~\ref{T:localRSC},
reduces to proving the second lemma:

\begin{lemma}
\label{L:p_exists} There exists an admissible perturbation $p$, with support
near $x^0(t_0)$ and its images under the vertex group $B(1,1)$, 
that satisfies \eqref{E:allvanish} and \eqref{E:conflict}.
\end{lemma}

We prove Lemma~\ref{L:p_exists} in the next subsection. To do so, 
we need some additional observations.

\subsection{Construction of the Perturbation}
\label{S:coP}

It remains to construct a $\GG$-admissible perturbation $p:P \to P$ 
that satisfies \eqref{E:allvanish} and \eqref{E:conflict}. We do this
by constructing a $\GG^\RR$-admissible perturbation $q:P^\RR\to P^\RR$
and lifting it to $p$ using Theorem~\ref{T:QQadmiss}.

To construct $q$,
we focus on the ODE \eqref{e:UOODE_induced} and the 
constraint \eqref{e:UOODE_conflict}, remembering that 
\[
m+1 \bz 1
\]
 To simplify notation, write
\beqn
U_c &=& (u^0_{c}(t_0), u^0_{[T(c)]}(t_0)) \\
\OO_c &=& \OO(U_c)\quad \mbox{for}\ \Gamma = B(c,c) \\
U_c^\eps &=& (u^\eps_{c}(t_0), u^\eps_{[T(c)]}(t_0)) \\
\OO_c^\eps &=& \OO(U_c^\eps)\quad \mbox{for}\ \Gamma = B(c,c)
\eeqn
The identifications in Remark~\ref{r:identifications}, transferred to $P^\RR$, imply that
the finite sets $\OO_c, \OO_d$ lie in the same space $P_c \times P_{[T(c)]}$
if $c \sim_I d$, and otherwise lie in distinct spaces. Moreover, when
$c \sim_I d$ the groups $B(c,c), B(d,d)$ and their actions on this
space are identified.

\begin{lemma}
\label{L:disj_orbits}
With the above notation,
\[
\OO_c \cap \OO_d = \emptyset\qquad \forall c,d: 1 \leq c \neq d \leq m+1, c \sim_I d
\]
\end{lemma} 
\begin{proof}
Proposition~\ref{P:separation}
 implies that the points $u_c^0(t_0)$ are distinct for $1 \leq c \leq m$.
Therefore, by projection onto $P_c$, the points
$U_c$ are distinct for $1 \leq c \leq m$.
Since $\Gamma$ acts trivially on the `base point'
$u^0(t_0)$ of $U_c$, it follows that $\OO_c \cap \OO_d = \emptyset$ whenever
$1 \leq c \neq d \leq m$. 

Since $m+1 \bz 1$, the only possible
non-empty intersection occurs for $\OO_1 \cap \OO_{m+1}$. These two sets
lie in the same space only when $m+1 \sim_I 1$, which is case (c).
But then, $\OO_1$ and $\OO_{m+1}$ are disjoint by \eqref{e:input2}.
\qed\end{proof}


\vspace{.1in}
\noindent
{\bf Proof of Lemma~\ref{L:p_exists}}

The proof splits into three cases:
\begin{itemize}
\item[\rm (a)] Node $m+1$ is not input equivalent to
any of nodes $1, \ldots, m$.

\item[\rm (b)] Node $m+1$ is input equivalent to node $k$
where $2 \leq k \leq m$.

\item[\rm (c)] Node $m+1$ is input equivalent to node 1.
\end{itemize}

The arguments are very similar in all cases, but differ in fine detail.
We take the three cases in turn.

\vspace{.1in}
\noindent
{\em Case} (a): 

Here we define the perturbation $p$ directly for $\GG$. 
We use a bump function and symmetrisation, as in Sections~\ref{S:BF} and~\ref{S:symmetrisation}, to define $p_{m+1}$ so that
\eqref{E:conflict} holds. Specifically, In Lemma~\ref{L:bump_sym} 
make $\delta$ sufficiently small and take
\begin{equation*}
\label{E:}
\begin{array}{lcl}
\Gamma = B(m+1,m+1) &&\\ 
V = P_{m+1} \times P_{[T(m+1)]} &&
\quad W = P_{m+1} \\
A = \emptyset \qquad &&
\quad B = \{(x^0_{m+1}(t_0), x^0_{[T(m+1)]}(t_0))\}
\end{array}
\end{equation*}
The hypothesis $\OO(A)\cap \OO(B) = \emptyset$ clearly holds.
 
Define $p_{m+1} = h$ where $h$ is as in Lemma~\ref{L:bump_sym}. Then
$p_{m+1}(B) \neq 0$, which is \eqref{E:conflict}.
Since variables are in standard order, the transition
maps are the identity, so this determines all $p_c$ with $c \sim_I m+1$.
On all remaining input equivalence classes, we set $p_c \equiv 0$.
In particular, $p_c \equiv 0$ for $1 \leq c \leq m$, so \eqref{E:allvanish} holds.
Since $h$ has compact support, so does $p$, and $p$ is $C^1$-bounded.

\vspace{.1in}
\noindent
{\em Case} (b):

Since $m+1 \sim_I k$ and variables are in standard order, 
$g_{m+1} = g_k$ for any $\GG$-admissible map $g$.
Let $\CC_1 = \{c : c \sim_I k\ \&\ 1 \leq c \leq m\}$. 

We define a perturbation $q$ for $\GG^\RR$ and then lift to obtain $p$
using Theorem~\ref{T:QQadmiss}. We define $q$ so that $q$ has
 compact support on $P^\RR$,
\begin{equation}
\label{E:allvanishQ}
q_c(U^\eps_c) =  0 \quad 1 \leq c \leq m,\ \eps \ll 1
\end{equation}
and
\begin{equation}
\label{E:conflictQ}
q_{k}(U_{m+1}) \neq 0
\end{equation}
To do so, we define $q_c$ for $c \in \CC_1$ and make it vanish on $\RR\setminus \CC_1$.
Two conditions on $q_k$ must be satisfied: \eqref{E:allvanishQ} when $c=k$, and
\eqref{E:conflictQ}.
If these conditions are satisfied, any lift $p$ of $q$
satisfies \eqref{E:allvanish} and \eqref{E:conflict} and also has
 compact support, hence is $C^1$-bounded. Therefore $\eps p$
 is $C^1$-small for $\eps \ll 1$. We are therefore finished once we show that
 \eqref{E:allvanish} and \eqref{E:conflict} can be satisfied simultaneously.
 
 To complete the proof, let
  $\CC_1 = \{c : c \sim_I k\ \&\ 1 \leq c \leq m\}$.
Let $q_d = 0$ for all $d \in \RR\setminus \CC_1$. For
 $c \in \CC_1$, we define $q_c$ using Lemma~\ref{L:bump_sym}. 
Make $\delta$ sufficiently small and take
\begin{equation*}
\label{E:case b}
\begin{array}{lcl}
\Gamma = B(m+1,m+1) &&\\ 
V = P_k \times P_{[T(k)]} &&
W = P_k \\
A = \bigcup_{c \in \CC_1} \OO_c &&
B = \{U_{m+1}\}
\end{array}
\end{equation*}
The hypothesis $\OO(A)\cap \OO(B) = \emptyset$  holds by Lemma~\ref{L:disj_orbits}.
This defines a $\GG^\RR$-admissible perturbation $q:P^\RR \to P^\RR$.
Lift $q$ to a $\GG$-admissible map $p:P \to P$, making it zero on all components
that are not input equivalent to a node in $\RR  \cup \{m+1\}$; then
$p$ has the required properties. Again $p$  
 is $C^1$-bounded, so $\eps p$
 is $C^1$-small for $\eps \ll 1$ and case (b) is proved. 

\vspace{.1in}
\noindent
{\em Case} (c):
Again we define a perturbation $q$ for $\GG^\RR$ and then lift to obtain $p$
using Theorem~\ref{T:QQadmiss}.

In this case $m+1 \sim_I 1$ so $q_{m+1}=q_1$. 
Let $\CC_1 = \{c : c \sim_I 1\ \&\ 1 \leq c \leq m\}$. As before, we
define $q_c$ for $c \in \CC_1$ and make it vanish on $\RR\setminus \CC_1$.
The only potential obstacle is that both \eqref{E:allvanish} and \eqref{E:conflict} 
impose conditions on $p_1$, which might be contradictory. 
We want $q_1$ to vanish at $U_1^\eps$ for $\eps \ll 1$,
but to be nonzero at $U^0_{m+1}$. Again we use Lemma~\ref{L:bump_sym}.
Make $\delta$ sufficiently small, and take
\begin{equation*}
\label{E:case c}
\begin{array}{lcl}
\Gamma = B(1,1) = B(m+1,m+1) &&\\ 
V = P_1 \times P_{[T(1)]} &&
W = P_1 \\
A = \bigcup_{c \in \CC_1} \OO_c &&
B = \{U_{m+1}\}
\end{array}
\end{equation*}
The hypothesis $\OO(A)\cap \OO(B) = \emptyset$  holds by Lemma~\ref{L:disj_orbits}.
This defines a $\GG^\RR$-admissible perturbation $q:P^\RR \to P^\RR$.
Lift $q$ to a $\GG$-admissible map $p:P \to P$, making it zero on all components
that are not input equivalent to a node in $\RR$. Then
$p$ has the required properties. Again $p$  
 is $C^1$-bounded, so $\eps p$
 is $C^1$-small for $\eps \ll 1$ and case (b) is proved. 
 
  This completes the proof of Theorem~\ref{T:localRSC}.
  \qed

\begin{remarks}\em


(a) The same strategy gives another proof of the Rigid Equilibrium Theorem, first
proved in~\cite[Theorem 7.6]{GST05}. This new proof (see~\cite{S20})
is simpler than the periodic case, and we can appeal to Sard's Theorem instead of
assuming strong hyperbolicity as an extra hypothesis.
This approach has some similarities to a proof based on transversality arguments
by Aldis~\cite[Theorem 7.2.3]{A10}, but is simpler.

(b) The same proof works if node spaces are arbitrary $C^\infty$
manifolds. Indeed, only $C^1$-smoothness is required throughout. In particular,
the above results and proofs remain valid for phase oscillators, where the node state spaces
are the circle $\Sone$. 
\qed\end{remarks}

\section{Global Rigid Synchrony and the\\Rigid Input Property}
\label{S:GRS}

Having established the main result of this paper, we can
deduce the usual global version of the Rigid Synchrony 
Property~\cite{GRW10,SP07}, assuming as before
that $\X$ is strongly hyperbolic. As discussed in Section~\ref{S:CLS},
the main obstacle is the possibility, in principle, that a local
synchrony pattern on some interval of time need not be the same as the global
synchrony pattern of the entire periodic orbit $\X$.  Indeed, the global synchrony 
pattern $\bowtie_\X$ need not equal the local synchrony pattern
$\bowtie_{x(t)}$ for any specific $t$.
In fact, rigidity prevents this happening, but the proof requires a little care.
To make the proof precise we require a number of technical definitions.
We also use the lattice of colourings from Section~\ref{S:lattice}.

\subsection{Local Rigidity Implies Global Rigidity}

We now appeal to the Local Rigid Synchrony Theorem~\ref{T:localRSC} to 
show that the above change of local synchrony pattern cannot
occur if  $\X$ is strongly hyperbolic and the global pattern $\bowtie_\X$ is rigid. The following remark is useful:

\begin{remark}\em
\label{R:manypert}  
Any finite number of perturbations performed in turn can be made arbitrarily small by
making successive sizes be $\eps/2, \eps/4, \eps/8, \ldots$. The triangle inequality
then shows that the combined perturbation has size $<\eps$.
We can also reduce the size of $\eps$ or the interval $J$ finitely many times
if required.
\qed\end{remark}

\begin{theorem}
\label{T:globalRSC}
Assume that $\X$ is strongly hyperbolic, and
suppose that the global synchrony pattern $\bowtie_{\X}$ on $\X$ is rigid.
Then it is balanced.
\end{theorem}

\begin{proof}
As before, denote the local synchrony pattern at $x_{t_0}$
by $\bz$.
Let $\bowtie_\X$ be the global synchrony pattern of $\X$ and assume this is rigid.

Since $\X \subseteq \Delta_{\bowtie_\X}$, the colouring $\bz$ is
coarser than $\bowtie_\X$. Proposition~\ref{P:semicon} implies that after any sufficiently small perturbation,
in which $\X$ becomes $\tilde\X$ and $x_0$ becomes $\tilde x_0$, the colouring
$\tbz$ can become finer than $\bz$, but not coarser.
Therefore $\tbz$ remains coarser than $\bowtie_\X$ since
$\bowtie_\X$ is rigid; that is, $\Delta_{\tbz} \subseteq \Delta_{\bowtie_\X}$.

Continue making perturbations until $\tbz$ is as fine as possible,
By Remark~\ref{R:manypert}, the combined perturbation can be made as small as we wish.
Now $\tbz$ is locally rigid. By Theorem~\ref{T:localRSC},
$\tbz$ is balanced. Therefore $\tilde \X$ intersects the
synchrony space for $\tbz$, and flow-invariance implies
that $\X \subseteq\, \Delta_{\tbz}$. By definition of the global
synchrony pattern, $\Delta_{\bowtie_\X} \subseteq\, \Delta_{\tbz}$.
This the two are equal, so $\bowtie_\X$ is balanced.
\end{proof}

\begin{corollary}
\label{C:global}
Let $\X$ be a strongly hyperbolic periodic orbit with a locally rigid
synchrony pattern on some non-empty open time interval. Then 
the entire orbit has that synchrony pattern, and it is balanced.
\end{corollary}
\begin{proof}
By Theorem~\ref{T:localRSC} the rigid local synchrony pattern $\bz$ is balanced.
Therefore $\Delta_{\bz}$ is flow-invariant.
But $\X \cap \Delta_{\bz} \neq \emptyset$,
so $\X \subseteq\, \Delta_{\bz}$.
\qed\end{proof}

Some points on the periodic orbit might have extra equalities
among their components, compared to $\bowtie_\X$,
but these cannot be balanced and cannot 
persist rigidly on any non-empty open interval of time. In particular,
the scenario of Figure~\ref{F:Delta_changes} cannot occur rigidly
when $\X$ is strongly hyperbolic.

\subsection{Local Rigid Input Property}

The Local (hence also the global) Rigid Input Property now follows trivially:

\begin{corollary}
\label{C:LRIP}
The Local Rigid Input Property holds for all strongly hyperbolic periodic orbits.
\end{corollary}

\begin{proof}
The Rigid Synchrony Property implies the Rigid Input Property
for synchronous nodes, because any colour-preserving input isomorphism
is, in particular, an input isomorphism.
\qed\end{proof}

\section{Rigid Phase Property}
\label{S:RPC}

In this section we deduce the Local Rigid Phase Property from Theorem~\ref{T:localRSC} 
using the `doubling' trick of Golubitsky \etal~\cite{GRW12}; 
see also Aldis~\cite[Chapter 10]{A10}. 

We state the Rigid Phase Property in the following local form.
The global Rigid Phase Property is the case $J = \R$.
We use $p$ rather than $\eps p$ and make $\|p\|_1$ small, and
write $\Sone  = \R/\Z$ for the circle group, representing the phase
as a proportion of the period.

\begin{theorem}[\bf Local Rigid Phase Property]
\label{T:LRPP}
Let $\X$ be a strongly hyperbolic periodic orbit of a $\GG$-admissible ODE,
and let $\tilde \X$ be the corresponding perturbed periodic orbit for an
admissible perturbation $p$ with  $\|p\|_1\ll 1$. 
Suppose that two nodes $c,d$ in $\GG$ are rigidly phase-related on a time interval $J$;
that is
\begin{equation}
\label{E:rigidphase}
\tilde{x}^h_c(t) \equiv \tilde{x}_d(t+ \theta \tilde{T}) 
	\quad \forall\ t \in J,\ \forall p \ll 1
\end{equation}
for a fixed proportion $\theta \in \Sone$ of the perturbed period $\tilde{T}$.
Then there exists a vertex symmetry $\beta \in B(d,d)$ such that
\begin{equation}
\label{E:rig_phase_input}
x_{I(c)}(t) \equiv \beta^\ast x_{I(d)}(t+ \theta T) \quad \forall\ t \in J
\end{equation}
\end{theorem}
Here $\tilde \X$ and $\tilde T$
depend on $p$, but we suppress $p$ in the notation. 
Informally, a locally rigid phase shift implies that input sets of phase-related
nodes are related by the same phase shift, up to the action of a vertex group element.

The proof
of Theorem~\ref{T:LRPP} closely mimics that of Theorem~\ref{T:localRSC}, 
so we omit routine details. First, we state a simple corollary:

\begin{corollary}
\label{C:rigphase}
If nodes are related by a locally rigid phase shift, they are input equivalent.
\end{corollary}
\begin{proof}
The map $\beta$ introduced just before equation~\eqref{E:rig_phase_input} is
an input equivalence.
\qed\end{proof}

The central idea in the proof of Theorem~\ref{T:LRPP}
is a trick from~\cite{A10,GRW12}, namely:
construct two isomorphic copies $\GG_1, \GG_2$ of $\GG$
and form the disjoint union
\[
2\GG = \GG_1 \dot{\cup}\, \GG_2
\]
If the state space for $\GG$ is $P$ then that for $2\GG$ is $P \times P$.
Take coordinates $(x,y)$ on $P \times P$ where $x = (x_1, \ldots, x_n)$ and
$y = (y_1, \ldots,y_n)$.
We recall some
simple properties of the {\em doubled network} $2\GG$ that are
proved in~\cite{GRW12}.

\begin{lemma}
\label{L:2Gadmiss}
Let $f$ be $\GG$-admissible. Then (with obvious identifications) $(f,f)$ is $2\GG$-admissible,
and all $2\GG$-admissible maps are of this form.
\end{lemma}
\begin{proof}
See \cite[Lemma 4.3]{GRW12}.
\qed\end{proof}

Assume as usual an admissible ODE~\eqref{E:admiss_ODE} on $P$ for $\GG$. 
This induces an admissible ODE on $P$ for $2\GG$ of the form:
\begin{equation}
\label{E:ffODE}
\dot{x} = f(x) \qquad \dot{y} = f(y)
\end{equation}

The dynamics of $\GG_1$ and $\GG_2$ are {\em decoupled}, so
a periodic state $\X = \{x(t)\}$ for $f$ on $\GG$ gives rise to a 
$2$-torus $\T^2$ for $(f,f)$ on $2\GG$, foliated by periodic orbits
\begin{equation}
\label{E:Xtheta}
\X_\theta= \{ (x(t), x(t+\theta T)) : t \in \R \}
\end{equation}
where $\theta \in\Sone$.  We call $\X_\theta$ a {\em $\theta$-sheared} periodic orbit.
     
The following result
 converts rigid phase relations on $\GG$ into rigid synchronies on $2\GG$;
see~\cite[Section 4]{GRW12}. The proof is immediate.
\begin{lemma}
\label{L:phase->synch}
Nodes $c,d$ are rigidly phase-related by $\theta$ in $\GG$, with $c$ corresponding to $c_1$
in $\GG_1$ and $d$ corresponding to $d_2$ in $\GG_2$, if and only if $c_1, d_2$ are
rigidly synchronous on $\X_\theta$.
\qed\end{lemma}

The method of proof assumes that $\GG$ has the Rigid Synchrony Property for
strongly hyperbolic $\X$, and uses properties of $\X_\theta$, where $\theta$ is the
assumed rigid
phase relation. The idea is to deduce that $\X_\theta$ has the Rigid 
Synchrony Property for the periodic orbit on $2\GG$. Then Lemma~\ref{L:phase->synch}
yields the Rigid Synchrony Property for $\X$ on $\GG$.
However, the Rigid Synchrony Property for $\X_\theta$ is not immediate
because the foliation by tori is an obstacle to hyperbolicity. 
To avoid this obstacle we establish a modified version of the 
Rigid Synchrony Property for $\X_\theta$ on $2\GG$.
This is achieved as follows.

\vspace{.1in}
\noindent
{\em Proof of Theorem~{\rm \ref{T:LRPP}}}:

(a) Assume a rigid phase relation
\[
x_c(t) = x_d(t+\theta T) \qquad t \in J
\]
where either $c \neq d$ or $c=d$ and $\theta \not\equiv 0 \pmod {T}$.
Assume for a contradiction that this relation does not extend to the
corresponding input tuples.

(b) Assume that $\X$ is strongly hyperbolic (or is locally Kupka-Smale or
has the strong isolation property). This implies that $\GG$ has the
Rigid Synchrony Property for $\X$ but in principle is a stronger condition.

(c) Consider the corresponding admissible ODE \eqref{E:ffODE} for $2\GG$.
Let $\X_\theta$ be defined by \eqref{E:Xtheta}, and consider
an admissible perturbation of the form $(\eps p,\eps p)$ where $p$ is
admissible for $\GG$ and $C^1$-bounded (which follows if
$p$ is compactly supported). The perturbed ODE has the form
\begin{equation}
\label{E:pert_ffODE}
\dot{x} = f(x)+\eps p(x)  \qquad \dot{y} = f(y)+\eps p(y)
\end{equation}

If the phase shift $\theta$ is rigid, the
unique perturbed periodic orbit $\tilde\X = \{(\tilde x^\eps(t), \tilde x^\eps(t+\theta \tilde T)\}$
satisfies
\[
\tilde x_c(t) = \tilde x_d(t+\theta \tilde T) \qquad t \in J
\]
By Lemma~\ref{L:phase->synch}, this corresponds to a rigid
synchrony relation for $2\GG$.
 
Theorem~\ref{T:localRSC} does not apply directly, as noted above, but we can
use the same method of proof with extra conditions. The proof
has three key ingredients:

(a) There exists a generic point $x(t)$ on the periodic orbit with $t \in J$.

(b) There is a conflicting component of the ODE; that is, one that is formally
inconsistent with the equation
\[
f_c(x_c,x_{T(c)}) = f_d(x_d(t+\theta T), x_{T(d)} (t+\theta T))
\]
in the sense that either $d \not\sim c$ or
$x_{T(c)}$ and $x_d(t+\theta T)$ do not lie in the same $B(c,c)$-orbit. 

(c) $\X^\RR$ is hyperbolic, if necessary after a pre-perturbation of $f$.

If we can arrange analogous statements for $\X_\theta$ on $P \times P$,
the proof goes through and the resulting contradiction establishes
Theorem~\ref{T:LRPP}. 

A useful simplifying step is to form the quotient $\GG/{\bowtie}$ where $\bowtie$
is the relation of (local or global) rigid synchrony, which we know is balanced.
Now nodes of $\GG/{\bowtie}$ are synchronous if and only if they are identical.
Replacing $\GG$ by this quotient (and renaming this $\GG$), we may
assume that the only synchrony relations for $\X_\theta$ on $2\GG$
are those of the form $x_c(t) \equiv x_d(t+\theta)$; that is, between $P \times \{0\}$
and $\{0\}\times P$.

Condition (a) is straightforward, except perhaps when $x_1(t) \equiv x_1(t+\theta)$,
so the orbit on node 1 has period $\theta < T$. (This could, for example, occur for a
multirhythm on node 1.) This possibility corresponds to case (c) of the proof
of Lemma~\ref{L:p_exists}, and is dealt with in the same manner.
We are assuming that node 1 does not satisfy the
Rigid Phase Property, so $x_{[T(1)]}(t_0)$ and $x_{[T(1]}(t_0+\theta)$ lie
in distinct $B(1,1)$-orbits. Therefore the required conditions on $p_1$ can be
satisfied.

Condition (b) is immediate because we are assuming, for a contradiction,
that a formal inconsistency occurs.

Condition (c) can be dealt with by working only with
sheared periodic orbits for $2\GG$. We need:

\begin{definition}\em
\label{D:theta_hyp}
A periodic orbit $\X_\theta$ is {\em quasi-hyperbolic} if all of its
Floquet multipliers are off the unit circle except for two that equal 1.
Of these, one is associated with a phase shift along $\X_\theta$,
while the other is associated with a change from $\X_\theta$
to $\X_\phi$ where $\phi \in \Sone$ with $\phi$ near $\theta$ and $\phi \neq \theta$.
\qed\end{definition}

\begin{definition}\em
\label{D:quasiP}
Let $\PP$ be one of the properties `strongly hyperbolic', `stably isolated',
`locally Kupka-Smale' for $\X$. Then the property {\em quasi-$\PP$} for
$\X_\theta$ is defined for
`strongly hyperbolic' and `locally Kupka-Smale'
by replacing `hyperbolic' by `quasi-hyperbolic'
in the definition of $\PP$. For `stably isolated' it is defined by
being stably isolated except for nearby periodic orbits $\X_\phi$.
\qed\end{definition}

Now condition (c) follows from:
\begin{lemma}
\label{L:G->2G}
Let $\PP$ be one of the properties `strongly hyperbolic', `stably isolated', or
`locally Kupka-Smale'. If $\X$ has property $\PP$ for~\eqref{E:admiss_ODE}, 
then $\X_\theta$ has the property quasi-$\PP$ for~\eqref{E:ffODE}.
\end{lemma}
\begin{proof}
First, observe that $2\GG$-admissible perturbed families have the form
\eqref{E:pert_ffODE}.

{\em Strongly hyperbolic}: There is a perturbation $p$ such that $\{\tilde x^\RR(t)\}$ is
hyperbolic on $P$. Since $\tilde x(t+\theta T)$ is a phase-shifted copy of $\tilde x(t)$,
the orbit $\{\tilde x^\RR(t+\theta)\}$ is hyperbolic on $P$ for the same perturbation $p$.
Now all Floquet multipliers of $\X^\theta$ lie off the unit circle except
for two multipliers that are equal to 1: one for the first component $\{\tilde x(t)\}$
and the other for the second component $\{\tilde x(t+\theta T)\}$. Restricting
to any sheared periodic orbit (which includes $\tilde\X_\theta$) reduces
these to one Floquet multiplier equal to 1 and the rest off the unit circle.
Therefore $\X^\RR_\theta$ is quasi-hyperbolic.

{\em Stably isolated}: If $\{\tilde x(t)\}$ is stably isolated on $P$ for a small perturbation $p$,
then since $\tilde x(t+\theta T)$ is a phase-shifted copy of $\tilde x(t)$,
the orbit
$\{\tilde x^\RR(t+\theta)\}$ is stably isolated on $P$ for the same perturbation $p$.
Therefore $\tilde\X_\theta$ is quasi-stably-isolated on $P\times P$.

{\em Locally Kupka-Smale}: The proof is similar to case (a).
\qed\end{proof}

This completes the proof of Theorem~\ref{T:LRPP}.

\begin{remark}\em
The example of Section~\ref{S:3NE} sheds light on the requirement of rigidity. Among the
admissible maps~\eqref{e:3nodeODE}, consider the `non-generic' case
when $g = f$. The ODE then has $\Z_3$ symmetry, and therefore
supports rotating waves with $\sot$-period phase shifts --- see for
example~\cite[Chapter XVIII Section 0]{GSS88}, and
\cite[Section 4.8]{GS02}. These phase shifts would be rigid if the dashed
arrow in Figure~\ref{F:3node_ring} had the same arrow-type as the solid ones. However, since
this is not the case, a perturbation of the form $(p,p,q)$ can
(and, as we proved above, does) change the phase shifts, so that
they are no longer one third of the period.
\qed\end{remark}

\section{Full Oscillation Property}
\label{S:FOC}

We now deduce the Full Oscillation Property
for strongly hyperbolic $\X$. 
Following Gandhi \etal~\cite{GGPSW19}, say that node $d$ is 
{\em upstream} from node $c$ if there is a directed
path in $\GG$ from $d$ to $c$.
We prove a stronger local version of the property: if some node of a network
is rigidly steady for some non-empty open interval of time $J$,
then all upstream nodes are also steady for $t \in J$. The global
version follows immediately since `oscillate' is local in time.

A node is always synchronous with itself. In the
proof of the Local Rigid Synchrony Property, we 
do not assume that $\X$ is oscillating at any particular node; only
that the overall orbit is periodic (not steady) and nodes 1 and $m+1$ are distinct. 

Phase relations are different. A node
can be phase-related to {\em itself} in a nontrivial manner; indeed, this is precisely what
happens in multirhythms. 
A multirhythm occurs when some nodes oscillate
at frequencies rationally related to the overall period, because the phase pattern
requires those nodes to oscillate as nontrivially 
phase-shifted copies of themselves. This phenomenon
goes back to~\cite[Chapter XVIII Section 0]{GSS88}, 
and is discussed in~\cite[Section 3.6]{GS02}. 

The proof of the Rigid Phase Property
allows multirhythms, because the two copies of such a node 
are distinct in $2\GG$. In fact, if node $c$ experiences a multirhythm, rigidly,
then Theorem~\ref{T:LRPP} implies that $x_{T(c)}$ is invariant under the same phase shift,
up to some input automorphism in $\beta \in B(c,c)$. That is,
\begin{equation}
\label{E:phaselift}
x_c(t) \equiv x_c(t+\theta) \implies x_{T(c)}(t) \equiv \beta^*x_{T(c)}(t+\theta)
\end{equation}
The input
automorphism $\beta$ is essential here. Indeed, without some such automorphism,
the phase shift $\theta$ would propagate back through the entire (transitive) network and imply that
all nodes oscillate with the same {\em minimal} period. This is false for
multirythms, and is why they are interesting.

\begin{definition}\em
\label{D:rig_steady} 
Let $\X = \{x(t)\}$ be a hyperbolic periodic orbit of a network ODE.
A node $c$ is {\em rigidly steady} on a non-empty open set
$J \subseteq\, \R$ if $\tilde{x}_c$ is steady (that is, $\tilde x_c(t)$ is constant)
for all sufficiently small perturbed periodic orbits $\tilde{\X} = \{\tilde{x}(t)\}$.
\qed\end{definition}

The proof of the Local Rigid Phase Property makes no extra assumptions about 
the phase shift $\theta$,
except that it is rigid. In particular, other phase relations are also permitted
(and occur in multirhythms). The connection with the Full Oscillation Property
arises because equilibria can be viewed as extreme cases of multirhythms:
\begin{lemma}
\label{L:anyphase}
Let $\X=\{x(t)\}$ be a periodic state with period $T$, and let $J$
be a non-empty open subset of $\R$. Then $x_c(t)$ is an equilibrium for $t\in J$
 if and only if
\[
x_c(t) = x_c(t+\theta T) \quad \forall \theta \in \Sone = \R/\Z, t \in J
\]
\qed\end{lemma}

As observed in~\cite[Section 2]{GRW12}, in order to prove the Local Full Oscillation Property it is enough to
prove that a small enough perturbation makes at least one additional node oscillate.
When $\GG$ is transitive, iterating with smaller and smaller perturbations,
as in Remark~\ref{R:manypert}, 
makes all nodes oscillate, because `node $c$ oscillates on $J$' is an open property. 

We prove a stronger result.
The key observation is:

\begin{lemma}
\label{L:upstream}
Suppose that $\GG$ has the Local Rigid Phase Property and node $c$ is rigidly steady
for $t \in J$.
Then every node $d$ upstream from $c$ is also rigidly steady for $t \in J$.
\end{lemma}
\begin{proof}
We prove that any input node $d \in T(c)$ is steady for $t \in J$. That is is rigidly 
steady then follows by Remark~\ref{R:manypert}.

Since node $c$ is rigidly steady, any sufficiently small perturbation
creates a unique perturbed periodic orbit $\tilde{\X}$ that is also steady at $c$.
Therefore by Lemma~\ref{L:anyphase},
\[
\tilde{x}_c(t) = \tilde{x}_c(t+\theta \tilde{T}) \quad \forall \theta \in \Sone
\]
where $\tilde{T}$ is the period of $\tilde{\X}$.
That is, the phase shift $\theta$ at node $c$ is rigid for all $\theta \in \Sone$.
(All phase shifts in $\Sone$ occur, but any particular one is preserved by perturbation,
and is distinguished by having that value of $\theta$ --- which is what the 
proof of Theorem~\ref{T:LRPP} requires.)
Theorem~\ref{T:LRPP} therefore implies that there exists $\beta \in B(c,c)$ such that
\[
x_{T(c)}(t) \equiv \beta^*x_{T(c)}(t+\theta T) \quad \forall \theta \in \Sone
\]
Since $B(c,c)$ is a finite group, $\beta^k = \id$ for some $k$. Then
\[
x_{T(c)}(t) \equiv (\beta^*)^k x_{T(c)}(t+k\theta T) 
	 \equiv x_{T(c)}(t+k\theta T) \quad \forall \theta \in \Sone.
\]
Now $k\theta$ ranges over the whole of $\Sone$ since $\theta$ does, so
by Lemma~\ref{L:anyphase}, $x_{T(c)}$ is in equilibrium. In particular
any node $d \in T(c)$ is in equilibrium. As noted at the start of the proof, node
$d$ is rigidly steady, so we can iterate. Since $\GG$ is finite, 
after finitely many steps we deduce, using Remark~\ref{R:manypert},
 that any given upstream node must be  rigidly steady.
\qed\end{proof}

\begin{theorem}
\label{T:FOC}
The Full Oscillation Property holds for all strongly hyperbolic periodic orbits.
\end{theorem}
\begin{proof}
Suppose not. Then some transitive network has a hyperbolic
periodic orbit that is rigidly steady at some node $c$. But
in a transitive network, every node is upstream from
$c$. Therefore every node is steady, so
the state is not periodic --- contradiction.
\qed\end{proof}

\section{Cyclic Automorphisms and the $H/K$ Theorem}
\label{S:CAGHKT}

It is known that if 
conjectures (a, b, c, d) are valid for a network $\GG$, which we have proved is the case
for strongly hyperbolic periodic orbits, then
there are important consequences for the combinatorial structure of $\GG$.
In Golubitsky \etal~\cite{GRW12} and~\cite{SP08} it is
proved that, on the assumption that
these conjectures are valid for a given network $\GG$, there is
a natural network analogue of the $H/K$ Theorem of Buono and Golubitsky~\cite{BG01};
see also Golubitsky and Stewart~\cite{GS02} and Golubitsky \etal~\cite{GMS16}.

In Section~\ref{S:rigidity} we mentioned that for
equivariant dynamics with symmetry group $\Gamma$,
the  $H/K$ Theorem characterises, for each $\Gamma$, 
the possible spatiotemporal patterns of 
periodic states $\X$ that can occur for suitable $\Gamma$-equivariant ODEs.
This characterisation is stated in terms of the the {\em spatiotemporal symmetry group} $H$,
which fixes $\X$ setwise, and the
{\em spatial symmetry group} $K \subseteq H$, which fixes $\X$ pointwise. 
It is easy to prove that $K \lhd H$ and (when $\Gamma$ is finite) 
the quotient group $H/K$ is cyclic and corresponds to phase shifts
through certain rational multiples of the period. The synchrony and phase patterns
determined by such subgroups 
$H$ and $K$ are always rigid~\cite[Corollary 3.7]{GS02}. 

For a network, the natural analogue of $K$ is a balanced
colouring $\bowtie$, determined by the synchrony pattern, and the natural analogue of
$H$ is the phase pattern. The crucial feature in common with the
equivariant $H/K$ Theorem is proved in~\cite{SP08} under the assumption that
$\GG$ has the Rigid Phase Property. Namely, when $\GG$ is transitive,
the existence of a rigid phase pattern implies
that the quotient network $\GG/{\bowtie}$ of $\GG$ by the synchrony colouring
$\bowtie$ has a cyclic symmetry group. Moreover, this symmetry group
implies the existence of a discrete rotating wave with the 
corresponding phase pattern. So rigid phase relations
on an arbitrary network occur if and only if they come from
a cyclic group symmetry on $\GG/{\bowtie}$. 

The same proof works if we assume only the Local Rigid Phase Property
on some interval $J$. We can also pass from a specific rigid phase relation
to the entire phase pattern in the sense of~\cite{SP07}, to establish:

\begin{theorem} \label{T:spatiotemporal}
Let $\GG$ be a transitive\index{transitive}
network, and assume that $\GG$ has the (global) Rigid Phase Property.
Suppose that there is a rigid phase pattern corresponding to a 
$T$-periodic state $x(t)$ of an admissible ODE.  
Then there is a balanced coloring
$\bowtie$ of $\GG$ and a symmetry $\gamma$ of the quotient network $\GG_{\bowtie}$,
generating a cyclic group $\Gamma \cong \Z_k$,
such that: 
\begin{itemize}
\item[\rm (a)] $x_c(t) \equiv x_d(t)$ if and only if $c \bowtie d$.   
\item[\rm (b)] For each pair
$(c,d)$ of nodes that are rigidly phase-related:
\[
x_c(t) \equiv x_d(t+\theta_{cd} T)
\]
we have $\theta_{cd} = \frac{m}{k}$ for some integer $m$.
Moreover, $\gamma\bar{c} = \bar{d}$ where $\bar{c},\bar{d}$ are 
the quotient nodes corresponding to $c,d$. 
\end{itemize}
\qed\end{theorem}

The converse is true with node spaces $\R^m$ for which $m \geq 2$,
by results of Josi\'c and T\"or\"ok~\cite{JT06}. The precise characterisation
of $H,K$ pairs for network ODEs remains open when $m = 1$ or node spaces are $\Sone$
(phase oscillators).

\begin{corollary}
\label{C:HKsh}
The above theorem holds if $\X$ is strongly hyperbolic.
\qed\end{corollary}

Transitivity of $\GG$ is required in Theorem~\ref{T:spatiotemporal}
because nodes can be removed from
a feedforward network without affecting the phase pattern or its rigidity,
but destroying the cyclic symmetry. This issue is raised in 
Stewart and Parker~\cite[Section 3.1]{SP08}.
Golubitsky \etal~\cite[Theorem 1.4]{GRW12} prove a similar result. 
Finally, Golubitsky \etal~\cite[Section 7]{GRW12} provide a detailed
and thorough discussion of rigid phase patterns in non-transitive
networks. Here it may be necessary to {\em complete} the network $\GG$ by adding
further downstream nodes and arrows to obtain a larger network
$\GG^*$. This extension does not affect the dynamics
on $\GG$ because nodes in $\GG$ force those in $\GG^*\setminus\GG$,
but it restores cyclic group symmetry.

\section{Proof of Local Rigidity Properties for all\\
 1- and 2-Colourings}
 \label{S:12C}

We now remove the hypothesis of strong hyperbolicity
in some special cases. These results are new and add evidence in support
of the Rigidity Conjectures.

A colouring with $k$ colours is called a {\em $k$-colouring}.
We prove the first three Rigidity Conjectures (RIC, RSC, RPC) 
for 1- and 2-colourings by proving that strong hyperbolicity is generic. 
For the Rigid Synchrony
Conjecture, the number of colours refers to the number of synchrony
classes. For the Rigid Phase
Conjecture, the number of colours refers to the number of
node waveforms that are the same up to a phase shift. That is,
$c \bowtie d$ whenever $x_c(t) \equiv x_d(t+\theta T)$ in the above
notation. If the Rigid Phase Property holds, this colouring is balanced.
We exclude the Full Oscillation Conjecture (FOC) because this is not 
associated with a specific number of colours.

We will prove, for any network $\GG$:

\begin{theorem}
\label{T:2colRC}

{\rm (a)} The Rigid Input, Synchrony, and Phase  Properties 
hold for any $1$-colouring of any network.

{\rm (b)} The Rigid Input, Synchrony, and Phase  Properties 
hold for any $2$-colouring of any network.
\end{theorem}

The proof is deferred to Sections~\ref{S:RC1C} and \ref{S:RC2C}.

\begin{corollary}
\label{C:upto3}
The Rigid Input, Synchrony, and Phase  Properties and the
Full Oscillation Property hold for any network with $1$, $2$, or $3$ nodes.
\qed\end{corollary}
Here we include the Full Oscillation Property because 
this can be defined for a specific number of nodes.

These results are new. They add evidence supporting the Rigidity Conjectures.
In particular, if a counterexample exists, it must have at least 4 nodes and
the synchrony pattern must involve at least $3$ colours. Moreover,
the periodic orbit $\X$ must fail to be strongly hyperbolic.
Indeed, for any unbalanced synchrony pattern $\bowtie^0$,
after any small perturbation, each $\X^\RR$ is the limit of a continuum
of periodic orbits of the induced ODE for $\RR$, containing periodic
orbits distinct from $\X^\RR$ that meet any neighbourhood of the point
$x^0(t_0)$.

\subsection{ODE-Equivalence}

The proof of Theorem~\ref{T:2colRC} depends on the concept of ODE-equivalence
\cite{DS05, LG06}. We summarise the definition and basic properties.

A fundamental feature of network dynamics is that for a given choice
of node spaces $P_c$ and total state space $P = \prod P_c$, each network $\GG$
determines a class of differential equations on $P$ defined by 
the admissible maps. There is a bijection between
network diagrams and these `admissible 
classes'. Despite this, it was pointed out in~\cite{LG06} that
networks with different diagrams can define the same {\em space} of admissible maps. 
Such networks are said to be ODE-equivalent,
because any admissible ODE for one of them can be interpreted as an
admissible ODE for the other, and the same goes for the solutions of the ODEs.

Nontrivial ODE-equivalence can occur because of the
technical but vital distinction between a component $f_c:P \to P_c$ 
and the associated $\hat f_c : P_c\times P_{T(c)} \to P_c$;
see Definition~\ref{D:admiss} (b). Although
$\hat f_c$ determines $f_c$ uniquely, different choices of
$\hat f_c$ can determine the same $f_c$. Admissible classes
use specific {\em presentations} $\hat f_c$ of the components of the map,
whose domain lists the tail nodes of arrows. The space of admissible maps
$P \to P$ is determined by the $f_c$, and this difference 
creates the ambiguity. 

\begin{remark}\em
\label{R:ODEpres}
Properties of specific orbits (such as periodicity,
hyperbolicity, and synchrony and phase patterns) are preserved when
passing to an ODE-equivalent network.
\qed\end{remark}

\begin{figure}[htb]
\centerline{
\includegraphics[width=.45\textwidth]{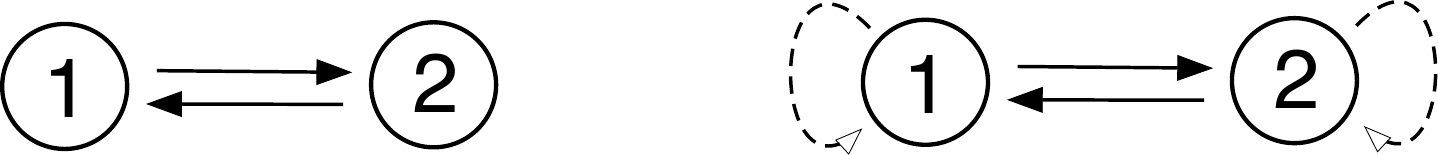}}
\caption{Two $2$-node networks with different network topologies
that define the same space of admissible maps. {\em Left}: $\GG_1$. {\em Right}: $\GG_2$.}
\label{F:2nodeODE}
\end{figure}

\begin{example}\em
\label{ex:2nodeODEequiv}

Figure~\ref{F:2nodeODE}
shows a simple example of  ODE-equivalence, discussed 
briefly in~\cite{GST05}. In $\GG_1$ both nodes have the same node-type, 
 and similarly for $\GG_2$.
Suppose that the state space for all four nodes is $\R^k$.
The admissible maps for $\GG_1$ have the form
\[
F(x_1, x_2) = (f(x_1, x_2), f(x_2, x_1))
\]
where $f : \R^k\times\R^k \to \R^k$ is any smooth map. 
The admissible maps for $\GG_2$  have the form
\[
G(x_1, x_2) = (g(x_1, x_1, x_2), g(x_2, x_2, x_1))
\]
where $g : \R^k \times \R^k\times\R^k \to \R^k$ is any smooth map. 

It is now easy to 
see that the set $\{G\}$ of all $G$ is the same as the set $\{F\}$ of all $F$. 
Namely, given $f$ we can define $g(u, v, w) = f(u, w)$
so that $\{G\} \subseteq \{F\}$. Conversely, given $g$ we can define $f(u, v) = g(u, v, v)$,
so that $\{F\} \subseteq \{G\}$. Therefore $\GG_1$ and $\GG_2$ are ODE-equivalent. 
\qed\end{example}
 
 \begin{definition}\em
\label{D:ODEequiv}
Networks $\GG_1$  and $\GG_2$ with the same sets of nodes
(up to the numbering of the nodes) are {\em ODE-equivalent}, written
$\GG_1 \simODE \GG_2$, if, for the same choices of node spaces,
$\GG_1$  and $\GG_2$ have the same spaces of admissible maps.

They are {\em linearly equivalent} if they have the same spaces of linear admissible maps.
\qed\end{definition}

Recall that the adjacency matrix $A$ for a given arrow-type is the matrix
$(a_{ij})$ for which $a_{ij}$ is the number of arrows $e$ of that type
such that $\hd(e) = i$ and $\tl(e) = j$. 
When arrow-types are irredundant (see Section~\ref{S:redundancy})
and node spaces are 1-dimensional, it is easy to see that
the space of linear admissible maps is spanned by
the adjacency matrices for the separate arrow types, including internal node `arrows'
distinguished by node type.
Linear equivalence then becomes `the adjacency matrices span the same space'.

The key result on ODE-equivalence is~\cite[Theorem 7.1]{DS05}:

\begin{theorem}
\label{T:DS7.1}
Two networks are ODE-equivalent if and only if they are linearly equivalent.
Moreover, when verifying linear equivalence we may assume that all node
spaces are $1$-dimensional.
\qed\end{theorem}

This theorem reduces ODE-equivalence to routine linear algebra, applied to
the space spanned by the adjacency matrices.

\begin{example}\em
\label{ex:lin->ODE}
Consider the networks $\GG_1$  and $\GG_2$ in Figure~\ref{F:2nodeODE}.
The linear admissible maps for $\GG_1$ are spanned by
\[
\Matrix{1 & 0 \\ 0 & 1} \qquad \Matrix{0 & 1 \\ 1 & 0}  
\]
(adjacency matrices for internal node arrows, solid arrows, respectively).
Those for $\GG_2$ are spanned by
\[
\Matrix{1 & 0 \\ 0 & 1} \qquad \Matrix{0 & 1 \\ 1 & 0} \qquad \Matrix{1 & 0 \\ 0 & 1} 
\]
(internal node arrows, solid arrows, dashed arrows, respectively).
Clearly these spaces are the same, so $\GG_1 \simODE \GG_2$.
\qed\end{example}

\subsection{1-Colour Synchrony}
\label{S:RC1C}

The proof of the first three Rigidity Conjectures for 1-colourings 
(that is, when the orbit is fully synchronous for the RSC and RIC,
and has a single waveform up to phase for the RPC)
is now straightforward:
\vspace{.1in}

\noindent 
{\bf Proof of Theorem~\ref{T:2colRC} (a)}

The result is true for equilibria by~\cite{A10,GST05,S20}, 
so we may assume $\X$ is not an equilibrium.
Assume for a contradiction that 
$\GG$ is not homogeneous. 
The periodic orbit $\X$ is fully synchronous; that is, $x_c(t) \equiv x_d(t)$
for all nodes $c,d$, so the colouring $\bowtie\ = \{1,2, \ldots, n\}$.
Let $\RR = \{1\}$.

The quasi-quotient $\GG^\RR$ is a 1-node network, on which
the induced ODE is $\dot u = f_1(u, \ldots,u)$. We claim that 
$\GG^\RR$ is ODE-equivalent to a 1-node network with no arrows
(other than the internal node `arrow'). This follows from
Theorem~\ref{T:DS7.1}, because all
adjacency matrices are integer multiple of the identity.
Admissible maps are therefore arbitrary smooth functions of $u$,
so the standard Kupka-Smale Theorem implies that $\GG^\RR$
is a Kupka-Smale network, which implies the Rigidity Properties.
\qed

The proof of the Rigidity Conjectures for 2-colourings 
requires further preparation, done in the next subsection.
We complete the proof in Section~\ref{S:RC2C}.

\subsection{Classification of $2$-Node Networks up to ODE Equivalence}

To deal with 2-colourings we first
classify all possible 2-node networks up to ODE-equivalence.
This is straightforward, but seems not to be in the literature, so we
give details. 

\begin{figure}[htb]
\centerline{
\includegraphics[width=.12\textwidth]{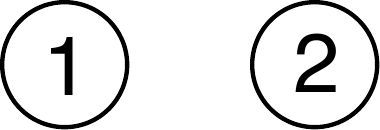} \qquad\qquad
\includegraphics[width=.12\textwidth]{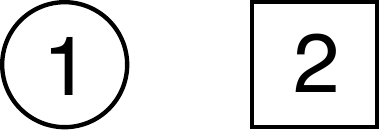} \qquad\qquad
\includegraphics[width=.12\textwidth]{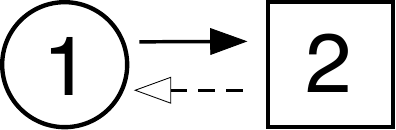} \qquad\qquad
\includegraphics[width=.12\textwidth]{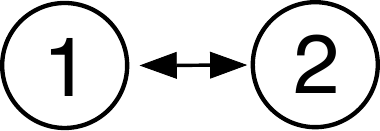}}
\centerline{
\quad\ \ (1)  \qquad \quad \quad  \qquad(2)  \ \ \qquad \qquad \qquad \qquad 
(3) \qquad \qquad \qquad \quad  (4)
}
\vspace{.2in}
\centerline{
\includegraphics[width=.12\textwidth]{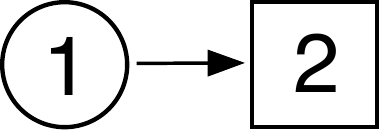} \qquad\qquad
\includegraphics[width=.17\textwidth]{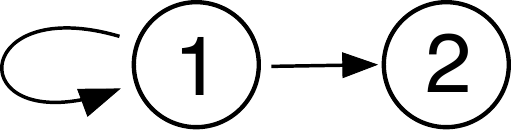} \qquad\qquad
\includegraphics[width=.17\textwidth]{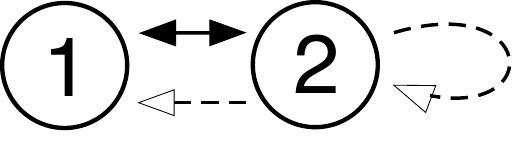} \qquad\qquad
\includegraphics[width=.17\textwidth]{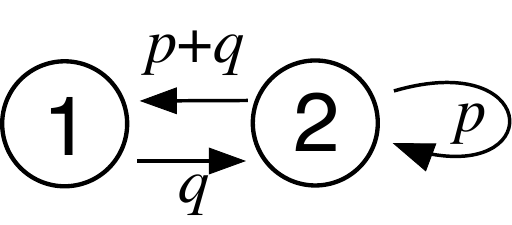}
}
\centerline{
(5) \qquad\qquad \qquad \qquad \quad (6)  \qquad \qquad \qquad \qquad 
(7) \qquad \qquad \qquad \qquad  (8)
}
\caption{Classification of 2-node networks up to ODE-equivalence. In case
(8) we require $p, q >0$ and $\gcd(p,q)=1$.}
\label{F:2nodeODElist}
\end{figure}

\begin{theorem}
\label{T:2nodeODE}
Every $2$-node network is ODE-equivalent to precisely one of the networks
illustrated in Figure~{\rm \ref{F:2nodeODElist}}.
\end{theorem}
\begin{proof}

By Theorem~\ref{T:DS7.1} it is enough to
classify $2$-node networks up to linear equivalence, assuming all node spaces
are $1$-dimensional. The classification therefore reduces to considering the
vector spaces spanned by the adjacency matrices for each arrow-type (assuming irredundancy).
There are three main cases:

{\rm (a)} The network $\GG$ is disconnected.

{\rm (b)} The network $\GG$ is feedforward; that is, its is connected and
has more than one transitive component.

{\rm (c)} The network $\GG$  is transitive.
For a 2-node network, transitivity is equivalent to being all-to-all connected.
Throughout, let $I$ be the $2 \times 2$ identity matrix.

\noindent
{\em Case} (a):
This is trivial. Either $\GG$ is homogeneous, or not. If it is homogeneous
then all adjacency matrices are multiples of 
$I$, so $\GG$ is ODE-equivalent to network (1) in the figure. If not, 
all adjacency matrices are positive integer multiples of either of the following two:
\[
\Matrix{1 & 0 \\ 0 & 0} \qquad \Matrix{0 & 0 \\ 0 & 1}
\]
Now $\GG$ is ODE-equivalent to network (2) in the figure.

\noindent
{\em Case} (b):
Clearly any network ODE-equivalent to a feedforward
network is also feedforward (all adjacency matrices have a common
block-triangular structure).
Renumbering if necessary we can assume that that node 1 is the upstream node.
Either $\GG$ is homogeneous, or not. If it is homogeneous then
the internal arrow-type has adjacency matrix a multiple of $I$.
Every other arrow-type also has adjacency matrix a multiple of $I$,
or it has adjacency matrix of the form
\[
\Matrix{a+b & 0 \\ a & b}
\]
with $a \geq 1$,
by homogeneity. By connectedness, at least one such arrow-type has
$a > 0$. Subtracting $bI$ we obtain
\[
a \Matrix{1 & 0 \\ 1 & 0}
\]
So the span of the adjacency matrices is that of
\[
\Matrix{1 & 0 \\ 0 & 1} \qquad \Matrix{1 & 0 \\ 1 & 0}
\]
however many arrow-types there may be. 
(Different types may have different entries $a,b$ but the
same statement holds.) Therefore $\GG$ is ODE-equivalent 
to network (6) in the figure.

If $\GG$ is not homogeneous then the internal arrow-types have adjacency matrices
\[
\Matrix{1 & 0 \\ 0 & 0} \qquad \Matrix{0 & 0 \\ 0 & 1}
\]
Any other arrow-type either has adjacency matrix a multiple of one of these,
or its adjacency matrix has the form
\[
\Matrix{a & 0 \\ c & d}
\]
with at least one arrow-type for which $c > 0$.
We can subtract the diagonal terms without changing
the span of the adjacency matrices, and then divide by $c$.
Therefore the span of the adjacency matrices is that of
\[
\Matrix{1 & 0 \\ 0 & 0} \qquad \Matrix{0 & 0 \\ 0 & 1} \qquad \Matrix{1 & 0 \\ 0 & 0}
\]
and $\GG$ is ODE-equivalent 
to network (5) in the figure.

\noindent
{\em Case} (c):
Either $\GG$ is homogeneous or not. If it is homogeneous then
the internal arrow-type has adjacency matrix $I$. Any other arrow-type
is either a self-loop whose adjacency matrix is a multiple of $I$, 
or its adjacency matrix has the form
\[
\Matrix{a & b \\ c & d}
\]
where $a+b = c+d$. Swapping nodes if necessary
we may assume that $a \leq d$. Subtracting $aI$, the above matrix becomes
\[
\Matrix{0 & b \\
c & b-c}
\]
and by scaling, we may assume that $\gcd(b,c) = 1$.

The largest possible span is the set of all matrices 
\[
\Matrix{a & b \\ c & d}
\]
where $a+b = c+d$, which has dimension 3. Therefore there are two
possibilities: the span is 3-dimensional or it is 2-dimensional.
If it is 3-dimensional it is spanned by
\[
\Matrix{1 & 0 \\ 0 & 1} \qquad \Matrix{0 & 1 \\ 1 & 0} \qquad \Matrix{0 & 1 \\ 0 & 1}
\]
which is case (7) in the figure.
If it is 2-dimensional, then every adjacency matrix
is either a multiple of the identity or a multiple of some
\[
\Matrix{0 & b \\
c & b-c}
\]
with $\gcd(b,c) = 1$. Let $b=p+q, c=q$, so $b-c = p$. 

If $p=0$ we can scale $q$ to equal $1$, 
giving case (4) in the figure. If $q=0$ we can scale $p$ to $1$,
giving case (6) in the figure (again).
Otherwise $p, q >0$ and $\gcd(p,q)=1$.
Two such matrices are linearly dependent modulo the identity
if and only if they are the same. So now we get a 2-parameter family
of networks, with parameters $p, q$ such that $\gcd(p,q) = 1$,
which is case (8) in the figure.

Finally, suppose $\GG$ is not homogeneous. The nodes are not
input equivalent, so by irredundancy they have different types.
There must be at least one arrow in each direction,
and these have different arrow-types. 
The node-types give adjacency matrices
\[
\Matrix{1 & 0 \\ 0 & 0} \quad \Matrix{0 & 0 \\ 0 & 1}
\]
Any self-loops give adjacency matrices that are multiples of these,
and can be deleted.
Arrows from node 1 to node 2 and the reverse give adjacency matrices of the form
\[
\Matrix{0 & 0 \\ a & 0} \quad \Matrix{0 & b \\ 0 & 0}
\]
respectively, where $a,b \neq 0$ for at least one arrow-type each way.
Thus the adjacency matrices span the 4-dimensional space of all $2\times 2$ matrices.
The simplest such network is number (3), and all others are ODE-equivalent to it
since their adjacency matrices span the same 4-dimensional space.
\end{proof}

\subsection{Rigidity Conjectures for 2-Colourings}
\label{S:RC2C}

Some (non-standard) terminology is useful:

\begin{definition}\em
\label{D:LKS}
A periodic orbit $\Y$ of a dynamical system $\dot x = f(x)$ is
{\em locally Kupka-Smale} if there exist $\delta_1>\delta_2>0$ and tubular neighbourhoods
$\Y_{\delta_1} \supseteq \Y_{\delta_2} \supseteq \Y$ such that

(a) The flow of $f$ maps $\Y_{\delta_2}$ into  $\Y_{\delta_1}$.

(b) There exists an arbitrarily small perturbation $p$ of $f$
$\delta_3, \delta_4 > 0$ with $\delta_3 \leq \delta_1$ and
$\delta_4 \leq \delta_2$
such that the flow of $f+p$ maps $\Y_{\delta_4}$ into  $\Y_{\delta_3}$.

(c) Every periodic orbit for $f+p$ that is contained in $\Y_{\delta_4}$
is hyperbolic.
\qed\end{definition}

Lemma~\ref{L:LKS} states that every periodic orbit $\X$ for
a general dynamical system is locally Kupka-Smale. Below we use
this to deduce that $\X^\RR$ is locally Kupka-Smale for certain
networks $\GG$.
Before embarking on the proof of Theorem~\ref{T:2colRC} (b), we explain why it
reduces to a case-by-case analysis showing that all
networks (3)--(7) in Figure~\ref{F:2nodeODElist} are Kupka-Smale
or locally Kupka-Smale.

If $\bowtie$ contains just two colour classes, all induced ODEs
are defined by sets of representatives $\RR$ of cardinality $2$.
The corresponding quasi-quotients are $2$-node networks.
The Rigidity Conjectures then
follow from the previous analysis, provided we can prove that
every $2$-node quasi-quotient network $\GG^\RR$ occurring in the proof
has the stable isolation property. In particular this follows if $\GG^\RR$ 
is Kupka-Smale or locally Kupka-Smale.
By Remark~\ref{R:ODEpres} the stable isolation property and the (local and global)
Kupka-Smale properties
are preserved by ODE-equivalence. It therefore suffices to prove that all
networks in Figure~\ref{F:2nodeODElist} except the disconnected
networks (1) and (2) are Kupka-Smale or locally Kupka-Smale.
(By Section~\ref{S:FH}, networks (1) and (2) lack the stable isolation property,
but we can deal with them by a different, trivial, argument.)

\vspace{.1in}
\noindent{\bf Proof of Theorem~\ref{T:2colRC} (b)}

{\em Case} (1):

This network is disconnected. It arises only if
the original network $\GG$ has no arrows whose head and tail have different colours.
This implies that $\GG$ is the disjoint union of two networks, each
having all nodes of the same colour. Such a pattern is automatically balanced.

{\em Case} (2):

This network is also disconnected, and the same argument applies.

{\em Case} (3):

Admissible ODEs for this network are arbitrary dynamical systems
on $P_1 \times P_2$. The standard Kupka-Smale Theorem~\cite{K63,P66,S63}
therefore applies.

{\em Case} (4):

Admissible ODEs for this network are arbitrary $\Z_2$-equivariant dynamical systems
on $P_1 \times P_1$. The equivariant Kupka-Smale Theorem of Field~\cite{F80}
therefore applies.

{\em Case} (5): This is a general `forced' dynamical system on $P_1 \times P_2$.
We prove it is locally Kupka-Smale, which
implies the stable isolation property.
Admissible ODEs have the `forced' or `skew product' form
\begin{equation}
\label{E:FFadODE}
\begin{array}{rcl}
\dot x_1 &=& f(x_1) \\
\dot x_2 &=& g(x_2,x_1)
\end{array}
\end{equation}
If $x_1$ is steady, with equilibrium $x_1 = \alpha$, we can perturb
$f$ to make $\alpha$ a hyperbolic fixed point. Then the second component is
$\dot x_2 = g(x_2,\alpha)$, which is an arbitrary ODE in $x_2$. By the
Kupka-Smale Theorem we can perturb $g$ to make $x_2$ hyperbolic. Now
$\X$ is hyperbolic.

\begin{figure}[htb]
\centerline{
\includegraphics[width=.4\textwidth]{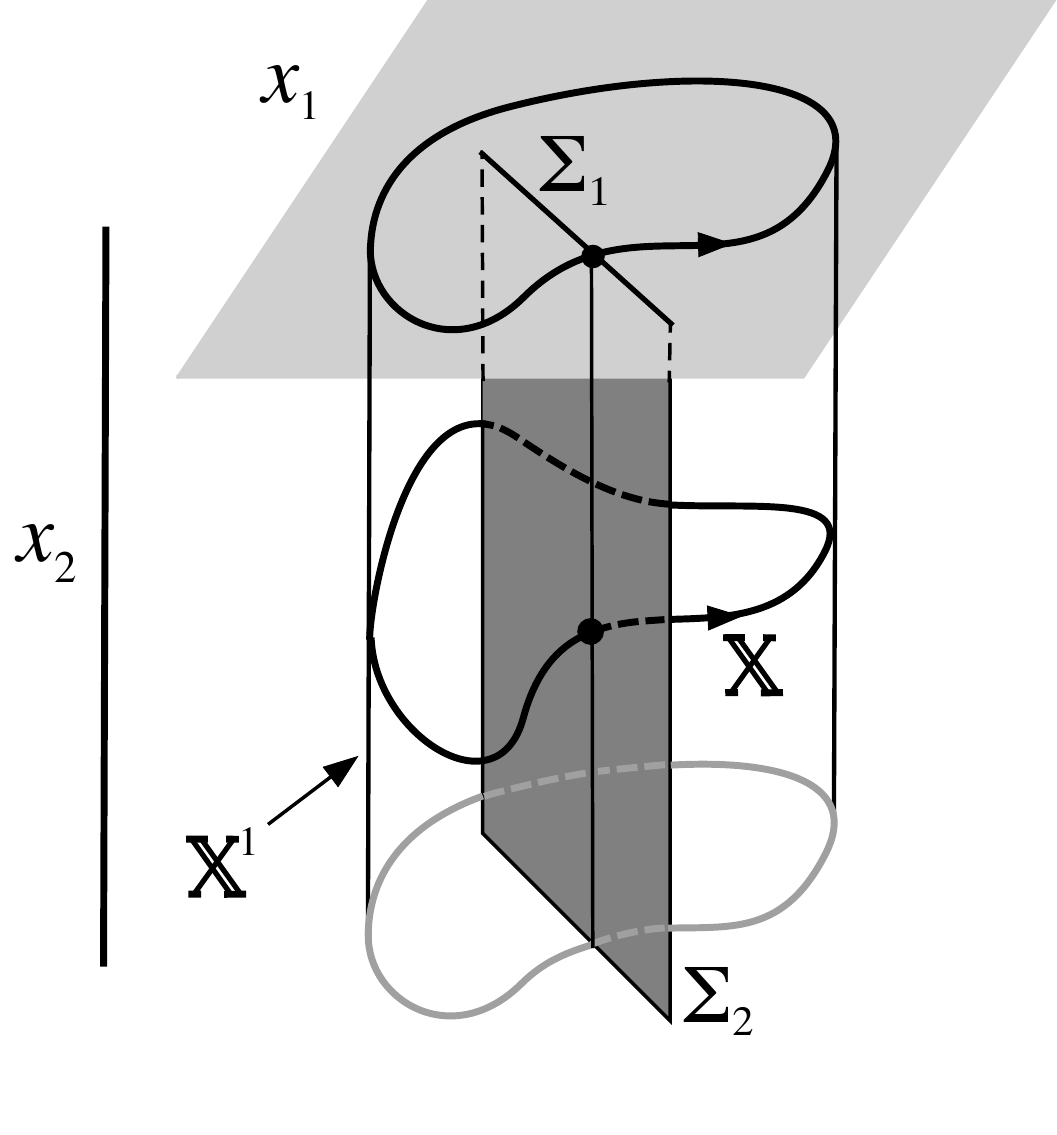}}
\caption{Poincar\'e Section for~\eqref{E:FFadODE}.}
\label{F:ffppo}
\end{figure}

If $x_1$ oscillates, let the minimal period of $\X$ be $T$.
Let $\Sigma_1 \subseteq P_1$ be a Poincar\'e section 
transverse to $\{x_1(t)\}$ at $t=0$, and define
$\Sigma_2 = \{x_1(t_0)\} \times \Sigma_1$ as in Figure~\ref{F:ffppo}. The
Poinca\'re map $\pi: \Sigma_2 \to \Sigma_2$ has the form
\[
\pi(u,v) = (\pi_1(u),\pi_2(u,v))
\]
since the ODE is feedforward. The derivative is
\[
\mathrm{D}\pi = \Matrix{\mathrm{D}_1\pi_1 & 0 \\ \mathrm{D}_1\pi_2 &\mathrm{D}_2\pi_2}
\] 
and we want to make $\mathrm{D}\pi|_{\X^\ast}$ hyperbolic, where
$\X^\ast = (x_1(t_0),x_2(t_0))$ is the fixed point of $\pi$.

Perturbing $f$ and applying the Kupka-Smale Theorem we can make
$\mathrm{D}_1\pi_1$ hyperbolic on $P_1$. 
The flow of \eqref{E:FFadODE} leaves
the cylinder $\X^1 \times P_2 \cong \Sone\times P_2$ invariant.
The component $\pi_2(x_1(t_0), x_2)$ is a local diffeomorphism near $(x_1(t_0), x_2(t_0) \in \X$.
A suitable perturbation of $g$ perturbs $\pi_2$ to any nearby diffeomorphism
(consider a suspension). By the
Kupka-Smale Theorem for diffeomorphisms (discrete dynamics) we can
make $\mathrm{D}_2\pi_2$ hyperbolic at $\X^\ast$.
Therefore $\X$ is hyperbolic.

{\em Case} (6): This is similar, but the network is
homogeneous, so minor modifications are required. Admissible ODEs have the form
\beqn
\dot x_1 &=& f(x_1,x_1) \\
\dot x_2 &=& f(x_2,x_1)
\eeqn
The subspace $\Delta = \{(u,u)\}$ is flow-invariant. If $\X$ meets
$\Delta$ then it remains in $\Delta$ for all time, the ODE reduces
to $\dot u = f(u,u)$, and $\X$ can be made hyperbolic by the
Kupka-Smale Theorem applied to $f$.
If $\X$ does not meet $\Delta$ then on a small tubular neighbourhood of $\X$,
 admissibility imposes no
constraints: $f(x_1,x_1)$ and $f(x_1,x_2)$ are independent functions.
We can now argue as in case (5), perturbing near $\X$.

{\em Case} (7): 
 Admissible ODEs have the form
\beqn
\dot x_1 &=& f(x_1,x_2,x_2) \\
\dot x_2 &=& f(x_2,x_1,x_2)
\eeqn
and $\Delta = \{(u,u)\}$ is flow-invariant. 
By similar reasoning, either $\X \subseteq \Delta$ and the result follows,
or we can reduce this case to a 
general dynamical system in some tubular neighbourhood of $\X$.
To do so we must show that $\X^1 \cap \X^2 = \emptyset$,
where
\beqn
\X^1 &=& \{(x_1(t),x_2(t),x_2(t)): t \in \R\} \\
\X^2 &=& \{(x_2(t),x_1(t),x_2(t)): t \in \R\} 
\eeqn
Suppose that $\X^1 \cap \X^2 \neq \emptyset$. Then there exist times $s,t$ such that
\[
(x_1(t),x_2(t),x_2(t)) = (x_2(s),x_1(s),x_2(s))
\]
Therefore
\[
x_1(t) = x_2(s) = x_2(t)
\]
so $\X \cap \Delta \neq \emptyset$, and $\X \subseteq \Delta$.
Now tubular neighbourhoods of $\X^1$ and $\X^2$ are disjoint, so
we can choose $f$ independently on these neighbourhoods
without destroying admissibility. The analysis then reduces to case (3).

{\em Case} (8): This is similar to case (7). 
Admissible ODEs have the form
\beqn
\dot x_1 &=& f(x_1,\underbrace{x_2,\ldots x_2}_{p+q}) \\
\dot x_2 &=& f(x_2,\underbrace{x_1,\ldots,x_1}_p,\underbrace{x_2,\ldots x_2}_q)
\eeqn
Define 
\beqn
\X^1 &=& \{(x_1(t),\underbrace{x_2(t),\ldots x_2(t)}_{p+q}): t \in \R\} \\
\X^2 &=& \{(x_2,\underbrace{x_1(t),\ldots,x_1(t)}_p,\underbrace{x_2(t),\ldots x_2(t)}_q): t \in \R\} 
\eeqn
Again we can prove that $\X^1 \cap \X^2 = \emptyset$ by comparing suitable entries
of $\X^1(t)$ and $\X^2(s)$ and deducing that $x_1(t) =x_2(s) = x_2(t)$.
 The rest is as before.
\qed

\subsection{Final Remarks}
\label{S:FinRem}

The proof in this paper of the Rigidity Conjectures for strongly hyperbolic
periodic orbits uses three unorthodox methods, which together reveal
a link with network analogues of the Kupka-Smale Theorem.
These methods are:

\begin{itemize}
\item[\rm (a)]
The use of overdetermined ODEs in which some components 
are formally inconsistent with others.
\item[\rm (b)]
A construction analogous to the usual quotient network by a balanced 
colouring, applied to a colouring that is not balanced.
\item[\rm (c)]
Construction of an admissible perturbation $p$ that leaves the perturbed
periodic orbit $\tilde\X$ unchanged. This avoids the main obstacle to
proving the Rigidity Conjectures: keeping track of $\tilde\X$. Here this
is not a problem because $\tilde\X = \X$. 
\end{itemize}

This combination works because the rigidity assumption causes enough structure to
be preserved for the formal inconsistency (a) to contradict rigidity
of the local synchrony pattern for $\X$. In the current state of
knowledge, this contradiction relies on strong 
hyperbolicity, which is closely related to the Kupka-Smale Theorem and
possible network analogues.

A more specific version of strong hyperbolicity is logically equivalent to the
Rigid Synchrony Property (local or global). Definition~\ref{D:strongly_hyperbolic}
is stated for all colourings, but we use it only for colourings determined
by a local rigid synchrony pattern. We can therefore weaken the definition
of strong hyperbolicity by considering only these colourings.
This weaker version still implies the
Rigid Synchrony Property, with the same proof. Conversely, the Rigid Synchrony Property
implies that any local rigid synchrony pattern is balanced, so its
synchrony space is flow-invariant. This implies that the Floquet multipliers
of the induced periodic orbit $\X^\RR$ are a subset of those of $\X$.
Since $\X$ is assumed hyperbolic, so is $\X^\RR$.

We have shown that if a counterexample to
the Rigidity Conjectures exists, its dynamics must be remarkably degenerate
from the viewpoint of general dynamical systems theory. Specifically, for some unbalanced
colouring and all small admissible perturbations, and for any 
set of representatives $\RR$ of that colouring,
the solution $\X^\RR$ of the induced ODE is non-isolated
in an extreme manner: it is included in a continuum of distinct periodic orbits. 
It is difficult to see how the constraints imposed by network topology could create
such degeneracy rigidly. Be that as it may, the focus for proving
the Rigidity Conjectures now shifts towards network analogues
of the Kupka-Smale Theorem --- at least until some alternative method is found.


\begin{thebibliography}{99}

\footnotesize

\bibitem{AM78}
R.~Abraham and J.E.~Marsden. {\em Foundations of Mechanics},
Benjamin/Cummings, New York 1978.

\bibitem{AMR83}
R. Abraham, J.E. Marsden, and T. Ratiu.
{\em Manifolds, Tensor Analysis, and Applications}, Addison-Welsey, Reading MA 1983.

\bibitem{A10}
J.W. Aldis.
{\em On Balance}, PhD Thesis, University of Warwick 2010. 

\bibitem{AS06}
F. Antoneli and I. Stewart. 
Symmetry and synchrony in coupled cell networks 1: fixed-point spaces, 
{\em Internat. J. Bif. Chaos} {\bf 16} (2006) 559--577.
 
\bibitem{AS07}
F. Antoneli and I. Stewart. 
Symmetry and synchrony in coupled cell networks 2: group networks, 
{\em Internat. J. Bif. Chaos} {\bf 17} (2007) 935--951.
 
\bibitem{AS08}
F. Antoneli and I. Stewart. 
Symmetry and synchrony in coupled cell networks 3: exotic patterns, 
{\em Internat. J. Bif. Chaos} {\bf 18} (2008) 363--373.

\bibitem{A89}
V.I. Arnold. {\em Mathematical Methods of Classical Mechanics}, Springer, Berlin 1989.

\bibitem{AP90}
D.K. Arrowsmith and C.M. Place.
{\em An Introduction to Dynamical Systems}, Cambridge University Press,
Cambridge 1990.

\bibitem{BBH00}
V. Belykh, I. Belykh, and M. Hasler.
Hierarchy and stability of partially synchronous oscillations of diffusively 
coupled dynamical systems, {\em Phys. Rev. E} {\bf 62} (2000) 6332--6345. 

\bibitem{BBNH03}
I. Belykh, V. Belykh, K. Nevidin, and M. Hasler.
Persistent clusters in lattices of coupled nonidentical chaotic systems, 
{\em Chaos} {\bf 13} (2003) 165; doi: 10.1063/1.1514202.

\bibitem{BH11}
I. Belykh and M. Hasler.
Mesoscale and clusters of synchrony in networks of bursting neurons, 
{\em Chaos} {\bf 21} (2011) 016106.

\bibitem{BPP01}
S.~Boccaletti, L.M.~Pecora, and A.~Pelaez.
A unifying framework for synchronization of coupled dynamical
systems, {\em Phys.~Rev E} {\bf 63} (2001) 066219.

\bibitem{B87}
R.~Brown.  From groups to groupoids: a brief survey,
{\em Bull. London Math. Soc.} {\bf 19} (1987) 113--134.

\bibitem{BG01}
P.-L. Buono and M. Golubitsky. Models of central pattern generators for 
quadruped locomotion: I. primary gaits, {\em J. Math. Biol.}  {\bf 42} (2001) 291--326.

\bibitem{BRS99a}
R.J. Butera Jr., J. Rinzel, and J.C. Smith.
Models of respiratory rhythm generation in the pre-B\"otzinger complex I:
Bursting pacemaker neurons, {\em J. Neurophysiol.} {\bf 82} (1999) 382--397.

\bibitem{BRS99b}
R.J. Butera Jr., J. Rinzel, and J.C. Smith.
Models of respiratory rhythm generation in the pre-B\"otzinger complex II:
Populations of coupled pacemaker neurons, {\em J. Neurophysiol.} {\bf 82} (1999) 398--415.

\bibitem{CMS10}
R. Campos, V. Matos, and C. Santos.
Hexapod locomotion: a nonlinear dynamical systems approach,
{\em IECON 2010 - 36th Annual Conference on IEEE Industrial Electronics Society}
(2010) 1546--1551;
doi 10.1109/IECON.2010.5675454.

\bibitem{CTB13}
J.D. Chambers, E.A. Thomas, and
C. Bornstein. Mathematical modelling of enteric neural motor patterns,
{\em Proc. Austral. Physiol. Soc.} {\bf 44} (2013) 75--84.

\bibitem{C10}
R. Curtu. Singular Hopf bifurcations and mixed-mode oscillations in a 
two-cell inhibitory neural network, {\em Phys.~D} {\bf 239} (2010) 504--514.

\bibitem{DP90}
B.A.~Davey and H.A.~Priestley.
{\em Introduction to Lattices and Order},
Cambridge University Press, Cambridge 1990.

\bibitem{DS04}
A.P.S.~Dias and I.~Stewart.  
Symmetry groupoids and admissible vector fields for coupled cell networks, 
{\em J. London Math. Soc.} {\bf 69} (2004) 707--736.

\bibitem{DS05}
A.P.S.~Dias and I.~Stewart.  
Linear equivalence and ODE-equivalence for coupled cell networks, 
{\em Nonlinearity} {\bf 18} (2005) 1003--1020.

\bibitem{DG14} 
C. Diekman and M. Golubitsky.
Network symmetry and binocular rivalry experiments,
{\em J. Math. Neuro.} {\bf 4} (2014); doi 10.1186/2190-8567-4-12.

\bibitem{DGMW12} 
C. Diekman, M. Golubitsky, T. McMillen, and Y. Wang. Reduction and
 dynamics of a generalized rivalry network with two learned patterns,
 {\em SIAM J.~Appl. Dynam. Sys.} {\bf 11} (2012) 1270--1309. 
 
\bibitem{DGW13} 
C. Diekman, M. Golubitsky, and Y. Wang. 
Derived patterns in binocular rivalry networks,
{\em J. Math. Neuro.} {\bf 3} (2013); doi 10.1186/2190-8567-3-6.

\bibitem{F80} M.~Field. Equivariant dynamical systems,
{\em Trans. Amer. Math. Soc.} {\bf 259} (1980) 185--205.

\bibitem{F04} M.~Field. Combinatorial dynamics,
{\em Dynamical Systems} {\bf 19} (2004) 217--243.

\bibitem{GGPSW19}
P. Gandhi, M. Golubitsky,  C. Postlethwaite, I. Stewart, and Y. Wang.
Bifurcations on fully inhomogeneous networks,  {\em SIAM J. Appl. Dyn. Sys.}
{\bf 19} (2020) 36--41; supplementary material SM1--SM31.

\bibitem{GBEE13}
J. Gjorgjieva, J. Berni, J.F. Evers, and S.J. Egle.
Neural circuits for peristaltic wave propagation in crawling {\em Drosophila} larvae: analysis and modeling,
{\em Front. Comput. Neurosci.} {\bf 7} (2013); doi 10.3389/fncom.2013.00024. 

\bibitem{G03}
A.F. Glova.
Phase locking of optically coupled lasers,
{\em Quantum Electronics} {\bf 33} (2003)
283--306; doi 10.1070/QE2003v033n04ABEH002415.

\bibitem{GMS16}
M. Golubitsky, L. Matamba Messi, and L. Spardy.  
Symmetry types and phase-shift synchrony in networks, {\em Physica D} {\bf 320} (2016) 9--18.

\bibitem{GNS04}
M.~Golubitsky, M.~Nicol, and I.~Stewart. Some curious phenomena in coupled cell networks,
{\em J. Nonlinear Sci.} {\bf 14} (2004) 207--236.

\bibitem{GRW10}
M. Golubitsky, D. Romano and Y. Wang. 
Network periodic solutions: full oscillation and rigid synchrony,
{\em Nonlinearity} {\bf  23} (2010) 3227--3243.

\bibitem{GRW12}
M. Golubitsky, D. Romano and Y. Wang. 
Network periodic solutions: patterns of phase-shift synchrony,
{\em Nonlinearity} {\bf 25} (2012) 1045--1074.

\bibitem{GS85}
M. Golubitsky and D.G. Schaeffer. {\em Singularities and Groups
in Bifurcation Theory I}, Applied Mathematics Series {\bf 51},
Springer, New York 1985.  

\bibitem{GS02}
M. Golubitsky and I. Stewart. 
{\em The Symmetry Perspective: from equilibria to chaos in 
phase space and physical space}, Progress in Mathematics 
{\bf 200}, Birkh\"auser, Basel 2002.

\bibitem{GS06}
M. Golubitsky and I. Stewart.
Nonlinear dynamics of networks: the groupid formalism, {\em Bull. Amer. Math. Soc.}
{\bf 43} (2006) 305--364.

\bibitem{GS17}
M. Golubitsky and I. Stewart.
Coordinate changes that preserve admissible maps for network dynamics, 
{\em Dynamical Systems} {\bf 32} (2017) 81--116. [{\em Equivariance and Beyond: M. Golubitsky's 70th Birthday.}] doi 10.1080/14689367.2016.1235136.

\bibitem{GS22}
M. Golubitsky and I. Stewart.
{\em Dynamics and Bifurcation in Networks}, SIAM, to appear 2022.
 
\bibitem{GSBC98}
M. Golubitsky, I. Stewart, P.-L. Buono, and J.J. Collins.
A modular network for legged locomotion, {\em Physica} D {\bf 115} (1998) 56--72.

\bibitem{GSCB99}
M. Golubitsky, I. Stewart, J.J. Collins, and P.-L. Buono.
Symmetry in locomotor central pattern generators and animal gaits, 
{\em Nature} {\bf 401} (1999) 693--695.

\bibitem{GSS88}
M.~Golubitsky, I.~Stewart, and D.G.~Schaeffer. {\em Singularities
and Groups in Bifurcation Theory II}, Applied Mathematics Series
{\bf 69}, Springer, New York 1988.

\bibitem{GST05}
M. Golubitsky, I. Stewart, and A. T\"or\"ok.
Patterns of synchrony in coupled cell networks with multiple arrows,
{\em SIAM J. Appl. Dynam. Sys.} {\bf 4} (2005) 78--100.

\bibitem{GH83}
J. Guckenheimer and P. Holmes.
{\em Nonlinear Oscillations, Dynamical Systems, and Bifurcations of Vector Fields},
Springer, New York 1983.

\bibitem{HKW81}
B.D. Hassard, N.D. Kazarinoff, and Y.-H. Wan.
{\em Theory and Applications of Hopf Bifurcation},
London Math. Soc. Lecture Notes {\bf 41}, 
Cambridge University Press, Cambridge 1981.

\bibitem{H71}
P.J.~Higgins.
{\em Notes on Categories and Groupoids}, Van Nostrand Reinhold
Mathematical Studies {\bf 32}, Van Nostrand Reinhold,
London 1971.

\bibitem{HPS77}
M.W.~Hirsch, C.C.~Pugh, and M.~Shub.
{\em Invariant Manifolds}, Lect. Notes Math. {\bf 583},
Springer, New York 1977.

\bibitem{HS74}
M. W. Hirsch and S. Smale.
{\em Differential Equations, Dynamical Systems, and Linear Algebra},
Academic Press, New York 1974.


\bibitem{I93}
K. It\^o (ed.). 
{\em Encyclopaedic Dictionary of Mathematics} (2nd ed.) vol. 1,
MIT Press, Cambridge MA 1993.

\bibitem{J12}
R. Joly.
Observation and inverse problems in coupled cell networks, {\em Nonlinearity} {\bf 25} (2012) 657--676.

\bibitem{JT06}
K. Josi\v{c} and A. T\"or\"ok.
Network architecture and spatio-temporally symmetric dynamics,
{\em Physica} D {\bf 224} (2006) 52--68.

\bibitem{KH95}
A. Katok and B. Hasselblatt.
{\em Introduction to the Modern Theory of Dynamical Systems},
Cambridge University Press, Cambridge 1995.

\bibitem{KL94}
N. Kopell and G. LeMasson.
Rhythmogenesis, amplitude modulation, and multiplexing
in a cortical architecture, {\em Proc. Natl. Acad. Sci. USA} (1994)
{\bf 91} 10586--10590.

\bibitem{K63}
I. Kupka. Contribution \`a la th\'eorie des champs g\'en\'eriques,
{\em Contrib. Diff. Eqs.} {\bf 2} (1963) 457--484; 
{\em Contrib. Diff. Eqs.} {\bf 3} (1964) 411--420.

  
\bibitem{LT85}
P.~Lancaster and M.~Tismenetsky.
{\em The Theory of Matrices}, Academic Press, Orlando 1985.

  \bibitem{LG06}
M.C.A. Leite and M. Golubitsky. Homogeneous three-cell networks,
{\em Nonlinearity} {\bf 19} (2006) 2313--2363.

\bibitem{LCZ09}
C. Liu, Q. Chen, and J. Zhang.
Coupled van der Pol oscillators utilised as central pattern generators for
quadruped locomotion, {\em IEEE
Chinese Control and Decision Conference} 3677--3682 (2009);
doi 10.1109/CCDC.2009.5192385.

\bibitem{MMZ04}
S.C.~Manrubia, A.S.~Mikhailov, and D.H.~Zanette.
{\em Emergence of Dynamical Order}, World Scientific, Singapore 2004.

\bibitem{MMP02}
  E.~Mosekilde, Y.~Maistrenko, and D.~Postonov,
  {\em Chaotic Synchronization}, World Scientific, Singapore 2002.
  
  \bibitem{PSHMR13}
L.M. Pecora, F. Sorrentino, A.M. Hagerstrom, T.E. Murphy, and R. Roy.
Symmetries, cluster syn\-chron\-iz\-ation, and isolated
desynchronization in complex networks, {\em Nature Communications} {\bf 5}
(2013); doi: 10.1038/ncomms5079.

\bibitem{P66}
M.M. Peixoto.
On an approximation theorem of Kupka and Smale,
{\em J. Diff. Eq.} {\bf 3} (1966) 214--227.

\bibitem{P08}
A.Y. Pogromsky.
A partial synchronization theorem, {\em Chaos} {\bf 18} 037107 (2008). 

\bibitem{PSN02}
A.Y. Pogromsky, G. Santoboni, and H. Nijmeijer.
Partial synchronization: from symmetry towards stability, {\em Physica D} {\bf 172} (2002) 65--87. 

\bibitem{S99}
  W.~Singer. Neuronal synchrony: a versatile code for the definition of
  relations, {\em Neuron} {\bf 24} (1999) 49--65.
  
\bibitem{S63}
Smale, S.
Stable manifolds for differential equations and diffeomorphisms,
{\em Ann. Scuola Normale Superiore Pisa} {\bf 18} (1963) 717--86

\bibitem{S67}
Smale, S.
Differentiable dynamical systems,
{\em Bull. Amer. Math. Soc.} {\bf 73} (1967) 747--817. 

\bibitem{S07}
I.~Stewart.
The lattice of balanced equivalence relations of a coupled cell network, 
{\em Math. Proc. Camb. Phil. Soc.} {\bf 143} (2007) 165--183.

\bibitem{S20}
I.~Stewart.
Overdetermined constraints and rigid synchrony patterns for network equilibria,
{\em Portugaliae Mathematica} {\bf 77} (2020) 163--196.

\bibitem{SG19}
I. Stewart and M. Golubitsky.
Symmetric networks with geometric constraints as models of 
visual illusions, {\em Symmetry} {\bf 11} (2019) 799; doi: 10.3390/sym11060799.

\bibitem{SGP03}
I. Stewart, M. Golubitsky, and M. Pivato. Symmetry groupoids 
and patterns of synchrony in coupled cell networks, 
{\em SIAM J. Appl. Dynam. Sys.} {\bf 2} (2003) 609--646. 

\bibitem{SP07}
I. Stewart and M. Parker. 
Periodic dynamics of coupled cell networks I:  
rigid patterns of synchrony and phase relations, 
{\em Dynamical Systems} {\bf 22} (2007) 389-450.

\bibitem{SP08}
I. Stewart and M. Parker. 
Periodic dynamics of coupled cell networks II: cyclic symmetry, {\em Dynamical Systems}
{\bf 23} (2008) 17--41.

  \bibitem{UPLMNNS09}
P.J. Uhlhaas, G. Pipa, B. Lima, L. Melloni, 
S. Neuenschwander, D. Nikoli\'c, and W. Singer. 
Neural synchrony in cortical networks: history, concept and current status,
{\em Front. Integr. Neurosci.} {\bf 30} (2009); doi: 10.3389/neuro.07.017.2009.

\bibitem{VH01}
C. van Vreeswijk and D. Hansel.
Patterns of synchrony in neural networks with spike adaptation,
{\em Neural Computation} {\bf 13} (2001) 959--992.

\bibitem{W02}
X.-F. Wang.  Complex networks: Topology, dynamics and synchronization,
{\em Internat. J. Bif. Chaos} {\bf 12} (2002) 885--916.

\bibitem{wiki_strongly} Wikipedia. Strongly connected component,
en.wikipedia.org/wiki/Strongly\_connected\_component.

\bibitem{ZPYLZX19}
L. Zhang, W.B. Pan,  L. Yan, B. Luo, X. Zou, and M. Xu.
Cluster synchronization of coupled semiconductor lasers network with complex topology,
{\em IEEE J. Selected Topics in Quantum Electronics} {\bf 25} (2019)
1501007;  doi 10.1109/JSTQE.2019.2913010.
    

\end{thebibliography}
\end{document}